\documentclass[reqno,10pt,a4paper]{amsart}
\usepackage{amsmath}
\usepackage{amssymb}
\usepackage{latexsym}
\usepackage{epsf}
\usepackage{graphics}
\usepackage{tikz}
\usetikzlibrary{decorations.pathreplacing}

\usepackage[dvips]{rotating}

\usepackage[all]{xy}
\usepackage{multicol}

\UseComputerModernTips
\CompileMatrices

\DeclareMathOperator{\Hom}{Hom}
\DeclareMathOperator{\Ext}{Ext}

\renewcommand{\ge}{\geqslant}
\renewcommand{\le}{\leqslant}

\newcommand{\C}{\mathbb{C}}

\newcommand{\N}{\mathbb{N}}
\newcommand{\R}{\mathbb{R}}

\newcommand{\Z}{\mathbb{Z}}

\DeclareMathOperator{\sgn}{sgn}

\newcommand{\half}{\textstyle\frac{1}{2}}

\newcommand{\ee}{\varepsilon}

\newcommand{\Sb}{{S}} 
 
\newcommand{\dd}{\delta}
\newcommand{\dl}{\delta_L}
\newcommand{\dr}{\delta_R}
\newcommand{\kl}{\kappa_L}
\newcommand{\kr}{\kappa_R}
\newcommand{\kk}{\kappa_{LR}}
\newcommand{\twoBTL}{\mathrm{2BTL}}

\DeclareMathOperator{\uri}{{ur}_1}
\DeclareMathOperator{\uro}{{ur}_0}

\newcommand{\W}[2]{W^{#1}_{#2}}  

\newcommand\WW{\W{n,m}{\ee_1, \ee_2}}
\newcommand\WWB{B^{n,m}_{\ee_1, \ee_2}}  
\newcommand\Wb{\W{n}{}(b)}
\newcommand{\WWGam}{\Gamma^{n,m}_{\ee_1,\ee_2}} 
\newcommand{\Gramdet}[2]{\Gamma^{#1}_{#2}}  
\newcommand{\Gramcon}{G^{n,m}_{\ee_1,\ee_2}} 
\newcommand{\Gram}[2]{G^{#1}_{#2}}  

\newcommand{\HCn}{H(\tilde{C}_{n})}  
\newcommand{\bnp}{b'_n }  
\newcommand{\beq}{\begin{equation}}
\newcommand{\eq}{\end{equation}}

\newcommand{\bluex}[1]{}

\newcommand{\redx}[1]{}
\newcommand{\bnx}{{b_n^x}}

\newcommand{\de}{d} 
\newcommand{\we}[2]{\frac{#1}{#2}} 
\newcommand{\thetam}{-m + \ee_1 w_1 +\ee_2 w_2} 

\newcommand{\mat}[1]{\left( \begin{array}{#1}}
\newcommand{\tam}{ \end{array} \right)}

\newcommand{\dmpone}{  
\tikz[scale=0.5,baseline={(0,0.4)}]{
      \draw (1,2) to[out=-90,in=-90] node[pos=0.5]{$\bullet$} +(1,0);
      \draw (3,2) to[out=-90,in=-90] node[pos=0.5]{$\bullet$} +(1,0);
      \draw (1,0) to[out=90,in=90] node[pos=0.5]{$\bullet$} +(1,0);
      \draw (3,0) to[out=90,in=90] node[pos=0.5]{$\bullet$} +(1,0);
      \draw (2.5,1) node {$\cdots$};
      \draw (4.8,0) to (4.8,2);
      \draw (5.5,1) node {$\cdots$};
      \draw (6,0) to (6,2);
      \draw (0.5,2) rectangle (6.5,0);
      \draw
          [decorate,decoration={brace,amplitude=4pt,mirror},yshift=-4pt]
          (4.8,0) -- (6,0) node [midway,yshift=-10pt] {${}^{m+1}$};}
}

\newcommand{\dmmone}{
\tikz[scale=0.5,baseline={(0,0.4)}]{
      \draw (1,2) to[out=-90,in=-90] node[pos=0.5]{$\bullet$} +(1,0);
      \draw (3,2) to[out=-90,in=-90] node[pos=0.5]{$\bullet$} +(1,0);
      \draw (1,0) to[out=90,in=90] node[pos=0.5]{$\bullet$} +(1,0);
      \draw (3,0) to[out=90,in=90] node[pos=0.5]{$\bullet$} +(1,0);
      \draw (2.5,1) node {$\cdots$};
      \draw (4.8,0) to node[pos=0.5]{$\bullet$} (4.8,2);
      \draw (5.3,0) to (5.3,2);
      \draw (6,1) node {$\cdots$};
      \draw (6.5,0) to (6.5,2);
      \draw (0.5,2) rectangle (7,0);
      \draw
          [decorate,decoration={brace,amplitude=4pt,mirror},yshift=-4pt]
          (5.3,0) -- (6.5,0) node [midway,yshift=-10pt] {${}^{m}$};}
}

\newcommand{\dppone}{
\tikz[scale=0.5,baseline={(0,0.4)}]{
      \draw (1,2) to[out=-90,in=-90] node[pos=0.5]{$\bullet$} +(1,0);
      \draw (3,2) to[out=-90,in=-90] node[pos=0.5]{$\bullet$} +(1,0);
      \draw (1,0) to[out=90,in=90] node[pos=0.5]{$\bullet$} +(1,0);
      \draw (3,0) to[out=90,in=90] node[pos=0.5]{$\bullet$} +(1,0);
      \draw (2.5,1) node {$\cdots$};
      \draw (4.8,0) to (4.8,2);
      \draw (5.5,1) node {$\cdots$};
      \draw (6,0) to (6,2);
      \draw (6.5,0) to node[pos=0.5,white]{$\bullet$} node[pos=0.5]{$\circ$} (6.5,2);
      \draw (0.5,2) rectangle (7,0);
      \draw
          [decorate,decoration={brace,amplitude=4pt,mirror},yshift=-4pt]
          (4.8,0) -- (6,0) node [midway,yshift=-10pt] {${}^{m}$};}
}

\newcommand{\dpmmone}{
\tikz[scale=0.5,baseline={(0,0.4)}]{
      \draw (1,2) to[out=-90,in=-90] node[pos=0.5]{$\bullet$} +(1,0);
      \draw (3,2) to[out=-90,in=-90] node[pos=0.5]{$\bullet$} +(1,0);
      \draw (1,0) to[out=90,in=90] node[pos=0.5]{$\bullet$} +(1,0);
      \draw (3,0) to[out=90,in=90] node[pos=0.5]{$\bullet$} +(1,0);
      \draw (2.5,1) node {$\cdots$};
      \draw (4.8,0) to node[pos=0.5]{$\bullet$} (4.8,2);
      \draw (5.3,0) to (5.3,2);
      \draw (6,1) node {$\cdots$};
      \draw (6.5,0) to (6.5,2);
      \draw (7,0) to node[pos=0.5,white]{$\bullet$} node[pos=0.5]{$\circ$} (7,2);
      \draw (0.5,2) rectangle (7.5,0);
      \draw
          [decorate,decoration={brace,amplitude=4pt,mirror},yshift=-4pt]
          (5.3,0) -- (6.5,0) node [midway,yshift=-10pt]
          {${}^{m-1}$};}
}

\newcommand{\epsfboxd}[1]{\epsfbox{#1}}

\newcommand{\NN}{\N_0} 
\newcommand{\lrb}{left-right blob}
\newcommand{\Bx}{B^{x'}}  
\newcommand{\Bxx}{B^{x}}
\newcommand{\bx}{b^{x}}
\newcommand{\ddd}{\pmb{\delta}}

\newcommand{\basis}{\beta}  

\newcommand{\params}{\dd,\dl,\dr,\kl,\kr,\kk }

\newcommand{\PPi}{\Pi}   

\newcommand{\axioma}{Fix $\ddd \in k^6$ except for $\kk$,
  generic, but 
       so that $w_1,w_2$ are defined.}
\newcommand{\axiomb}{Fix ${\we{n,m}{\ee_1,\ee_2}} \in \Lambda^+_n$.}
\newcommand{\PP}{{\mathcal P}}  
\newcommand{\po}{p_0}     

\begin{document}
\theoremstyle{plain}
\newtheorem{thm}{Theorem}[section]
\newtheorem{prop}[thm]{Proposition}
\newtheorem{cor}[thm]{Corollary}
\newtheorem{clm}[thm]{Claim}
\newtheorem{lem}[thm]{Lemma}
\newtheorem{conj}[thm]{Conjecture}
\theoremstyle{definition}
\newtheorem{defn}[thm]{Definition}
\newtheorem{rem}[thm]{Remark}
\newtheorem{eg}[thm]{Example}
\newtheorem*{rem*}{Remark}

\title%
[Blocks for the symplectic blob algebra over the complex field]%
{Blocks for \\ the symplectic blob algebra over the complex field}
\author{O. H. King \and  P. P. Martin
\and A. E. Parker} 
\email{oking@dhfs.uk} 
\address{Department of Mathematics \\ University of Leeds \\ Leeds,
  LS2 9JT \\ UK}
\email{p.p.martin@leeds.ac.uk}
\email{A.E.Parker@leeds.ac.uk}


\begin{abstract}
The \emph{symplectic blob algebra} is a 
physically motivated quotient of the Hecke algebra  
$H(\tilde{C}_n)$
with a diagram calculus. 
We find the blocks for the symplectic blob algebra for all
specialisations of its parameters over the complex numbers. 
We determine Gram
determinants for the cell modules
with respect to a canonical contravariant form. 
We  show in particular that
the algebra is semisimple over the complex numbers unless at least
one of the ``quantisation'' parameters, 
or the sum or difference of two of these 
parameters is integral, 
or the bulk parameter $q$ is a root of unity. 
We 
find decomposition numbers in many of the $q$-generic cases.
\end{abstract}

\maketitle

\section*{Introduction}\label{intro}

The \emph{symplectic blob algebra}, $\bnx$, introduced in
\cite{gensymp},
is a quotient of the Hecke algebra $H(\tilde{C}_n)$ (see for instance  
\cite{humph2}, or Definition \ref{de:HeC} below).
It is of interest from a representation theory perspective 
both formally 
(i.e. in representation Theory, in the sense of \cite[\S 5.1]{schiffler}), 
and combinatorially. 
In statistical mechanics it 
controls boundary conditions in computation
for various important lattice models
(see e.g. \cite{degiernichols,gensymp}
and cf. \cite{Baxter81,marbk,martsaleur}). 
It links to established
objects of current study such as the blob algebra
\cite{Plaza,bowmancoxspeyer,LibedinskyPlaza}, 
Hecke algebras \cite{daughram,humph2},
KLR algebras \cite{kholaud,rouquier,brundankleshchev},
and Lie theory \cite{brauerstroppelehrig,jantz,marwood98}. 
Our focus in this paper is the computation
of fundamental invariants and the role of alcove geometry 
(confer
\cite{jantz,CoxDevisscherMartin0609,marwood98,blobcgm,Soergel97a,MartinWoodcock03}).

We may define $\bnx$ using a basis of diagrams which
can be thought of loosely as type-$\tilde{C}_n$ Temperley-Lieb diagrams.
These are obtained by suitably stacking
the `decorated' generators shown in Figure~\ref{fig:1}
(see \cite[\S6]{gensymp} or \S\ref{sect:notn} below for details).

\begin{figure}[ht]
\begin{multline*} 
e:=\raisebox{-0.4cm}{\epsfboxd{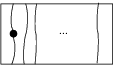}},\  
e_1:=\raisebox{-0.4cm}{\epsfboxd{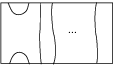}},\ 
e_2:= \raisebox{-0.4cm}{\epsfboxd{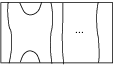}},\ 
\cdots,\ \\
e_{n-1}:=\raisebox{-0.4cm}{\epsfboxd{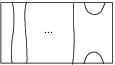}},\ 
f:=\raisebox{-0.4cm}{\epsfboxd{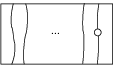}}. 
\end{multline*} 
\caption{Generating diagrams for the symplectic blob algebra. \label{fig:1}}
\end{figure}

\noindent
The algebra $\bnx$ is defined over any commutative ring $k$ 
containing  a 6-tuple
$\ddd= ( \params )$
of `straightening parameters'. 
Since we 
study representation theory, 
our aim is primarily the Artinian cases, 
and indeed the cases where $k$ is an algebraically closed field. 
Thus for each such $k$ there is an algebra $\bnx$ for each point in $k^6$.
One knows \cite{gensymp} 
(and see \S\ref{sec:gram})
that the non-semisimple cases 
lie on certain algebraic sets. 
The generic semisimple cases are well-understood \cite{gensymp}, 
so it is the points on the algebraic sets that are of
interest.

It turns out that 
the dependence of representation theory on position in the
variety is more easily described in terms of alternative
variety-specific parameterisations. 
In these the non-semisimple sub-varieties correspond to integrality of
some or all of the new parameters, as we shall see in section~\ref{sec:hom}
(see also \cite{martsaleur,degiernichols}). 
In particular we introduce the bulk parameter $q$ (where $\dd =
q+q^{-1}$) and quantisation
parameters $w_1$, $w_2$ (derived from $\dl$ and $\dr$).
\label{sec:qdef} 

In \cite{gensymp} 
various general properties of the algebra $\bnx$ are established. 
For instance 
 a cellular basis is constructed; 
its generic semisimple structure over $\C$ is determined; 
and it is shown to be  quasi-hereditary 
on an open subvariety of the  non-semisimple variety.
Full tilting modules are constructed in \cite{reeves2}.
An efficient presentation is found in \cite{mgp2I};
and in  \cite{degiernichols} a closely related algebra is studied,
leading to useful alternative bases for certain cell modules, which
are crucial to our calculation of the action of a certain special central
element.

It follows from comparison with 
the ordinary blob algebra case 
\cite{blobcgm} 
that the programme of study of the non-generic
non-semisimple representation theory of $\bnx$ is a considerably
harder challenge. 
As in  \cite{blobcgm}, however, 
a key component is to construct `enough' standard
module morphisms; and  these were  constructed  in \cite{mgp3}.
This paper, using the morphisms from \cite{mgp3}  and 
also using \cite{degiernichols},
investigates the sufficiency of this set.

Quite generally, 
if there is a non-zero homomorphism between two standard modules, then the
two modules  belong to the same block. 
Indeed,  
determination of {\em all} homomorphisms between standard modules  
in a quasi-hereditary structure 
 allows a complete description of the blocks (see the appendix). 
Our main block result in this paper, Theorem \ref{thm:blocksbnx}
is a complete description of blocks over the complex numbers.

\begin{figure}[ht]
  \centering
  \begin{tikzpicture}[scale=0.2,>=latex,baseline=0] 
    \pgfmathsetmacro{\w}{3}
    \pgfmathsetmacro{\ww}{1}
    \draw (0,-16)--(0,16);
    \draw (-16,0)--(16,0);
    \foreach \m in {0,2,4,6,8,10,12,14,16}
    \foreach \e in {-1,1}
    \foreach \f in {-1,1}
    {       
      \draw (\e * \m - \e * \e * \w - \e * \f * \ww,\f * \m - \f * \e * \w - \f * \f * \ww)  node {$\bullet$};
    }
    \fill[white] (0 -\w + \ww, 0 + \w - \ww) circle (15pt);
    \fill[white] (0 -\w - \ww, 0 + \w - \ww) circle (15pt); 
    \draw (0 - \w - \ww, 0 - \w - \ww) -- ++(16,16) node[anchor=south west] {$+,+$};
    \draw (-2 - \w + \ww, 2 + \w - \ww) -- ++(-14,14) node[anchor=south east] {$-,+$};
    \draw (2 - \w + \ww, -2 + \w - \ww) -- ++(14,-14) node[anchor=north west] {$+,-$};
    \draw (-2 - \w - \ww, -2 - \w - \ww) -- ++(-14,-14) node[anchor=north east] {$-,-$};

    \draw[dashed,->] (0,0) .. controls (-3,0) and (0,3) .. (0,0);
    \draw[dashed,->] (6 - \w - \ww, 6 - \w - \ww) -- (6 - \w - \ww,-6 + \w + \ww);
    \draw[dashed,->] (6 - \w - \ww, 6 - \w - \ww) to[out=-90,in=0] (-6 + \w + \ww, -6 + \w + \ww);
    \draw[dashed,->] (8 - \w - \ww, 8 - \w - \ww) -- (8 - \w - \ww,-8 + \w + \ww);
    \draw[dashed,->] (8 - \w - \ww, 8 - \w - \ww) -- (-8 + \w + \ww,8 - \w - \ww);
    \draw[dashed,->] (8 - \w - \ww, 8 - \w - \ww) to[out=180,in=90] (-8 + \w + \ww, -8 + \w + \ww);

    \draw[dashed] (10 - \w - \ww, 10 - \w - \ww) rectangle (-10 + \w + \ww,-10 + \w + \ww);
    \draw[dashed] (12 - \w - \ww, 12 - \w - \ww) rectangle (-12 + \w + \ww,-12 + \w + \ww);
    \draw[dashed] (14 - \w - \ww, 14 - \w - \ww) rectangle (-14 + \w + \ww,-14 + \w + \ww);
    \draw[dashed] (16 - \w - \ww, 16 - \w - \ww) rectangle (-16 + \w + \ww,-16 + \w + \ww);

    \draw[dashed,->] (-10 - \w - \ww, -10 - \w - \ww) -- (10 + \w + \ww, -10 - \w - \ww);
    \draw[dashed,->] (-10 - \w - \ww, -10 - \w - \ww) -- (-10 - \w - \ww, 10 + \w + \ww);

    \draw[dashed,->] (-12 - \w - \ww, -12 - \w - \ww) -- (-12 - \w - \ww, 12 + \w + \ww);
    \draw[dashed,->] (-14 - \w - \ww, -14 - \w - \ww) -- (-14 - \w - \ww, 14 + \w + \ww);
  \end{tikzpicture}
  \caption{Graphical depiction of morphisms and reflection orbits for 
the cell modules of $b^x_{13}$ 
with quantisation parameters $w_1=3$ and $w_2=1$.}\label{fig:w1andw2-intro}
\end{figure}

The homomorphisms found in \cite{mgp3} are not shown there to be a complete set,
so only give a lower bound on the size of blocks. 
However
these results combined with a result about the action of certain
central elements on the standard modules 
allow us to 
obtain an {\em upper} bound on the size of blocks.
The homomorphisms (along with some restriction results to the blob
algebra) then allow a complete characterisation of the blocks.

Algebras related to towers of recollement \cite{cmpx} often have a 
geometric linkage principle, describing their blocks  
in terms of an alcove geometry on some Euclidean `weight' space, 
similar to that seen in Lie theory 
\cite{jantz,ander1}.  
In some cases the link with Lie theory is direct Ringel duality
\cite{dlabring2,ringel2,Martin92}, 
and in others it can be intriguingly less direct
(cf. e.g. 
\cite{marwood00,marwood98,CoxDevisscherMartin0609,brauerstroppelehrig}). 
A uniform recipe for this is not yet known. 
In such characterisations there are two
challenges that vary in difficulty. 
A fundamental one is the complexity of the underlying weight
space and arrangement of reflection hyperplanes. 
Then  
there is 
the `(parabolic) Kazhdan-Lusztig polynomial aspect' 
\cite{Soergel97a,Soergel97b,Deodhar94}: 
determining which
reflections between weights correspond to homomorphims between modules. 
To illustrate these points, we refer first to the
Temperley-Lieb and blob algebras over $\C$. 
Both  have 
$\R$ as the   
underlying space. 
The reflection hyperplanes are
also easily described 
(see \cite[Ch.7]{marbk} and \cite{marwood00} respectively for the in-depth
results). 
However in the case of the former only reflections of
weights through the adjacent hyperplanes correspond to non-zero
homomorphisms, whereas in the latter we have homomorphisms coming from
all reflections. 
At the other end of the spectrum we have the Brauer
algebra, where the underlying space and alcove geometry is much more
complicated, as well as the correspondence between the reflections and
the representation theory of the algebra. 
In the process of studying the symplectic
blob algebra, we hope to obtain an example 
which sheds light on the general phenomenon indicated by 
these two extremes. 
As indicated by Figure
\ref{fig:w1andw2-intro}, 
taken from Section~\ref{sec:w1andw2}, 
this paper does indeed report progress on this front.

\medskip

The paper is structured as follows. 
We give a brief review of notation together with an index in section \ref{sect:notn},
and of the construction of cell modules in \S\ref{sect:cell}.
In \S\ref{sec:repara} we discuss the role of the ground ring. 
In \S\ref{sec:dnpath} we review the De Gier--Nichols path basis of
cell modules. The first main theorems are in \S\ref{sec:necblocks},
which 
gives conditions for two cell modules to be in the same block. In
\S\ref{sec:gram} we compute Gram determinants, and in
\S\ref{sec:decompblocks}--\ref{sec:wnb} the main theorems on
decomposition matrices and geometric characterisation of blocks are
given.

\section{Notation and preliminary definitions} \label{sect:notn}

Let $k$ be a field and 
\begin{equation} \label{eq:defddd}
\ddd = (\dd,\dl,\dr,\kl,\kr,\kk ) \in k^6 .  
\end{equation}

\noindent 
Let $\NN$ denote the natural numbers including $0$. 
Let $n,m \in \NN$.

\noindent
The set $\Bx_n$ of 
 {\em \lrb\ pseudo-diagrams} \cite{mgp3} may be defined as    
follows.
Consider the set of decorated Temperley--Lieb diagrams on $n$ strings
in Figure \ref{fig:1} 
(as usual for Temperley--Lieb diagrams, 
isotopic pictures are identified). 
Then $\Bx_n$ is 
the set of pictures (up to isotopy) obtained by stacking such pictures. 
Write $d|d'$ for diagram $d$ stacked over diagram $d'$.

\begin{table}[ht]
$$
\begin{tabular}{|c|c|c|c|c|c|c|c|}
\hline
$\raisebox{-0.2cm}{\epsfbox{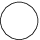}}$ 
&$\raisebox{-0.6cm}{\epsfbox{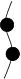}}$ 
&$\raisebox{-0.6cm}{\epsfbox{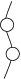}}$ 
&$\raisebox{-0.2cm}{\epsfbox{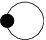}}$ 
&$\raisebox{-0.2cm}{\epsfbox{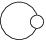}}$ 
&$\raisebox{-0.2cm}{\epsfbox{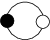}}$ 
&$\raisebox{-0.6cm}{\epsfbox{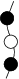}}$ 
&$\raisebox{-0.6cm}{\epsfbox{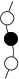}}$ 
\\
\hline
\end{tabular}
\qquad
\qquad
\raisebox{-.63cm}{\epsfbox{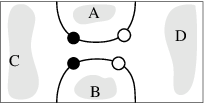}}
$$
\caption{
Some features appearing in diagrams in the set $\Bx_n$.
\label{blobtab0}}
\end{table}

\begin{table}[ht]
$$
\begin{tabular}{|c|c|c|c|c|c|c|c|}
\hline
$\raisebox{-0.2cm}{\epsfbox{loop.eps}}$ 
&$\raisebox{-0.6cm}{\epsfbox{llline.eps}}$ 
&$\raisebox{-0.6cm}{\epsfbox{rrline.eps}}$ 
&$\raisebox{-0.2cm}{\epsfbox{lloop.eps}}$ 
&$\raisebox{-0.2cm}{\epsfbox{rloop.eps}}$ 
&$\raisebox{-0.2cm}{\epsfbox{lrloop.eps}}$ 
&$\raisebox{-0.6cm}{\epsfbox{lrlline.eps}}$ 
&$\raisebox{-0.6cm}{\epsfbox{rlrline.eps}}$ 
\\
\hline
$\delta$ 
& $\dl \raisebox{-0.6cm}{\epsfbox{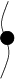}}$ 
& $\dr \raisebox{-0.6cm}{\epsfbox{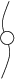}}$ 
& $\kl$
& $\kr$
& $\kk$
& $\kk\raisebox{-0.6cm}{\epsfbox{lline.eps}}$ 
& $\kk\raisebox{-0.6cm}{\epsfbox{rline.eps}}$ \\
\hline
\end{tabular}
$$
\caption{Table encoding  
straightening relations for  $b^{x}_n$.\label{blobtab}}
\end{table}

Let
$\Bxx_n$ denote the subset of   $\Bx_n$   
excluding diagrams with features as in  
Table~\ref{blobtab0}.
Given  
$d \in \Bx_n$, an element $f(d)$ of $k \Bxx_n$ is
obtained by
applying the straightening relations encoded in Table~\ref{blobtab} 
(the feature on the top is replaced by the 
given scalar multiple of the feature beneath) 
and  
 the ``topological relation'':  
\begin{equation} \label{topquot}
\kappa_{LR} \;\; \raisebox{-0.2in}{\epsfbox{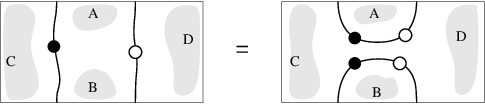}}
\end{equation}
(where each labelled shaded area is a subdiagram without propagating lines)
until such operations are exhausted. 
It is shown in \cite{mgp3} that $f(d)$ does not depend on the details
(i.e. we have confluence in a Bergman diamond sense). 
Thus we have in particular a well-defined map 
$\Bxx_n \times \Bxx_n \rightarrow k \Bxx_n$
given by 
\begin{equation} \label{de:fdd}
(d,d') \mapsto f(d|d') .
\end{equation}

\begin{defn} 
Fix $k$ and $\ddd \in k^6$. 
Then the {\em symplectic blob algebra}
$\bx_n    = \bx_n(\ddd)  $ is the $k$-algebra 
with basis
$\Bxx_n$, and 
multiplication as in (\ref{de:fdd}).
\end{defn}

For example, 
consider the poset $(\Lambda_n , \prec )  $ given in Fig.\ref{fig:cmwl}(a)
and the elements $\de_l$ ($l \in \Lambda_n$) of  $\bx_n$ as indicated
in Figure \ref{fig:cmwl2},
where in particular
  \begin{align*}
    d_0 = 
 \tikz[scale=0.5,baseline={(0,0.4)}]{
      \draw (1,2) to[out=-90,in=-90] node[pos=0.5]{$\bullet$} +(1,0);
      \draw (3,2) to[out=-90,in=-90] node[pos=0.5]{$\bullet$} +(1,0);
      \draw (1,0) to[out=90,in=90] node[pos=0.5]{$\bullet$} +(1,0);
      \draw (3,0) to[out=90,in=90] node[pos=0.5]{$\bullet$} +(1,0);
      \draw (2.5,1) node {$\cdots$};
      \draw (4.5,0) to[out=90,in=90] node[pos=0.5]{$\bullet$} +(1,0);
      \draw (4.5,2) to[out=-90,in=-90] node[pos=0.5]{$\bullet$} +(1,0);
      \draw (0.5,2) rectangle (6,0);}\text{~~~$\;$ if $n$ is even, }
\qquad
    d_0&=  
 \tikz[scale=0.5,baseline={(0,0.4)}]{
      \draw (1,2) to[out=-90,in=-90] node[pos=0.5]{$\bullet$} +(1,0);
      \draw (3,2) to[out=-90,in=-90] node[pos=0.5]{$\bullet$} +(1,0);
      \draw (1,0) to[out=90,in=90] node[pos=0.5]{$\bullet$} +(1,0);
      \draw (3,0) to[out=90,in=90] node[pos=0.5]{$\bullet$} +(1,0);
      \draw (2.5,1) node {$\cdots$};
      \draw (4.5,0) to[out=90,in=90] node[pos=0.5]{$\bullet$} +(1,0);
      \draw (4.5,2) to[out=-90,in=-90] node[pos=0.5]{$\bullet$} +(1,0);
      \draw (6,0) to node[pos=0.5]{$\bullet$} (6,2);
      \draw (0.5,2) rectangle (6.5,0);}\text{~~~$\;$ if $n$ is odd.}
  \end{align*}
Verification of the following is a simple exercise in $\bx_n$ arithmetic.

\begin{prop}[\cite{gensymp}]
The ideals of $\bx_n$ generated by the elements $\de_l$ in Figure \ref{fig:cmwl2}
include according to the 
poset structure indicated by the Hasse diagram in Figure \ref{fig:cmwl}.
\qed
\end{prop}

Let $I_n(0)$ denote the ideal generated by $d_0$.
Define $b_n'$ as the quotient algebra by this ideal.

\begin{figure} 
$$
(\Lambda_n , \prec ) \;\; = \;  
\begin{array}{c}
\xymatrix@R=8pt@C=4pt{
&0 \ar@{-}[dl] \ar@{-}[dr]& \\
1 \ar@{-}[d]\ar@{-}[drr] & &-1 \ar@{-}[d]\ar@{-}[dll]\\
2 \ar@{-}[d]\ar@{-}[drr] & &-2 \ar@{-}[d]\ar@{-}[dll]\\
{\genfrac{}{}{0pt}{}{\vdots}{\vdots}} \ar@{-}[d]\ar@{-}[drr] & 
&{\genfrac{}{}{0pt}{}{\vdots}{\vdots}} \ar@{-}[d]\ar@{-}[dll]\\
n-2 \ar@{-}[d]\ar@{-}[drr] & &-n+2 \ar@{-}[d]\ar@{-}[dll]\\
n-1 \ar@{-}[dr] & &-n+1 \ar@{-}[dl]\\
&-n& 
}
\end{array}
\qquad
\mbox{(b) }
\;\; (\Lambda_n^+ , \prec ) \;\; =   \!
\begin{array}{c}
\xymatrix@R=8pt@C=4pt{
&b \ar@{-}[dl] \ar@{-}[dr]& \\
\we{2m,1}{-+} \ar@{-}[d]\ar@{-}[drr] & &\we{2m,1}{+-} \ar@{-}[d]\ar@{-}[dll]\\
\we{2m,3}{--} \ar@{-}[d]\ar@{-}[drr] & &\we{2m,1}{++} \ar@{-}[d]\ar@{-}[dll]\\
{\genfrac{}{}{0pt}{}{\vdots}{\vdots}} \ar@{-}[d]\ar@{-}[drr] & 
&{\genfrac{}{}{0pt}{}{\vdots}{\vdots}} \ar@{-}[d]\ar@{-}[dll]\\
\we{2m,2m-1}{--} \ar@{-}[d]\ar@{-}[drr] & &\we{2m,2m-3}{++} \ar@{-}[d]\ar@{-}[dll]\\
\we{2m,2m-1}{-+} \ar@{-}[dr] & &\we{2m,2m-1}{+-} \ar@{-}[dl]\\
&\we{2m,2m-1}{++}& 
}
\end{array}
$$
\caption{(a) Cell-module weight-label poset.
Here $0$ is the maximal element and $-n$ is the minimal element.
(b) Cell-module weight label poset in DN-labelling
(even $n=2m$ case), see \S\ref{sec:dnlabel}.
\label{fig:cmwl}}
\end{figure}

\section{Review of construction of $\bx_n$ 
  cell modules}\label{sect:cell}

\begin{figure} 
$$
\begin{array}{c}
\xymatrix@R=26pt@C=6pt{
&\de_0 \ar@{-}[dl] \ar@{-}[dr]& \\
 \de_1 \ar@{-}[d]\ar@{-}[drr] & &  \de_{-1} \ar@{-}[d]\ar@{-}[dll]\\
 \de_2 \ar@{-}[d]\ar@{-}[drr] & &  \de_{-2} \ar@{-}[d]\ar@{-}[dll]\\
{\genfrac{}{}{0pt}{}{\vdots}{\vdots}} \ar@{-}[d]\ar@{-}[drr] & 
      &{\genfrac{}{}{0pt}{}{\vdots}{\vdots}} \ar@{-}[d]\ar@{-}[dll]\\
{\de_{m-1}=\dpmmone} \ar@{-}[d]\ar@{-}[drr] & 
         &{\de_{-(m-1)}} \ar@{-}[d]\ar@{-}[dll]\\
{\de_{m}=\dmmone} \ar@{-}[d]\ar@{-}[drr] & 
         &{\de_{-m}=\dppone} \ar@{-}[d]\ar@{-}[dll]\\
{\de_{m+1}} \ar@{-}[d]\ar@{-}[drr] & 
         &{\de_{-(m+1)}=\dmpone} \ar@{-}[d]\ar@{-}[dll]\\
{\vdots} \ar@{-}[d]\ar@{-}[drr] & &{\vdots} \ar@{-}[d]\ar@{-}[dll]\\
\de_{n-1} \ar@{-}[dr] & & \de_{-n+1} \ar@{-}[dl]\\
&\de_{-n}& 
}
\end{array}
$$
\caption{Representative diagrams in the cell ideal poset.
\label{fig:cmwl2}}
\end{figure}

\noindent
Consider the poset $(\Lambda_n , \prec )  $ given in Figure \ref{fig:cmwl}.
A set $\{ \Sb_{n}(l) \}_{l \in \Lambda_n}$
of $\bx_n$-modules 
is constructed over arbitrary $k$ in \cite{gensymp}. 
In this section we review the construction.
One should start by thinking of $k$ not as a field but 
rather as the commutative ring $\Z[\ddd ]$ here. 
Then one can pass to any case by base change. 
These modules pass to simple modules in the semisimple cases
(see \cite{gensymp}), 
so they can be thought of as the 
integral forms of the `ordinary' irreducibles in a
Brauer-modular system \cite{Brauer39,benson}. 
(Although our setup requires careful preparation to be properly
modular, cf. \cite{benson}  
--- we will not need to develop the full machinery here.)

\begin{figure}
\[ 
\begin{tikzpicture}[scale=0.4]
    \draw (0,3) to[out=-90, in=-90] node[pos=0.5]{$\bullet$} +(1,0);
    \draw (2,3) to node[pos=0.5]{$\bullet$} +(0,-2);
    \draw (3,3) to +(0,-2);
    \draw (4,3) to node[pos=0.5,white]{$\bullet$} node[pos=0.5]{$\circ$} +(0,-2);
  \end{tikzpicture},\hspace{1em}
 \begin{tikzpicture}[scale=0.4]
    \draw (0,3) to[out=-90, in=-90] +(1,0);
    \draw (2,3) to node[pos=0.5]{$\bullet$} +(0,-2);
    \draw (3,3) to +(0,-2);
    \draw (4,3) to node[pos=0.5,white]{$\bullet$} node[pos=0.5]{$\circ$} +(0,-2);
  \end{tikzpicture},\hspace{1em}
 \begin{tikzpicture}[scale=0.4]
    \draw (1,3) to[out=-90, in=-90] +(1,0);
    \draw (0,3) to[out=-90, in=90] node[pos=0.5]{$\bullet$} +(2,-2);
    \draw (3,3) to +(0,-2);
    \draw (4,3) to node[pos=0.5,white]{$\bullet$} node[pos=0.5]{$\circ$} +(0,-2);
  \end{tikzpicture},\hspace{1em}
 \begin{tikzpicture}[scale=0.4]
    \draw (2,3) to[out=-90, in=-90] +(1,0);
    \draw (0,3) to[out=-90, in=90] node[pos=0.5]{$\bullet$} +(2,-2);
    \draw (1,3) to[out=-90, in=90] +(2,-2);
    \draw (4,3) to node[pos=0.5,white]{$\bullet$} node[pos=0.5]{$\circ$} +(0,-2);
  \end{tikzpicture},\hspace{1em}
  \begin{tikzpicture}[scale=0.4]
    \draw (3,3) to[out=-90, in=-90] +(1,0);
    \draw (0,3) to[out=-90, in=90] node[pos=0.5]{$\bullet$} +(2,-2);
    \draw (1,3) to[out=-90, in=90] +(2,-2);
    \draw (2,3) to[out=-90, in=90] node[pos=0.5,white]{$\bullet$} node[pos=0.5]{$\circ$} +(2,-2);
  \end{tikzpicture},\hspace{1em}
\begin{tikzpicture}[scale=0.4]
    \draw (3,3) to[out=-90, in=-90] node[white,pos=0.5]{$\bullet$}node[pos=0.5]{$\circ$} +(1,0);
    \draw (0,3) to[out=-90, in=90] node[pos=0.5]{$\bullet$} +(2,-2);
    \draw (1,3) to[out=-90, in=90] +(2,-2);
    \draw (2,3) to[out=-90, in=90] node[pos=0.5,white]{$\bullet$} node[pos=0.5]{$\circ$} +(2,-2);
  \end{tikzpicture}.
\]
\caption{
Half-diagram basis of  cell module 
   $\Sb_{5}(1) = W^{(5,2)}_{-,-}$.
\label{fig:w52}}
\end{figure}

The left  $b_n^x$-module
$\Sb_{n}(l)$
has a 
basis of half-diagrams constructed similarly to 
the blob algebra case
(cf. \cite[p. 593]{blobcgm}, \cite[Section 8]{gensymp}).
See Figure \ref{fig:w52} for an example.
Note that by (\ref{de:fdd}) the left action corresponds to stacking a
diagram on top of the basis element.

Consider $l \in \Lambda_n$. To construct a basis $\basis_{n}(l)$ 
 for $\bx_n$-module  $\Sb_{n}(l)$ in general we proceed as follows. 
Consider the subset of $\Bxx_n$ of diagrams with $|l|$ undecorated propagating
lines. 
If
$l$ is positive, then further restrict to 
diagrams with a left blob on the first propagating line. 
Otherwise, if $l$ is negative, then there must be no such blob. 
Now pick any one of the remaining diagrams $d$, and take the subset 
of diagrams agreeing with $d$ in the lower half. 
Finally,
as the lower half is the same in all diagrams, and does not 
affect multiplication, we omit it.
(As another example, 
half-diagram bases for the cell modules for low rank $b_n^x$ are
listed in \cite[Figure 3]{gensymp}. 
There cut lines are used in place of blobs.)
The algebra action is by diagram stacking, except that diagrams
arising that lie outside the basis (necessarily with 
higher weight in the sense of Figure \ref{fig:cmwl}) are zero.

The case with no decorated propagating lines is easiest to explain.
In this case, as a left $b_n^x$-module, the $2$-sided ideal $b_n^x d_0 b_n^x$ is a direct sum of
isomorphic copies of the cell module $\Sb_n(0)$ where the number of
such copies is the same as the number of possible lower half diagrams.
We have $\Sb_n(0) \cong b_n^x d_0$.
(This is the formulation used in
Green et al's 
original analysis of the representation theory of $b_n^x$
\cite{gensymp},
but not that  used in the subsequent crucial work of 
De Gier and Nichols
\cite{degiernichols}.
When helpful, we  
colloquially refer to this as the ``blob-theoretic'' definition to distinguish
from other formulations.)

Recall:

\begin{prop}[{\cite{gensymp}}] The algebra $\bx_n$ is a cellular algebra,
in the sense of \cite{GrahamLehrer}. 
The modules $\{ \Sb_{n}(l) \}_{l \in \Lambda_n}$ are the cell modules.
The  labelling poset $\Lambda_n$ for the cell modules 
is as in Figure \ref{fig:cmwl}.
\end{prop}

When all parameters are invertible, 
$\Lambda_n$ also labels the simple modules, 
in which case the algebra is also quasi-hereditary with the
above poset and the cell modules are standard modules.

\newpage
As an aid to the reader we include the following index of notation.
\medskip

\begin{tabular}{ll}
$ \alpha^{(n,m)}_{\ee_1,\ee_2} = [n] \frac{[2(-m+\ee_1w_1+\ee_2w_2)]}{[-m+\ee_1w_1+\ee_2w_2]} $ & 
scalar for the action of $Z_n$ (theorem 6.7)\\
  ASTL & Affine Symmetric Temperley-Lieb\\
$b$ & De Gier-Nichols parameter (plays same role as $\kappa_{LR}$)\\
$b^x_n$ & symplectic blob algebra \\
  $B_x$ & left-right blob pseudo-diagrams without diagrams with
          features in table 1\\
  $B_x'$ & left-right blob pseudo-diagrams \\
 $\basis_{n}(l)$  & basis for the cell module $\Sb_n(l)$\\
$\WWB$& basis  of $W^{(n,m)}_{\ee_1,\ee_2}$\\ 
  $\C$ & complex numbers \\
  $d|d'$ & diagram $d$ stacked over $d'$\\
  $d_i$ & element of $\bx_n$, as defined in figure 4\\
  $\ddd = (\delta, \delta_L, \delta_R, \kappa_L, \kappa_R,
  \kappa_{LR})$ & $6$-tuple of parameters (cf. table 2)\\
DN & De Gier-Nichols parameterisation (table 3)\\
  $e$ & left blob (cf. figure 1)\\
  $e_i $ & TL generator (cf. figure 1)\\
  $E_n'= d_0$ &  element of $\bx_n$ (section 4)\\
$\ee_i \in \{\pm 1\}$ & sign parameters for cell modules \\ 
  $f$ & right blob (cf. figure 1)\\
  $F$ & a localisation functor (proposition 8.2)\\
  $F'$ & a localisation functor (proposition 8.2)\\
$f(h)$ & definition 7.4\\
$g_0$, $g_1$, $g_2$, $\ldots$, $g_n$& generators for the Hecke algebra of
                               type $\tilde{C}$\\
  $G$ & a globalisation functor (proposition 8.2)\\
  $G'$ & a globalisation functor (proposition 8.2)\\
$\Gram{n,m}{\ee_1,\ee_2}$& the Gram matrix (section 7)\\ 
$\WWGam$ & Gram determinant (section 7)\\ 
$g(h)$ & definition 7.4\\
  GMP1 & A Green-Martin-Parker parametrisation (table 3)\\
  GMP2 & A Green-Martin-Parker parametrisation (table 3)\\
  $H(\tilde{C}_n)$ & Hecke algebra of type affine-C\\
  $J_i$ & `Jucys-Murphy' elements of $H(\tilde{C}_n)$  (definition 6.3)\\
  $k$ & an algebraically closed field \\
$k(u) = - \frac{[(u-w_1 +\theta)/2+1][(u-w_1-\theta)/2]}{[u][w_2+1]}$ & an element of $k[\ddd]$\\
$\lambda_p$ & eigenvalue associated to path $p$ of Gram matrix (Proposition 7.5)\\ 
 \end{tabular}

\begin{tabular}{ll}
 $(\Lambda_n, \prec)$& labelling poset $=\{-n,-n+1, \ldots, 0, 1, \ldots,
              n-1\}$ with order as in figure 3\\
$ (\Lambda_n^+ , \prec )$ & De Gier-Nichols labelling poset see figure
                            3\\
$[m]$ & quantum integer (section 3.1)\\
  $\N_0$& natural numbers (including $0$)\\
$p_0$ & fundamental path  (section 4) \\
$\mathcal{P}_n$ & set of paths  (section 4) \\
$\pi_n$ & set of paths $w_p$ giving a diagram basis for $W^{(n)}(b)$ (section 4)\\
$\Pi_n$ & set of paths $v_p$ giving a diagram basis for $W^{(n)}(b)$ (section 4)\\
  $q = [2] = \delta + \delta^{-1}$ & ``bulk'' parameter\\
$q$, $Q_1$, $Q_2$ & indeterminates for the Hecke algebra of type
                    $\tilde{C}$\\
$r(u) = \frac{[u+1]}{[u]}$ & an element of $k[\ddd]$\\
 $\{ \Sb_{n}(l) \}_{l \in \Lambda_n}$& the cell modules for $b_n^x$\\
$\mathcal{T}_n$ & Tchebychev recursion (section 3.1) \\
$\theta$ & De Gier-Nichols parameter that reparametrises $b$\\
  TL & Temperley-Lieb\\
  2BTL & two boundary Temperley-Lieb\\
  $\uri(d)$ & the number of lines crossing the right wall \\
  $\uro(d)$ & the number of lines crossing the left wall \\
  $v_p$ & an element of $b^x_nE'_n$ associated to the path $p$ (section 4)  \\
  $w_1$ & a quantisation parameter \\
  $w_2$ & a quantisation parameter \\
  $w_p$ & an element of $b^x_nE'_n$ associated to the path $p$ (section 4)  \\
$\W{n, m}{\ee_1,  \ee_2}$  & De Gier-Nichols Cell module\\
  $Z_n=\sum_{i=0}^{n-1} (J_i + J_i^{-1})$ & a central  element\\
$\langle - , - \rangle$ & inner product on the cell module (section 7)\\
 \end{tabular}

\subsection{On standard and De~Gier--Nichols weight labelling}\label{sec:dnlabel}
\newcommand{\astl}{affine-symmetric TL}
\newcommand{\ASTL}{ASTL}

In  \cite{degiernichols} 
there is a useful reformulation of $ (\Lambda_n , \prec ) $ 
as follows.
The basis $\Bxx_n$ is equivalent to a basis of \astl\ (\ASTL) diagrams 
(see \cite{gensymp} for the equivalence). 
In an \ASTL\ diagram 
``blobs'' are indicated by paired lines that touch the left (for a
left blob) or the right (a right blob) side of a diagram.
A corresponding 
half-diagram can in principle have any number of lines touching the
left or right side, but the parity of each number is preserved in the
(\ASTL\ version of the) basis of a cell module. 
Thus for $\ee_i \in \{\pm 1\}$ the module 
$\W{n, m}{\ee_1,  \ee_2}$ 
is the cell module with \ASTL\ half-diagram basis with $\ee_1$
parity  on the left side, $\ee_2$ parity  on the right
side and $m + \frac12 (\ee_1 + \ee_2)$ propagating lines 
(here $\ee_i = +1$ for even,  written as $+$; and
$\ee_i = -1$ for odd, written as $-$). 
We write $\Lambda_n^+$ for the new labelling scheme 
--- see Figure \ref{fig:cmwl}(b).
(Note in \cite{degiernichols} they have brackets on the cell
modules. We have dropped the brackets as the notation is already
complicated enough. So our $\W{n,m}{\ee_1, \ee_2}$ is their
$W^{(n,m)}_{\ee_1,\ee_2}$.)

The correspondence (in both directions) is given as follows:
$$
S_n(l) = \begin{cases}
\W{n,|l|}{-\sgn l,\, \sgn l} &\mbox{ if $n$ and $l$ have opposite
  parity, $l \ne 0$},\\
\W{n,|l+1|}{-\sgn l,\, -\sgn l} &\mbox{ if $n$ and $l$ have the same
  parity, $l \ne 0$}, \\
\W{n}{}(b) & \mbox{ if $l =0$}
\end{cases}
$$
\beq \label{eq:WbE}
\W{n,l}{\ee_1, \ee_2} =
\Sb_n(-\ee_1(l+\frac12(\ee_1+\ee_2)))
\eq
where $\sgn l$ is the sign of $l$. 

Remark: The argument $b$ used for the cell module
with no propagating lines indicates that
the structure of
this module
depends on a parameter $b$. This  is essentially the same as 
$\kappa_{LR}$ (see \S\ref{ss:para}).

\section{On $\ddd$ 
   parameter conventions and reparameterisation}
\label{sec:repara}
\subsection{Ground ring arithmetic}  
In the modular system \cite{Brauer39}
one works largely in the integral ground ring,
passing to a specific modular case 
(to address specific Artinian representation theory)
as late as possible.
However for reasons of arithmetic manipulation it may be expedient to 
perform computations as if in a different ground ring.
This  looks like base-change away from the generality of the
integral ring. 
But provided the change is
arithmetically reversible back to the integral ring, it is not
restrictive. 

An example is as follows. 
The substitution homomorphism $\Z[\delta] \rightarrow \Z[q,q^{-1}]$
given by $\delta \mapsto q+q^{-1}$ is not an isomorphism.
However it is an injection, so the map can be inverted on any element
of the image. 
Thus one can do arithmetic on elements of $\Z[\delta]$ working in the
image, and then recover identities that hold in $\Z[\delta]$. 

In the example, a merit of the substitution is if one works with
elements of $\Z[\delta]$ satisfying the recursion 
$\mathcal{T}_n = [2]\mathcal{T}_{n-1}-\mathcal{T}_{n-2}$, with $\mathcal{T}_0=0 $  and $\mathcal{T}_1=1$
(for example certain Gram determinants of the Temperley--Lieb algebra
satisfy this recursion 
\cite{martsaleur}).
This is the Tchebychev recursion \cite[\S6.3.3]{marbk}. 
The complex roots of $\mathcal{T}_n$ are the so-called
Beraha numbers \cite{Baxter81} --- but factorisation is not obvious. 
However working in the image these elements take the simple form 
$\mathcal{T}_n = [n]$, where
$$
[m] \;  := \; q^{-m+1} + q^{-m+3} + \cdots + q^{m-3} + q^{m-1} .
$$ 
This formulation  has manifest factorisation properties.
In particular $[n]=0$ requires $q$ to be a root of unity.

\subsection{Parameterisation by exponents $w_1, w_2$}\label{ss:para}

$\;$ 
In order to determine the representation theory of $\bnx(\ddd)$ 
it is useful
to reparameterise as discussed 
in \cite[\S2]{mgp3}. 
We recall the key  parameterisations 
in Table~\ref{tab:repara}.  
\newcommand{\mo}{\mapsto \;\;\;}
\begin{table}
\[
\begin{array}{l|rrr|rrrrrr}
    &\multicolumn{3}{c}{\mbox{generator scaling} }
  &\multicolumn{6}{c}{\mbox{parameter scaling / reparameterisation} }
\\ \hline 
\mbox{label} & e\mo  &e_i\mo &   f\mo   &  \dd\mo &  \dl\mo &  \dr\mo & \kl\mo&\kr\mo&\kk\mo 
\\ \hline 
1
&\frac{e}{\kl}&e_i&\frac{f}{\kr}&\dd&\frac{\dl}{\kl}&\frac{\dr}{\kr}&1&1&\frac{\kk}{\kl\kr}
\\
2 &-e&-e_i&-f&-\dd&-\dl&-\dr&\kl&\kr&\kk
\\ \hline 
\mbox{DN}& & && [2] &\frac{[\omega_1]}{[\omega_1 +1]}
  &\frac{[\omega_2]}{[\omega_2 +1]} & 1&1&b
\\
\mbox{GMP1}&&&&[2]&[w_1]&[w_2]&[w_1+1]&[w_2+1]&\kk\\
\mbox{GMP2}&&&&-[2]&-[w_1]&-[w_2]&[w_1+1]&[w_2+1]&\kk
\end{array}
\]
\caption{
 Alternative parameterisations for $\bnx$. \label{tab:repara}}
\end{table}
Generator scaling ``1'' 
in Table~\ref{tab:repara}
induces an isomorphism with the algebra
with  parameters rescaled as shown,
reducing from 6 parameters to 4 \cite{mgp2I}.
 ``GMP1'', ``GMP2'' and ``DN'' reparameterise with parameters
 $q,w_1,w_2$
(cf. \cite{Baxter81,Lusztig83,Lusztig99,TL,martsaleur}). 
DN is the parameter choice of De Gier--Nichols in 
\cite{degiernichols}. 
GMP1 and GMP2 are the parameter
choices that were most useful for \cite{mgp3}. 
GMP1 and GMP2 can be 
converted from one to another by taking the isomorphic algebra with
generators multiplied by  $-1$, i.e. using  ``2'' to rescale.
GMP2 turns out to be the most convenient for presenting the results about
general families of homomorphisms in \cite{mgp3}.  Then  ``1'' 
 converts from DN to GMP1 and then to GMP2 via  ``2''.

De Gier--Nichols  further 
reparameterise $b$ in terms of 
a new parameter $\theta$:
\begin{equation}
\label{eq:b}
b = \left\{ \begin{array}{ll}
      {\displaystyle{\left[{\frac{w_1+w_2+\theta +1}{2}}\right]
      \left[{\frac{w_1+w_2-\theta +1}{2}}\right]}} &\mbox{if $n$
        even}\\
         & \\
      {\displaystyle{-\left[{\frac{w_1-w_2+\theta }{2}}\right]
      \left[{\frac{w_1-w_2-\theta }{2}}\right]}} &\mbox{if $n$
        odd.}
\end{array}\right.
\end{equation} 

\medskip

\section{Bases of the $\bx_n$-module $\W{n}{}(b) = S_n(0)$}\label{sec:dnpath}

\begin{defn}
For $n \in \N$   
 a \emph{{Pascal} path} $p$ is 
an element  of the subset of $\Z^{n+1}$ given by:
\[
\PP_n \; = \; \{ \;
p=(h_0,h_1,\dots,h_n)
\; | \; h_0=0; \mbox{and  $|h_{i+1}-h_i|=1$ for 
 $0\leq i \leq n-1$ 
}  \}
\]
In particular, define the \emph{fundamental path}
$
\po =(0,-1,0,-1,0,\dots).
$
\end{defn}

\begin{figure}[ht]
  \centering
  \tikz[scale=0.5]{
    \draw[dashed] (0,0) to (6,6);
    \draw[dashed] (1,-1) to (6,4);
    \draw[dashed] (2,-2) to (6,2);
    \draw[dashed] (3,-3) to (6,0);
    \draw[dashed] (4,-4) to (6,-2);
    \draw[dashed] (5,-5) to (6,-4);
    \draw[dashed] (0,0) to (6,-6);
    \draw[dashed] (1,1) to (6,-4);
    \draw[dashed] (2,2) to (6,-2);
    \draw[dashed] (3,3) to (6,0);
    \draw[dashed] (4,4) to (6,2);
    \draw[dashed] (5,5) to (6,4);
    \draw[dashed] (6,6) to (6,-6);
    \draw[ultra thick] (0,0) to ++(1,-1) to ++(1,1) to ++(1,-1) to ++(1,1) to ++(1,-1) to ++(1,1);
    }
  \caption{The tiled lattice and the fundamental path $p_0$.}
  \label{fig:pathlattice}
\end{figure}

\begin{figure}
\includegraphics[width=2in]{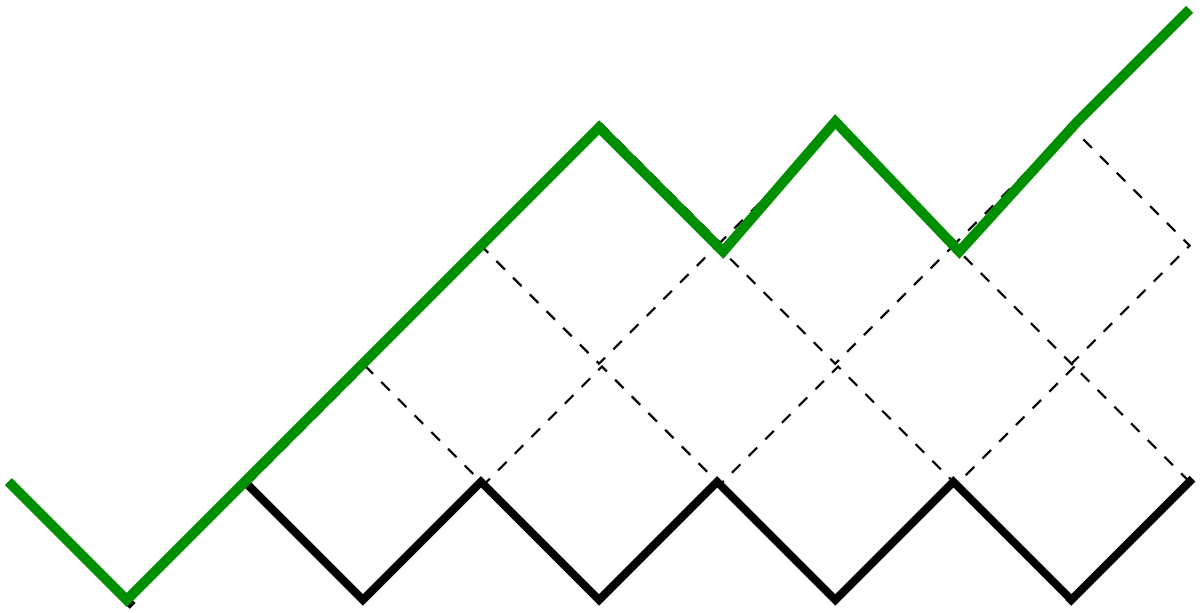}
\hspace{1cm}
\includegraphics[width=.432in]{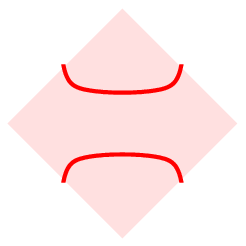}
\hspace{1cm}
\includegraphics[width=2in]{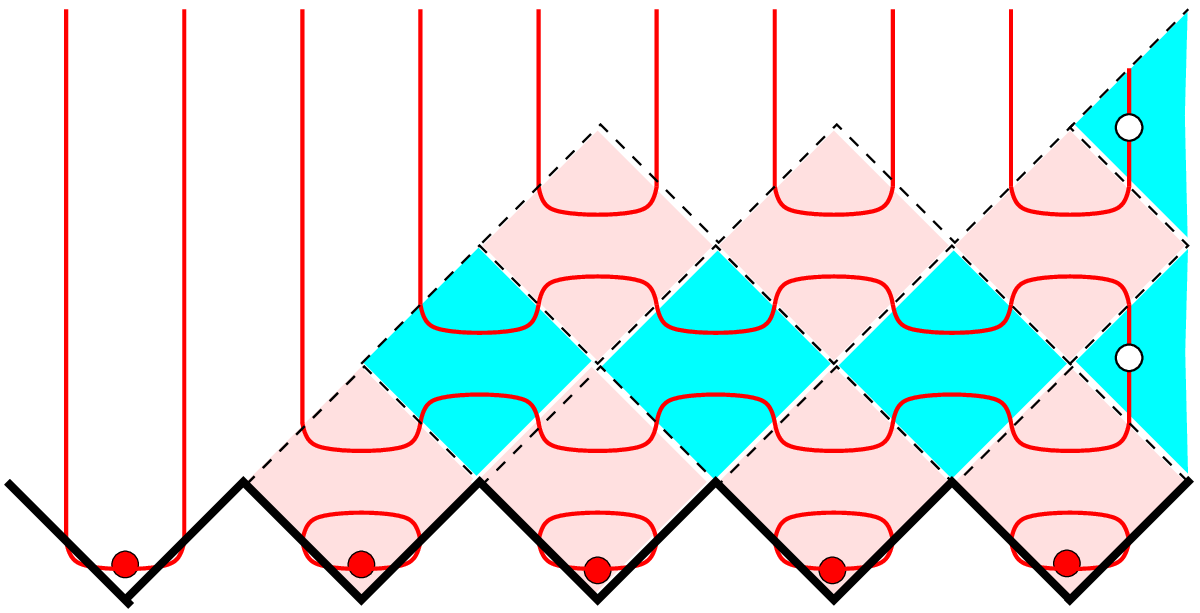}
\caption{Illustration of 
  envelope-tiling to basis-element correspondence
  in rank $n=10$.
  Example path strictly on-or-above-$p_0$  (green path);
example single tile;
  and
  tiled envelope (hence basis element in red) for this example.
\label{fig:tilex04}}
\end{figure}

\begin{figure}
\includegraphics[width=1.72in]{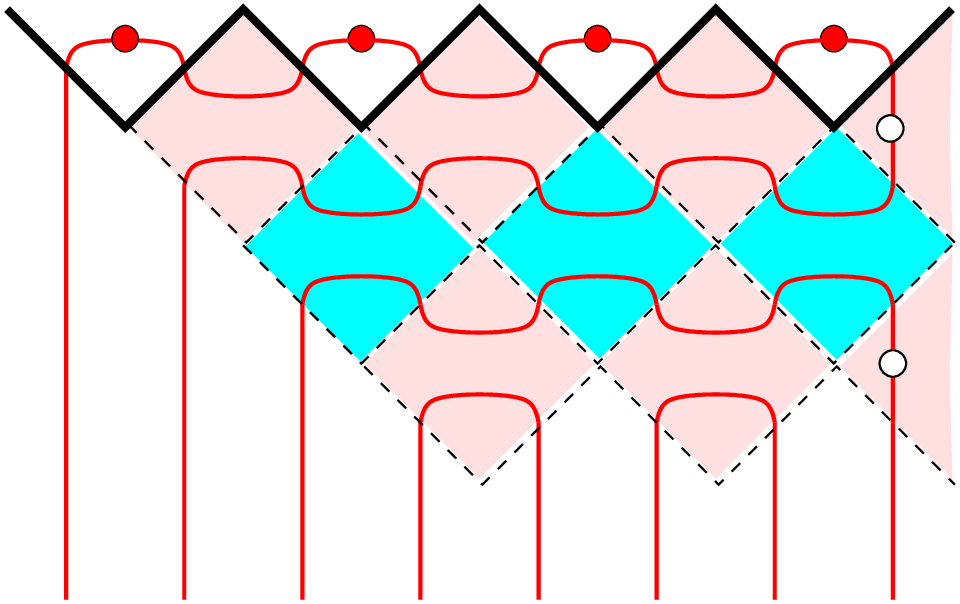}
\hspace{1cm}
\raisebox{2.53195cm}{
\includegraphics[width=1.72in]{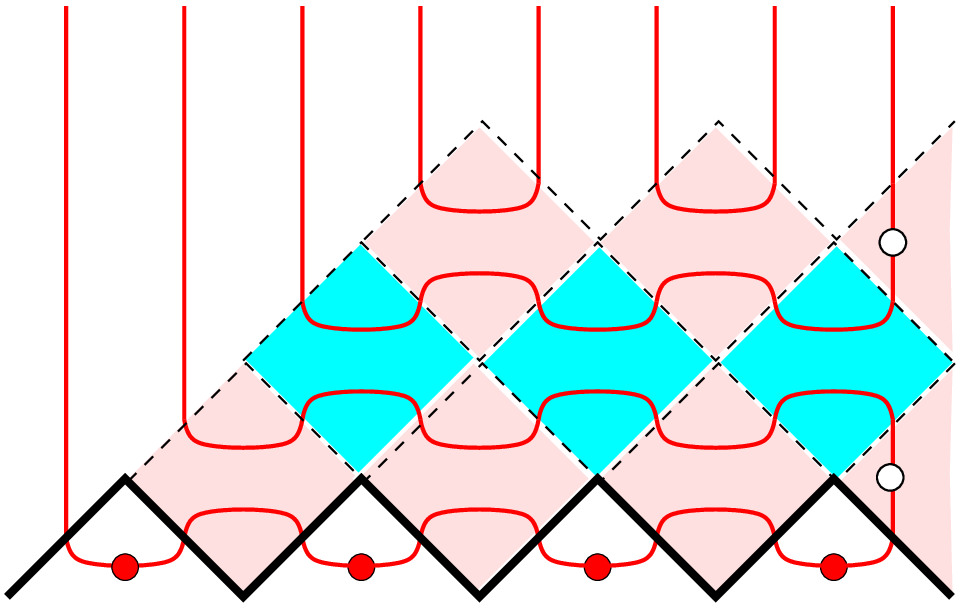}}
\caption{Envelope-tiling to basis-element correspondence.
  Path  strictly on-or-below-$p_0$ example
  --- note this renders the basis element upside-down;
  then same example {\em drawn upside-down}.
  \label{fig:tilex03}}
\end{figure}

\begin{figure}
\includegraphics[width=2in]{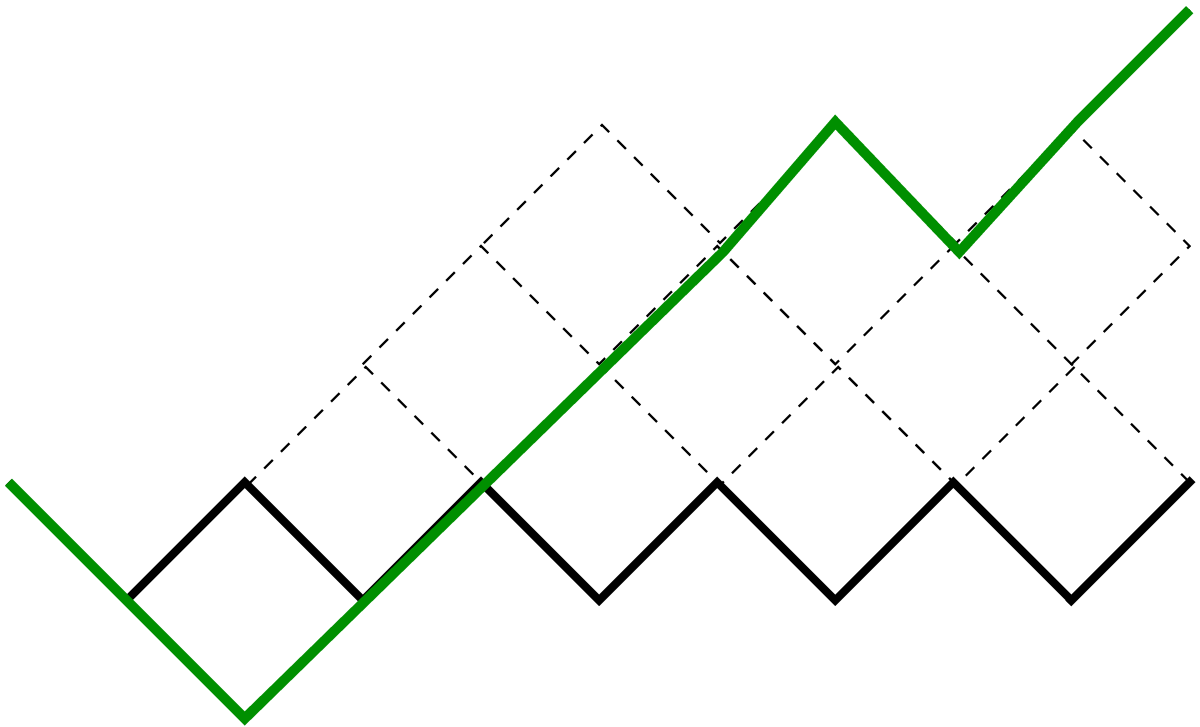}
\hspace{1cm}
\raisebox{-1.02cm}{
\includegraphics[width=2in]{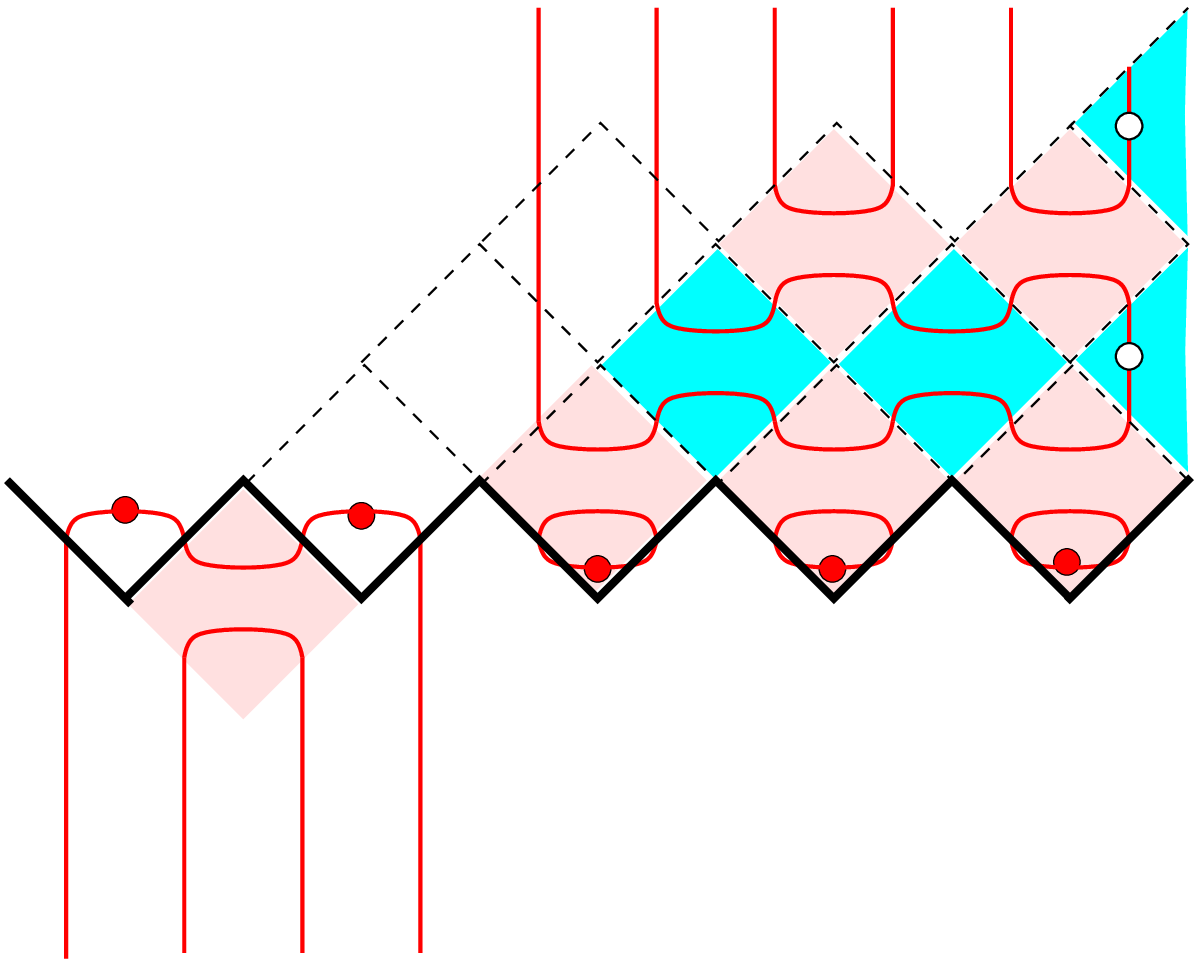}}
\caption{Illustration of 
  envelope-tiling to basis-element correspondence
  in rank $n=10$.
  Example path not strictly on-or-above-$p_0$  (green path);
  and
  tiled envelope (hence basis element in red,
  rendered as one down-oriented and one up-oriented factor) for this example.
\label{fig:tilex04-2}}
\end{figure}

\begin{figure}
  \includegraphics[width=2in]{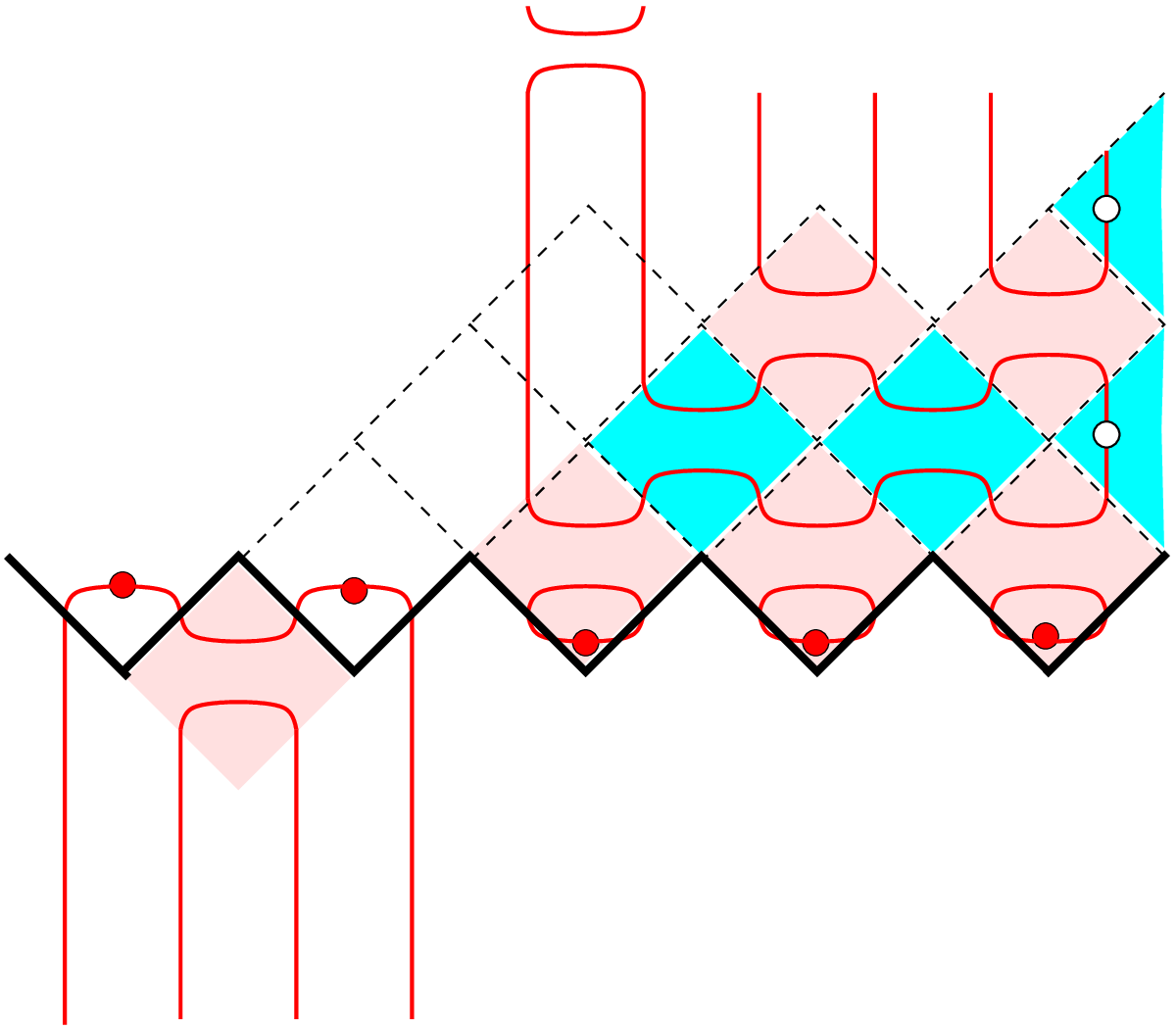}
  \hspace{1cm}
   \includegraphics[width=2in]{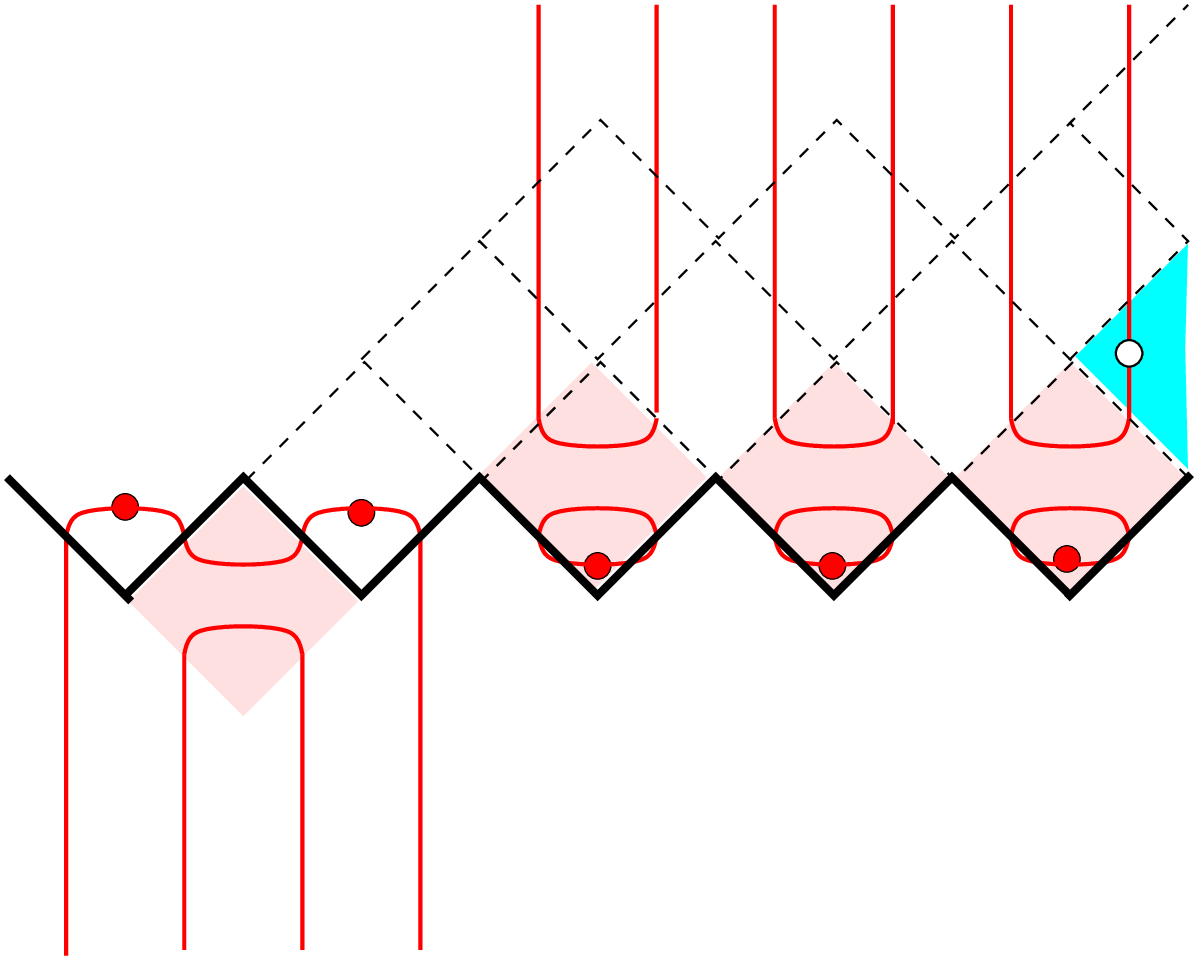} 
  \caption{Illustration of spanning property of the constructed set:
    action of $e_5$, showing $e_5 w_p$
    (on the left) lies in the span of a lower $w_{p'}$ (on the right).
\label{fig:tilex06}}
\end{figure}

We can draw Pascal
paths on a `tiled' square lattice as in Figure \ref{fig:pathlattice}.
Each path $p$ may be partitioned according to the points at
which it agrees with $\po$; 
the 
parts of $p$ that are above $\po$; 
and those
below $\po$.
Each of these latter two parts defines an `envelope' between the two paths.
We can move from $\po$ to $p$ 
through a sequence of intermediate paths $p_i$
by `adding tiles' (or half tiles on the right)  within each
envelope.
In particular note that if $p_i \neq p$ then 
there is always a lowest numbered position
(from left to right) at
which a tile can be added.
Define $\PP(p)$ as the ordered set passing from $\po$ to $p$ in this way.

\newcommand{\EE}{E'}

Define  $\EE_n$  in $b_n^x$ by $\EE_n = d_0$
(to make the dependence on $n$ manifest).
That is:
  \begin{align*}
    \EE_n = 
 \tikz[scale=0.5,baseline={(0,0.4)}]{
      \draw (1,2) to[out=-90,in=-90] node[pos=0.5]{$\bullet$} +(1,0);
      \draw (3,2) to[out=-90,in=-90] node[pos=0.5]{$\bullet$} +(1,0);
      \draw (1,0) to[out=90,in=90] node[pos=0.5]{$\bullet$} +(1,0);
      \draw (3,0) to[out=90,in=90] node[pos=0.5]{$\bullet$} +(1,0);
      \draw (2.5,1) node {$\cdots$};
      \draw (4.5,0) to[out=90,in=90] node[pos=0.5]{$\bullet$} +(1,0);
      \draw (4.5,2) to[out=-90,in=-90] node[pos=0.5]{$\bullet$} +(1,0);
      \draw (0.5,2) rectangle (6,0);}\text{~~~$\;$ if $n$ is even, }
\qquad
    \EE_n&=  
 \tikz[scale=0.5,baseline={(0,0.4)}]{
      \draw (1,2) to[out=-90,in=-90] node[pos=0.5]{$\bullet$} +(1,0);
      \draw (3,2) to[out=-90,in=-90] node[pos=0.5]{$\bullet$} +(1,0);
      \draw (1,0) to[out=90,in=90] node[pos=0.5]{$\bullet$} +(1,0);
      \draw (3,0) to[out=90,in=90] node[pos=0.5]{$\bullet$} +(1,0);
      \draw (2.5,1) node {$\cdots$};
      \draw (4.5,0) to[out=90,in=90] node[pos=0.5]{$\bullet$} +(1,0);
      \draw (4.5,2) to[out=-90,in=-90] node[pos=0.5]{$\bullet$} +(1,0);
      \draw (6,0) to node[pos=0.5]{$\bullet$} (6,2);
      \draw (0.5,2) rectangle (6.5,0);}\text{~~~$\;$ if $n$ is odd.}
  \end{align*}

\newcommand{\ppi}{\pi}     

Note from (\ref{eq:WbE}) that
$$
\W{n}{}(b)  = S_n(0) = \bnx d_0    = \bnx \EE_n . 
$$
Define a subset 
$\ppi_n = \{ w_p \mid p \in \PP_n \}$
of $\Wb$ as follows. To a path $p \in \PP_n$ we associate an
element $w_p$ defined recursively through $\PP(p)$:
firstly $w_{\po} = \EE_n$;
then $w_{p_{j+1}} = e_i w_{p_j}$ if $p_{j+1}$ obtained from $p_j$ by
adding a tile in position $i$.

To determine if $\ppi_n$ is spanning the argument is essentially
analogous to
the TL case (cf. \cite{marbk}).
We partially order paths by $p>q$ if $q$ lies in the envelope of $p$,
and work by induction on $p$ with $p_0$ as base.
We aim to show that 
$e_i w_p$ lies in the span of $\ppi_n$ for every $i$ if this holds
for each $q<p$.
(This holds for the base case since $e_i p_0$ is in the span 
by construction.) 
We need to consider the action of
elements $e_i$ on each $w_p$ when $p$ does not have a max or min at $i$
(since otherwise $e_i w_p$ is clearly in the span by construction of $\ppi_n$).
Note that if $p$ is straight at $i$
(consider e.g. $i=5$ in Fig.\ref{fig:tilex04-2} or \ref{fig:tilex06})
then
$w_p = w e_{i+1} e_i w_{p'}$ (or similar)
for some $w= e_{i+a}\cdots e_{i+2}$ commuting with $e_i$
(in our example it is simply $w=e_7$)
and some $w_{p'}$,
whereupon we have
$e_i w e_{i+1} e_i w_{p'} = w e_i  e_{i+1} e_i w_{p'} =  w  e_i w_{p'} $.
Note that $e_i w_{p'}$ is a $w_q$ for a $q<p$ and straight at $i+2$.
Thus either $w=1$ and we are done or we may iterate to an even lower path,
until the inductive step is completed. 
Thus $\ppi_n$ is spanning. And then, comparing with the dimension of
 $\Wb$ we have:

\begin{thm} \label{th:dbb}
  The subset $\ppi_n$ is
a basis for $\Wb$ as a left $\bnx$-module.
\qed
\end{thm}

\subsection{Path basis for $\Wb$ for generic $\ddd$} $\;$

Here we will use a notion of \emph{generic} $\ddd \in k^6$
\cite{Hartshorne}.
A point is \emph{generic} if it lies in the (Zariski) open subset
excluding   
a certain variety
(in our case the variety given by the collection of denominators in
a construction below --- see (\ref{eq:r(u)k(u)})).
The utility is that every $\ddd$ in $\C^6$ is the limit  of a set of
generic points, so that certain identities $f(\ddd)=0$ that hold generically
will hold at every point where  
$f$ makes sense.

We define a formal subset of the $b_n^x$-module 
$\Wb = \bnx \EE_n$ for generic $\dd$. 
To a path $p$ we associate an element $v_p$, defined recursively 
through $ \PP(p)$ as follows:
\begin{align*}
  v_{p_0}&= \EE_n,\\
  v_{p'}&=Y_iv_p\text{ if $p'$ is obtained from $p$ by adding a tile at position $i$},
\end{align*}
where $Y_i$ is one of the following operators:
\begin{itemize}
\item $X_i=e_i-r(w_1-h_{i-1}) 1 \;$ if a full tile is added from above
  at position $i$;
\item $X_i=e_n-k(w_1-h_{n-1}) 1 \;$ if a half tile is added from above at the right boundary;
\item $X'_i=e_i-r(-w_1+h_{i-1}) 1 \;$ if a full tile is added from below
  at position $i$;
\item $X'_i=e_n-k(-w_1+h_{n-1}) 1 $   if a half tile is added from below at the right boundary,
\end{itemize}
where
\begin{equation}
  r(u)=\frac{[u+1]}{[u]}\text{ , $\;$  and }
k(u)=-\frac{[(u-w_2+\theta)/2][(u-w_2-\theta)/2]}{[u][w_2+1]}.
  \label{eq:r(u)k(u)}
\end{equation}
Define
$$
\PPi_n = \{ v_p | \; p \in \PP_n \}
$$ 
Comparing the constructions
for $\ppi_n$ and $\PPi_n$
we have immediately from Theorem~\ref{th:dbb}:
\begin{thm}
When defined, the set $\PPi_n$ can be obtained from $\ppi_n$ by an
upper-unitriangular transformation;
and hence is a basis for $\Wb$.
\qed
\end{thm}

\begin{thm}
In general there are other ways of adding tiles to pass 
from $\po$ to each $p$ (cf. the ordered sequence $\PP(p)$). 
The construction does not depend on the choice of route. 
\end{thm}
\proof For each $p$ note that if there are two routes to $p$ from some
lower path then the different sequences of multiplications involve
pairwise commuting factors. \qed

\begin{rem*}If the scalar term is omitted in $X_i$, this construction builds 
the diagram basis, up to the DN rescaling factors. 
We shall later keep track of these scalars explicitly,
and hence recover `integral-valued' Gram matrices from 
certain nominal Gram matrix calculations. 
It is interesting to contrast this with the path basis for
the 
Temperley--Lieb
case
in \cite{marbk}.
There the orthogonal basis is orthonormal, 
so the nominal Gram matrix is the identity matrix,
and one {\em only} has to work out the basis scaling factor. 
\end{rem*}

\begin{thm}[{\cite[Theorem 5.9]{degiernichols}}]
  Let $p=(h_0,h_1,\dots,h_n) \in \PP_n$.
  Then the generators $e=e_0,\; e_1,\dots,e_n=f$ have the following
  action on $v_p$: 
  \begin{itemize}
  \item Each $v_p$ is an eigenvector for the left blob generator $e_0$:
    \begin{enumerate}
    \item If $h_1=-1$ then $e_0v_p=\frac{[w_1]}{[w_1+1]}v_p$.
    \item If $h_1=1$ then $e_0v_p=0$.
    \end{enumerate}
  \item The action of $e_i$ $(1\leq i\leq n-1)$ on $v_p$ is zero if
    $p$ has positive or negative slope at position $i$,
    i.e. $|h_{i-1}-h_{i+1}|=2$. If this is not the case, then let $p'$
    be the path obtained by adding a tile to $p$ at position $i$. Then
    $e_i$ acts on the pair $\{v_p,v_{p'}\}$ in the following way:
    \begin{enumerate}
    \item If $h_{i-1}\geq0$ then
      \begin{align*}
        e_iv_p&=v_{p'}+r(w_1-h_{i-1})v_p\\
        e_iv_{p'}&=r(-w_1+h_{i-1})v_{p'}+r(-w_1+h_{i-1})r(w_1-h_{i-1})v_p.
      \end{align*}
    \item If $h_{i-1}<0$ then
      \begin{align*}
        e_iv_p&=v_{p'}+r(-w_1+h_{i-1})v_p\\
        e_iv_{p'}&=r(w_1-h_{i-1})v_{p'}+r(-w_1+h_{i-1})r(w_1-h_{i-1})v_p.
      \end{align*}
    \end{enumerate}
  \item Let $p'$ be the path obtained by adding a half tile to $p$ at
    the right boundary. Then $e_n$ acts on the pair $\{v_p,v_{p'}\}$
    in the following way:
    \begin{enumerate}
    \item If $h_{n-1}\geq0$ then
      \begin{align*}
        e_nv_p&=v_{p'}+k(w_1-h_{n-1})v_p\\
        e_nv_{p'}&=k(-w_1+h_{n-1})v_{p'}+k(-w_1+h_{n-1})k(w_1-h_{n-1})v_p.
      \end{align*}
    \item If $h_{n-1}<0$ then
      \begin{align*}
        e_nv_p&=v_{p'}+k(-w_1+h_{n-1})v_p\\
        e_nv_{p'}&=k(w_1-h_{n-1})v_{p'}+k(-w_1+h_{n-1})k(w_1-h_{n-1})v_p.
      \end{align*}
    \end{enumerate}
  \end{itemize}
\label{thm:pathaction}
\end{thm}

\section{Restricting standard modules to the blob algebra}

The (left) blob algebra $b_n$
is the subalgebra of $\bnx$ 
generated by $\{e_0, e_1, \ldots, e_{n-1}\}$ \cite{martsaleur}. 
The generators $\{ e_1, \ldots, e_{n-1}, e_n \}$  
generate another copy of $b_n$ which we will
call the \emph{right blob algebra}.

In \cite[\S 8]{gensymp} the restriction to $b_n$ is used to
determine the dimensions of the standard modules, $\WW$.  
There it is shown that each
restricted $\WW$ is filtered by standard $b_n$-modules
(as defined in \cite{martsaleur}  
--- the construction is analogous to \S\ref{sect:cell}). 
We follow the notation of \cite{blobcgm} and use $W_{\pm t}(n)$ 
for the standard $b_n$-modules.  
Recall that $W_{t}(n)$ is the standard blob module with half diagram basis
that has $n$ northern nodes and $t$ (undecorated) propagating lines.
$W_{-t}(n)$ is the standard blob module with half diagram basis that
has $n$ northern
nodes and $t-1$
undecorated
propagating lines and one decorated
propagating line.

The restriction will again be useful here. 
 Any $\bnx$-homomorphism is also a left (right)
blob homomorphism upon
restriction, and thus must respect any left (right) blob structure.

Let $d$ be a half diagram that generates some $\WW$, 
as in \S\ref{sect:cell}.  
We define
$\uri(d)$ to be the number of lines crossing the $1$-wall
(in the sense of \cite{gensymp}, 
i.e. the right wall), 
not counting any lines that
are part of non-contractible loops.  
We similarly define $\uro(d)$ as the number of lines crossing the left
wall.

When we restrict  
$\WW$ 
 to the left blob algebra then it is
filtered by $\uri$ and each section is isomorphic to a standard blob
module. A similar
situation occurs when we restrict to the right blob algebra. 
We have the following. 

\begin{prop}[\cite{gensymp}]
The $b_n$-standard content of   
$\WW$
 is as in  
Table~\ref{tab:filtration}. \qed
\end{prop}

\begin{table}
\begin{tabular}{|c||c|c||c|c|} 
\hline & left blob module & &right blob module & \\ \hline & $\ee_2=1$ & $\ee_2=-1$ & $\ee_2=1$ & $\ee_2=-1$ \\ 
\hline $\ee_1=1$ & $\xymatrix@R=10pt{W_{n}(n) \\ \vdots \\ W_{m+3}(n) \\ W_{m+1}(n)}$ & $\xymatrix@R=10pt{W_{n}(n) \\ \vdots \\ W_{m+3}(n) \\ W_{m+1}(n)}$ & $\xymatrix@R=10pt{W_{n}(n) \\ \vdots \\ W_{m+3}(n) \\ W_{m+1}(n)}$ & $\xymatrix@R=10pt{W_{-n}(n) \\ \vdots \\ W_{-(m+3)}(n) \\ W_{-(m+1)}(n)}$ \\ 
\hline $\ee_1=-1$ & $\xymatrix@R=10pt{W_{-n}(n) \\ \vdots \\ W_{-(m+3)}(n) \\ W_{-(m+1)}(n)}$ & $\xymatrix@R=10pt{W_{-n}(n) \\ \vdots \\ W_{-(m+3)}(n) \\ W_{-(m+1)}(n)}$ & $\xymatrix@R=10pt{W_{n}(n) \\ \vdots \\ W_{m+3}(n) \\ W_{m+1}(n)}$ & $\xymatrix@R=10pt{W_{-n}(n) \\ \vdots \\ W_{-(m+3)}(n) \\ W_{-(m+1)}(n)}$ \\ 
\hline
\end{tabular}
\caption{The standard content of $\W{n,m}{\ee_1,\ee_2}$ as a left 
(resp. right) blob module.}\label{tab:filtration}
\end{table}

When the left (or right) blob algebra is semi-simple,
then every standard module is simple.

\section{A necessary block condition}\label{sec:necblocks}

In this section we recall a central element  $Z_n$
(see (\ref{eq:Z_n}) below) 
of $b^x_{n}$. 
We prove Conjecture~6.5
from
\cite{degiernichols} and 
deduce the action of $Z_n$
on cell modules. 
We use this to investigate the block structure. 

\subsection{The central element $Z_n$}

We shall need a surjection from $H(\tilde{C}_n)$ to
$\bnx$ as in 
\cite[Proposition 6.3.2]{gensymp}.  
Further details can be found in \cite[\S2]{degiernichols}
(caveat: there are typos in \cite{degiernichols}; 
cf. e.g. \cite{Lusztig83}).

\begin{defn} \label{de:HeC}
  Let $q$, $Q_1$ and $Q_2$ be indeterminates. The Hecke algebra
  $H(\tilde{C}_n)$ of type $\tilde{C}_n$ over
  $\Z[q^{\pm1}$, $Q_1^{\pm_1}$, $Q_2^{\pm1}]$ is the associative
  algebra with generators $g_0$, $g_1$, $\dots$, $g_n$ and relations:
  \begin{align}
    g_ig_{i+1}g_i&=g_{i+1}g_ig_{i+1},\hspace{1em}1\leq i\leq n-2,\\
    g_0g_1g_0g_1&=g_1g_0g_1g_0,\\
    g_{n-1}g_ng_{n-1}g_n&=g_ng_{n-1}g_ng_{n-1},\\
    g_ig_j&=g_jg_i,\hspace{1em}|i-j|>1,\\
    (g_i-q)(g_i+q^{-1})&=0,\hspace{1em}1\leq i\leq n-2, \label{eq:relg} \\
    (g_0-Q_1)(g_0 + Q_1^{-1})&=0, \label{eq:relg0} \\
    (g_n-Q_2)(g_n + Q_2^{-1})&=0. \label{eq:relgn}
  \end{align}
\end{defn}

For suitable base change and 
choices of the parameters we have successive quotients:
\begin{equation}\label{eq:quotients}
 H(\tilde{C}_n) \to \twoBTL \to b_n^x \to \frac{b_{n}^x}{I_{n}(0)}
\end{equation} 
where 
$\twoBTL$ is defined in  
\cite{mgp2I}.

The  algebra $\bnx$ is defined over a ring $k$ with
parameters $\ddd =(\dd,\dl,\dr,\kl,\kr,\kk)\in k^6$.
For any three units in $k$ we can view $k$ as a
$\Z[q^{\pm1},Q_1^{\pm1},Q_2^{\pm1}]$-algebra by making $q,Q_1$ and
$Q_2$ act as these units. 
For each such triple we understand
 $H(\tilde{C}_n)$ as a $k$-algebra
by base change.

Note that we are using the  Saleur normalisation 
\cite{martsaleur} for generators. 

\begin{prop}\label{prop:surjection}
By abuse of notation let us write 
$q,Q_1,Q_2$ for the actions of these three scalars 
in $k$ defining $\HCn$ as a
$k$-algebra as described above.
If they satisfy
\begin{equation}
\delta=[2], \hspace{1em} (qQ_1-q^{-1}Q_1^{-1})\delta_L=Q_1-Q_1^{-1},
            \hspace{1em} (qQ_2-q^{-1}Q_2^{-1})\delta_R=Q_2-Q_2^{-1}.
\label{eq:pigs}
\end{equation}
then there is a surjective $k$-algebra homomorphism
  $\pi:H(\tilde{C}_n)\longrightarrow b^x_{n}$, given by
  \begin{align}
    \pi(g_i^{\pm1})&=e_i-q^{\mp1}, \label{eq:pig} \\
    \pi(g_0^{\pm1})&=Q_1^{\pm1}-\left(q^{\pm1}Q_1^{\pm1}-q^{\mp1}Q_1^{\mp1}\right)e_0,
      \label{eq:pig0} \\
    \pi(g_n^{\pm1})&=Q_2^{\pm1}-\left(q^{\pm1}Q_2^{\pm1}-q^{\mp1}Q_2^{\mp1}\right)e_n.
      \label{eq:pign}
  \end{align}
(Note that there is no dependence on $\kk $.)
\end{prop}
\proof (Outline) Consider (\ref{eq:relg}):
\[
\pi(g_i^{})\pi(g_i^{}) = (e_i - q^{-1})(e_i -q^{-1}) 
       =  \delta e_i -2q^{-1} e_i  +q^{-2}  = (q-q^{-1}) e_i +q^{-2}
\]
\[
\pi(g_i^2) \stackrel{(\ref{eq:relg})}{=} 
       \pi((q - q^{-1}) g_i +1) = (q  - q^{-1}) (e_i -q^{-1}) +1
           =   (q-q^{-1}) e_i +q^{-2}
\]
Alternatively here note that by (\ref{eq:relg}) $g_i$ has eigenvalues 
$q$ and $-q^{-1}$; while $e_i$ has eigenvalues $\delta,0$.
Then $e_i - q^{-1}$ has eigenvalues $q,-q^{-1}$, by (\ref{eq:pigs}), 
as required.
Similarly by (\ref{eq:relg0}) $g_0$ has eigenvalues $Q_1,-Q_1^{-1}$;
while $e_0$ has eigenvalues $\delta_L , 0$.
Then $Q_1 - (q^{} Q_1 -q^{-1} Q_1^{-1})e_0$ 
has eigenvalues $Q_1,-Q_1^{-1}$ as required
provided that (\ref{eq:pigs}) holds.
The verification for $g_n$ is directly analogous.
\qed

The homomorphism $\pi$ allows elements of $H(\tilde{C}_n)$ 
 to act on $\bnx$.  
In particular,  
\begin{defn}[{\cite[Definition 2.8]{degiernichols}}]
  The {\em `Jucys-Murphy elements'} for $H(\tilde{C}_n)$ are:
  \begin{align*}
    J_0&=g_1^{-1}g_2^{-1}\dots g_{n-1}^{-1}g_ng_{n-1}\dots g_2g_1g_0\\
    J_i&=g_iJ_{i-1}g_i,\hspace{1em}1\leq i\leq n-1.
  \end{align*}
\end{defn}
\begin{prop}[{\cite[Proposition 2.10]{degiernichols}}]
\label{pr:dn210}
  The Jucys-Murphy elements $J_i$ are pairwise commuting and obey the following relations:
  \begin{align*}
    [g_0,J_j]&=0,\hspace{1em}j\neq0,\\
    [g_i,J_j]&=0,\hspace{1em}1\leq i\leq n-1,\;j\neq i-1,i,\\
    [g_i,J_{i-1}J_i]&=0,\hspace{1em}1\leq i\leq n-1,\\
    [g_i,J_{i-1}+J_i]&=0,\hspace{1em}1\leq i\leq n-1,\\
    [g_0,J_0+J_0^{-1}]&=0.
  \end{align*}
  In particular, the symmetric polynomials in $J_i$, $J_i^{-1}$
  $(0\leq i\leq n-1)$ are central in $H(\tilde{C}_n)$.
\end{prop}
We hence let $Z_n$ be the central element 
\begin{equation}
  \label{eq:Z_n}
  Z_n=\sum_{i=0}^{n-1}(J_i+J_i^{-1}).
\end{equation}

\subsection{Aside on substitutions}

We can interpret  $[w_1+a]$ ($a\in\Z$) in the following way: 
\[
[w_1+a]=\frac{q^{w_1+a}-q^{-w_1-a}}{q-q^{-1}}=\frac{Q_1q^a-Q_1^{-1}q^{-a}}{q-q^{-1}}.
\]
Similarly we have
\[[w_2+a]=\frac{Q_2q^a-Q_2^{-1}q^{-a}}{q-q^-1}\hspace{1em}\text{and}\hspace{1em}[w_1+w_2+a]=\frac{Q_1Q_2q^a-Q_1^{-1}Q_2^{-1}q^{-a}}{q-q^{-1}}.
\]

\subsection{The $Z_n$-action theorem}

The following lemma is mostly a restatement of
\cite[Proposition 5.19]{degiernichols}.
However we have also included the labels of the irreducible modules.
\begin{lem}
  The generic $b_n$-module with basis
  $\{v_p \mid p \mbox{  of fixed final height } h_n \}$
  in $\W{n}{}(b)$
  is isomorphic to the generic irreducible $b_n$-module $W_{h_n}(n)$.
\end{lem}
\begin{proof}
  Note first that this is indeed a module for the blob algebra, as the
  only elements of the symplectic blob algebra that change the final
  height of a path involve the generator $e_n$, which is not present
  when we consider the restricted action.

Now by \cite[Proposition 5.19]{degiernichols} these modules are the
generic irreducibles for the blob algebra, so it suffices to show that
the labelling matches up. 

First consider the case when $n$ is even. From \cite[(3.2)]{marwood00}
we have a maximal heredity chain of idempotents
\begin{equation}
  \left(\frac{1}{\dd^{n/2}}e_1e_3\dots
    e_{n-1},\frac{1}{\dl\dd^{n/2-1}}e_0e_3e_5\dots
    e_{n-1},\dots,\frac{1}{\dd}e_{n-1},\frac{1}{\dl}e_0,1\right)
  \label{eq:heredity}
\end{equation}
corresponding to the standard modules
$W_0(n),W_{-2}(n),\dots,W_{n-2}(n),W_{-n}(n),W_n(n)$ respectively.
We must therefore show that the module with basis
  $\{v_p \mid p \mbox{  of fixed final height } h_n \}$
is associated to the
correct heredity idempotent.
Suppose first that the final height is $h_n=0$, then this module
contains the element $v_{p_0}$, where $p_0$ is the fundamental
path. By Theorem \ref{thm:pathaction} none of the idempotents in
\eqref{eq:heredity} kill $v_{p_0}$, therefore this module must be
isomorphic to $W_0(n)$.

If now the final height is $h_n>0$, then all paths must either have a
slope at at least $h_n$ points, or start with $h_1=1$ and have a slope
at at least $h_n-1$ points. Since the $e_i$ in the heredity
idempotents commute, we therefore see that any idempotent containing a
product of at least $(n-h_n)/2+1$ of the $e_i$ will kill the basis
elements obtained from these paths, but those containing $(n-h_n)/2$
will not. Therefore the first heredity idempotent that does not
annihilate this module is
\[\frac{1}{\dd^{(n-h_n)/2}}e_{h_n+1}\dots e_{n-1},\]
which corresponds to the left blob module $W_{h_n}(n)$.

If the final height is $h_n<0$, then again all paths must have a slope
at at least $|h_n|$ points, or start with $h_1=-1$ and have a slope at
at least $h_n-1$ points. In this case, any idempotent containing a
product of at least $(n+h_n)/2+1$ of the $e_i$ for $i\neq0$ will kill
the basis elements obtained from these paths, but those containing a
product of $(n+h_n)/2$ of the $e_i$ $(i\neq0)$ and $e_0$ will
not. Therefore the first heredity idempotent that does not annihilate
this module is
\[\frac{1}{\dl\dd^{(n+h_n)/2}}e_0e_{-h_n+1}\dots e_{n-1},\]
which corresponds to the left blob module $W_{h_n}(n)$.

The proof for $n$ odd is similar.
\end{proof}

We now use the path basis to determine submodules of 
$\Wb$ for specific parameter choices.

\begin{prop}[{\cite[Proposition 6.3]{degiernichols}}]
\label{prop:moduleV}
Fix $m,n \in \NN$ and $\ee_2\in\{\pm1 \}$. 
\axioma 
\begin{enumerate}\renewcommand{\labelenumi}{(\roman{enumi})}
  \item  
Choose $\theta$ so that 
   $\left[\left(-m+w_1+\ee_2w_2 \pm \theta\right)/2\right]=0$. 
Then the $\bx_n$-module   $\Wb$ has a submodule
    $V^{(n,m)}_{+,\ee_2}$  with basis 
$
\pi_n^+(m) = \{v_p   
| \; p=(h_0,h_1,\dots,h_n) \mbox{ with } h_n\geq m+1 \}.
$ 
  \item 
Choose $\theta$ so that 
   $\left[\left(-m-w_1+\ee_2w_2 \pm \theta\right)/2\right]=0$. 
Then 
    $\Wb$ has a submodule $V^{(n,m)}_{-,\ee_2}$ with basis
$
\pi_n^-(m) = \{v_p  
| \; p=(h_0,h_1,\dots,h_n) \mbox{ with } h_n\leq    -m-1 \}.
$
\end{enumerate}
\end{prop}
This statement is slightly modified from \cite{degiernichols}. The key
point is that $k(\pm(w_1 -m))$ is zero, and this is
equivalent to requiring 
  $\left[\left(- m \pm w_1\pm w_2 \pm \theta\right)/2\right]=0$ for
  appropriate signs.

\begin{thm}[{\cite[Theorem 6.4]{degiernichols}}]\label{thm:centralelement}
\axioma\ 
\axiomb\ 
Let $\theta=\thetam$.   
The action of the central element $Z_n$ as defined in \eqref{eq:Z_n}
  on $V^{(n,m)}_{\ee_1,\ee_2}$ as defined in Proposition \ref{prop:moduleV}
is given by
\[
Z_n V^{(n,m)}_{\ee_1,\ee_2}
    =  \alpha^{(n,m)}_{\ee_1,\ee_2} V^{(n,m)}_{\ee_1,\ee_2}
\]
where
$$
\displaystyle\alpha^{(n,m)}_{\ee_1,\ee_2}
   = [n] \frac{[2(-m+\ee_1w_1+\ee_2w_2)]}{[-m+\ee_1w_1+\ee_2w_2]} .
$$
\end{thm}

\begin{thm}\label{thm:isomorphism}
  Let $\theta=-m+\ee_1w_1+\ee_2w_2$.
  Then the generic $\bnx$-module 
$\W{n,m}{\ee_1,\ee_2}$ 
is isomorphic to the submodule $V^{(n,m)}_{\ee_1,\ee_2}$ of $\Wb$. 
\end{thm}
\begin{proof}
We first note that the dimensions of the modules are equal
\cite[Theorem 6.4]{degiernichols}.
We also note that 
both modules are generically irreducible $b_n^x$-modules.
Our strategy will be to compare their left blob 
    content when restricted to the left blob
    algebra and then to further distinguish using the action of the
    right blob generator, $e_n=f$.

We know that the modules $\W{n,m}{\ee_1,\ee_2}$ and
$V^{(n,m)}_{\ee_1,\ee_2}$ are both generically irreducible for $\bnx$.
We also know that upon restriction to the left blob subalgebra of
$\bnx$ that they have the same irreducible content as left blob
modules.
Thus we can say that $V^{(n,m)}_{\ee_1,\ee_2}$
is either $\W{n,m}{\ee_1,\ee_2}$ or 
$\W{n,m}{\ee_1,-\ee_2}$. I.e. the left blob structure doesn't
distinguish between $\pm \ee_2$.  We now consider the action of the
right blob generator $e_n = f$.

We work out the trace of the action of $f$ on these two modules which
is an easy calculation.
As $f$ has a monic action, (i.e. maps a basis element to another
element),  we need only write down those elements which map to the same
basis element times a scalar. (We only need the diagonal entries in
the matrix representing the action of $f$.)
If $\ee_2=1$ then the trace is 
$
\delta_R$  times {number of basis elements
  with arcs with
  right blob ending on node n}.
If 
$\ee_2 = -1$ then the trace of $f$ is
$\delta_R$ times {number of  basis elements with arcs with
  right blob ending on node n} plus 
 $\delta_R$ times {number of basis elements with propagating lines 
starting at node n and decorated with a right blob}.

Now note 
there is a set bjection on the basis of the module
$\W{n,m}{\ee_1,+}$ to the module
$\W{n,m}{\ee_1,-}$ given by putting a right blob on the right most
propagating line. So these modules have the same dimension. 
Moreover,  while the first term in the two cases of the trace is clearly equal for either value of
$\ee_2$, 
the second term
in the second sum is non-zero and 
thus these sums are clearly different from each other.
(In principle, this is a number that could be made
  explicit by combinatorics. This won't be needed though.)
This means that the action of $f$ is enough to determine the sign
$\ee_2$.

Now consider the module $V^{(n,m)}_{\ee_1,\ee_2}$. 
Here we need to consider the paths in pairs and use the action defined
before.

  Let $p'$ be the path obtained by adding a half tile to $p$ at the
  right boundary. The path $p'$ is ``further away'' from the fundamental path
  than $p$. In particular, the absolute value of the last entry in
  $p'$ is bigger than $p$.
  Then $e_n$ acts on the pair $\{v_p,v_{p'}\}$ in the
  following way:
    \begin{enumerate}
    \item If $h_{n-1}\geq0$ then
      \begin{align*}
        e_nv_p&=v_{p'}+k(w_1-h_{n-1})v_p\\
        e_nv_{p'}&=k(-w_1+h_{n-1})v_{p'}+k(-w_1+h_{n-1})k(w_1-h_{n-1})v_p.
      \end{align*}
    \item If $h_{n-1}<0$ then
      \begin{align*}
        e_nv_p&=v_{p'}+k(-w_1+h_{n-1})v_p\\
        e_nv_{p'}&=k(w_1-h_{n-1})v_{p'}+k(-w_1+h_{n-1})k(w_1-h_{n-1})v_p.
      \end{align*}
    \end{enumerate}
So the trace of this action on the 2 by 2 matrix given by
$\{v_p,v_{p'}\}$ is in both cases
$$
        k(w_1-h_{n-1}) + k(-w_1+h_{n-1})
        $$

Using the formula for $k$ we get:
\begin{multline*}
k(w_1-h_{n-1})
+
k(-w_1+h_{n-1})\\
=-\frac{[(w_1-h_{n-1}-w_2+\theta)/2][(w_1-h_{n-1}-w_2-\theta)/2]}{[w_1-h_{n-1}][w_2+1]}\\
- \frac{[(-w_1+h_{n-1}-w_2+\theta)/2][(-w_1+h_{n-1}-w_2-\theta)/2]}{[-w_1+h_{n-1}][w_2+1]}
\end{multline*}
with $\theta$ such that
$[-m+\ee_1 w_1 +\ee_2 w_2 \pm \theta] =0$.

set $h:= h_{n-1}$ 
 and  let $\theta=-m+\ee_1w_1+\ee_2w_2$.
Then $k(w_1-h) + k(-w_1+h)$
\begin{multline*}
=-\frac{[(1+\ee_1)w_1-h-(1-\ee_2)w_2-m)/2][((1-\ee_1)w_1-h-(1+\ee_2)w_2+m)/2]}{[w_1-h][w_2+1]}\\
+\frac{[(-(1-\ee_1)w_1+h-(1-\ee_2)w_2-m)/2][(-(1+\ee_1)w_1+h-(1+\ee_2)w_2+m)/2]}{[w_1-h][w_2+1]}
\end{multline*}
Set $u= \frac{h+m}{2}$ and $v= \frac{h-m}{2}$ so $h = u+v$ and $m=u-v$.
Also set
${\alpha_1} = \frac{1+\ee_1}2$,
${\alpha_2} = \frac{1+\ee_2}2$,
$\bar{\alpha}_1 = \frac{1-\ee_1}2$, and
$\bar{\alpha}_2 = \frac{1-\ee_2}2$.
Note that these numbers are all $0$ or $1$,
$\alpha_i+\bar{\alpha}_i = 1$, and
$\alpha_i-\bar{\alpha}_i = \ee_i$.
$$
=-\frac{[\alpha_1 w_1-u-\bar{\alpha}_2 w_2][\bar{\alpha}_1w_1-v-\alpha_2 w_2]
+[\bar{\alpha}_1 w_1-v+\bar{\alpha}_2 w_2][\alpha_1 w_1-u+\alpha_2 w_2]}{[w_1-h][w_2+1]}
$$

We now expand the quantum integer products.
\begin{multline*}
  [\alpha_1 w_1-u-\bar{\alpha}_2 w_2][\bar{\alpha}_1w_1-v-\alpha_2
  w_2]\\
=[ w_1- w_2-h-1 ]
+ [w_1-w_2-h-3] + \cdots + [\ee_1 w_1  +\ee_2 w_2  -m +3  ] +
[\ee_1w_1 +\ee_2 w_2  -m +1]
\end{multline*}
\begin{multline*}
[\bar{\alpha}_1 w_1-v+\bar{\alpha}_2 w_2][\alpha_1 w_1-u+\alpha_2w_2]\\
=[ w_1+ w_2-h -1 ]
+ [w_1+w_2-h-3] + \cdots + [\ee_1 w_1  +\ee_2 w_2  -m +3  ] +
[\ee_1w_1 +\ee_2 w_2  -m+1 ]
\end{multline*}

The numerator is 
\begin{multline*}
[ w_1+ w_2-h -1 ]
+ [w_1+w_2-h-3] + \cdots + [ w_1- w_2-h +3] + [ w_1- w_2-h +1] \\
= [w_1-h][w_2]
\end{multline*}

Thus we get:
$$
k(w_1-h) + k(-w_1+h)
= \frac{[w_1-h][w_2]}{[w_1-h][w_2+1]}
= \frac{[w_2]}{[w_2+1]}
= \dr
$$
if both elements of the pair
  $\{v_p, v_{p'}\}$ are in $V^{n,m}_{\ee_1,\ee_2}$.

  But it may be that the $v_p$'s don't always occur in pairs like this
  inside $V^{n,m}_{\ee_1,\ee_2}$.

Now  we always get a pair $\{v_{p}, v_{p'}\}$ in
$V^{n,m}_{\ee_1,\ee_2}$
if the last entry of $p$ is at least $|m|$. (Since the last entry of
$p'$ has absolute value $|m| +2$.)

Let's consider the case where $v_{p'}$ is in $V^{n,m}_{\ee_1,\ee_2}$
but $v_p$ is not.

In this case: the second last entry of $p$ and $p'$ is $-m+1$ ($\ee_1 = -1$)
or $m-1$ ($\ee_1=1$).

So we check,
noting that: 

 {
    $k(-w_1+m) = 0$  
    if $\pm \theta = w_1+w_2 -m$ ($\ee_1=\ee_2=1$).
   }

 {
    $k(-w_1+m) = \dr$ and
    $k(w_1-m) = 0$ 
    if $\pm \theta = w_1-w_2 -m$ ($\ee_1=1$ and $\ee_2=-1$).
    }

 {
    $k(w_1+m) = 0$  
    if $\pm \theta =-w_1+w_2 -m$ ($\ee_1=-1$ and $\ee_2=1$).
    }

 {
    $k(w_1+m) = \dr$ and
    $k(-w_1-m) = 0$ 
    if $\pm \theta = -w_1-w_2 -m$ ($\ee_1=-1$ and $\ee_2=-1$).
    }

\newcommand{\VVV}{V^{n,m}_{\ee_1,\ee_2}}
    
 Thus, putting $p^+$ for the $p'$ with last entry $+|m|$ and 
 putting $p^-$ for the $p'$ with last entry $-|m|$.
 $$
f v_{p^+} = k(-w_1 + m) v_{p^+} + k(-w_1+m)k(w_1-m) v_{p}
 $$
 $$
f v_{p^-} = k( w_1 + m) v_{p^-} + k(-w_1-m)k(w_1+m) v_{p}
 $$
    
For $v_{p^+}$ we see that the coefficient of $v_p$ is zero for either
value of $\ee_2$ as it
needs to be for $f v_{p^+}\in \VVV$. 
The trace is $k(-w_1+m)$.

For $v_{p^-}$ we see that the coefficient of $v_p$ is zero
for either value of $\ee_2$ as it
needs to be for $f v_{p^-}\in \VVV$. 
The trace is $k(w_1+m)$.

Thus in both cases, the trace is one-dimensional with value $0$ if $\ee_2=1$ and $\dr$
if $\ee_2=-1$.

Thus we see at that the traces of $f$ are different on these modules
and allow us to distinguish as claimed. 
\end{proof}

\begin{cor}
Fix $\ddd$ and hence $\bnx$.
  Whenever $\alpha^{(n,m)}_{\ee_1,\ee_2}\in k$,
  the central element $Z_n$ acts by
  $\alpha^{(n,m)}_{\ee_1,\ee_2}$ on $\W{n,m}{\ee_1,\ee_2}$.
\end{cor}
\begin{proof}
  Use \cite[Theorem 6.4]{degiernichols} and the isomorphism from
  Theorem~\ref{thm:isomorphism}
  for the generic case. The other cases follow by analytic
  continuity
  --- $Z_n$ acts by some scalar, so the limit of generic actions
  approaching $\ddd$ exists and is this scalar. 
\end{proof}

An important and immediate consequence is the following corollary.

\begin{cor} \label{th:cellblockH}
Fix $\ddd$ and hence $\bnx$.
  A necessary condition for two cell modules $\W{n,m}{\ee_1,\ee_2}$
  and $\W{n,t}{\eta_1,\eta_2}$ to be in the same block is that
  $\alpha^{(n,m)}_{\ee_1,\ee_2}=\alpha^{(n,t)}_{\eta_1,\eta_2}$. 
\end{cor}

\section{Gram determinants for cell modules}\label{sec:gram}

We recall from \cite[\S8.2]{gensymp} that each cell module
$\WW$ has a basis $\WWB$ of half diagrams with arcs on the northern edge 
(hereafter referred to as the standard diagram basis). 

\begin{eg}\label{eg:gram1}
The cell module $\W{5,2}{-,-}$  
has the following basis of half diagrams:
 \[ \begin{tikzpicture}[scale=0.4]
    \draw (0,3) to[out=-90, in=-90] node[pos=0.5]{$\bullet$} +(1,0);
    \draw (2,3) to node[pos=0.5]{$\bullet$} +(0,-2);
    \draw (3,3) to +(0,-2);
    \draw (4,3) to node[pos=0.5,white]{$\bullet$} node[pos=0.5]{$\circ$} +(0,-2);
  \end{tikzpicture},\hspace{1em}
 \begin{tikzpicture}[scale=0.4]
    \draw (0,3) to[out=-90, in=-90] +(1,0);
    \draw (2,3) to node[pos=0.5]{$\bullet$} +(0,-2);
    \draw (3,3) to +(0,-2);
    \draw (4,3) to node[pos=0.5,white]{$\bullet$} node[pos=0.5]{$\circ$} +(0,-2);
  \end{tikzpicture},\hspace{1em}
 \begin{tikzpicture}[scale=0.4]
    \draw (1,3) to[out=-90, in=-90] +(1,0);
    \draw (0,3) to[out=-90, in=90] node[pos=0.5]{$\bullet$} +(2,-2);
    \draw (3,3) to +(0,-2);
    \draw (4,3) to node[pos=0.5,white]{$\bullet$} node[pos=0.5]{$\circ$} +(0,-2);
  \end{tikzpicture},\hspace{1em}
 \begin{tikzpicture}[scale=0.4]
    \draw (2,3) to[out=-90, in=-90] +(1,0);
    \draw (0,3) to[out=-90, in=90] node[pos=0.5]{$\bullet$} +(2,-2);
    \draw (1,3) to[out=-90, in=90] +(2,-2);
    \draw (4,3) to node[pos=0.5,white]{$\bullet$} node[pos=0.5]{$\circ$} +(0,-2);
  \end{tikzpicture},\hspace{1em}
  \begin{tikzpicture}[scale=0.4]
    \draw (3,3) to[out=-90, in=-90] +(1,0);
    \draw (0,3) to[out=-90, in=90] node[pos=0.5]{$\bullet$} +(2,-2);
    \draw (1,3) to[out=-90, in=90] +(2,-2);
    \draw (2,3) to[out=-90, in=90] node[pos=0.5,white]{$\bullet$} node[pos=0.5]{$\circ$} +(2,-2);
  \end{tikzpicture},\hspace{1em}
\begin{tikzpicture}[scale=0.4]
    \draw (3,3) to[out=-90, in=-90] node[white,pos=0.5]{$\bullet$}node[pos=0.5]{$\circ$} +(1,0);
    \draw (0,3) to[out=-90, in=90] node[pos=0.5]{$\bullet$} +(2,-2);
    \draw (1,3) to[out=-90, in=90] +(2,-2);
    \draw (2,3) to[out=-90, in=90] node[pos=0.5,white]{$\bullet$} node[pos=0.5]{$\circ$} +(2,-2);
  \end{tikzpicture}.
\]
\end{eg}

Consider  $u$, $v\in \WWB$. 
  We define a scalar $\langle u,v \rangle$  
as follows.
We first flip $v$ vertically and identify the southern nodes
of this diagram with the respective northern nodes of $u$. 
After applying the  straightening rules for $\bnx$, 
we obtain a diagram with a number of (possibly
  decorated) strings. The value of $\langle u,v \rangle$  is the
  coefficient of this diagram if the strings match the number and
  decorations needed for the cell module, and is zero otherwise. For
  instance 
  \[
  \begin{tikzpicture}[scale=0.4]
    \draw (0,3) to[out=-90, in=-90] node[pos=0.5]{$\bullet$} +(1,0);
    \draw (2,3) to node[pos=0.5]{$\bullet$} +(0,-2);
    \draw (3,3) to +(0,-2);
    \draw (4,3) to node[pos=0.5,white]{$\bullet$} node[pos=0.5]{$\circ$} +(0,-2);
    \draw (0,3) to[out=90, in=-90] node[pos=0.5]{$\bullet$} +(2,2);
    \draw (1,3) to[out=90, in=90] +(1,0);
    \draw (3,3) to +(0,2);
    \draw (4,3) to node[pos=0.5,white]{$\bullet$}node[pos=0.5]{$\circ$} +(0,2);
    \draw (6,3) node {$=\dl^2\dr$};
    \draw (8,2) to node[pos=0.5]{$\bullet$} +(0,2);
    \draw (9,2) to +(0,2);
    \draw (10,2) to node[white,pos=0.5]{$\bullet$}node[pos=0.5]{$\circ$} +(0,2);
    \draw (13,3) node {and};
    \draw (16,3) to[out=-90, in=-90] node[pos=0.5]{$\bullet$} +(1,0);
    \draw (18,3) to[out=-90, in=90] node[pos=0.5]{$\bullet$} +(0,-2);
    \draw (19,3) to +(0,-2);
    \draw (20,3) to node[pos=0.5,white]{$\bullet$} node[pos=0.5]{$\circ$} +(0,-2);
    \draw (16,3) to[out=90, in=-90] node[pos=0.5]{$\bullet$} +(2,2);
    \draw (17,3) to[out=90, in=-90] +(2,2);
    \draw (19,3) to[out=90, in=90] +(1,0);
    \draw (18,3) to[out=90, in=-90] node[pos=0.5,white]{$\bullet$}node[pos=0.5]{$\circ$} +(2,2);
    \draw (21.5,3) node {$=\dl$};
    \draw (23,4) to[out=-90, in=-90] node[pos=0.5]{$\bullet$} +(1,0);
    \draw (25,4) to[out=-90, in=90] node[pos=0.7]{$\bullet$}node[pos=0.3,white]{$\bullet$}node[pos=0.3]{$\circ$} +(-2,-2);
    \draw (24,2) to[out=90, in=90] node[pos=0.5,white]{$\bullet$}node[pos=0.5]{$\circ$} +(1,0);
  \end{tikzpicture},\]
  giving  
$\langle u,v \rangle = \dl^2\dr$ and $=0$ respectively. 
(This is  $\langle u,v \rangle$  as
defined in \cite[\S6]{saleurcombinatorial} with the
``semimeander'' convention.) 

\begin{prop}
  Let $\sigma:b^x_n\rightarrow b^x_n$ be the involution defined by
  flipping diagrams vertically. 
Then the inner product defined by 
  $\langle-,-\rangle$, together with $\sigma$, defines a contravariant
  bilinear form on $\WW$. 
That is,  
for $d\in b^x_n$, we have 
  \[\langle du,v\rangle=\langle u,\sigma(d)v\rangle.\]
\end{prop}
\begin{proof}
  This follows
from the definition of $\langle u,v \rangle$  
and the action of the algebra on half diagrams. 
\end{proof}

While this form can be used over the integral ring
$\Z[q^{\pm},Q_1^{\pm},Q_2^{\pm}]$,
we will need to specialise in order
to use results from the parameterisation of De~Gier--Nichols.

We  define the Gram matrix $\Gram{n,m}{\ee_1,\ee_2}$ to be the
matrix of entries $(\langle u,v\rangle )_{u,v}$, where $u,v$ runs over
the basis $\WWB$ of $\W{n,m}{\ee_1,\ee_2}$. 
We also define the Gram determinant
\begin{equation} \label{eq:defdetgamma}
\WWGam \; = \; \det{\Gram{n,m}{\ee_1,\ee_2}} .
\end{equation} 
When we base change to a field, the rank of  $\Gram{n,m}{\ee_1,\ee_2}$
is also the rank of a corresponding map from the module to its
contravariant dual. 
The module is thus simple if and only if the matrix
is non-singular.

\begin{eg}
With the  ordering of basis elements as in example \ref{eg:gram1}, 
the Gram matrix is therefore
  \[\Gram{5,2}{-,-}=\left(\begin{array}{cccccc}
    \dl^2\dr\kl&\dl\dr\kl&\dl^2\dr&0&0&0\\
    \dl\dr\kl&\dl\dr\dd&\dl\dr&0&0&0\\
    \dl^2\dr&\dl\dr&\dl\dr\dd&\dl\dr&0&0\\
    0&0&\dl\dr&\dl\dr\dd&\dl\dr&\dl\dr^2\\
    0&0&0&\dl\dr&\dl\dr\dd&\dl\dr\kr\\
    0&0&0&\dl\dr^2&\dl\dr\kr&\dl\dr^2\kr
    \end{array}\right).\]
  We wish to calculate the determinant of this matrix. 
By Laplace expansion, we obtain 
  \[(\dl\dr)^{-6}\Gramdet{5,2}{-,-}=\dl\kl\left|\begin{smallmatrix}
      \dd&1&0&0&0\\
      1&\dd&1&0&0\\
      0&1&\dd&1&\dr\\
      0&0&1&\dd&\kr\\
      0&0&\dr&\kr&\dr\kr
      \end{smallmatrix}\right|-\kl^2\left|\begin{smallmatrix}
        \dd&1&0&0\\
        1&\dd&1&\dr\\
        0&1&\dd&\kr\\
        0&\dr&\kr&\dr\kr
        \end{smallmatrix}\right|
        +(2\dl\kl-\dl^2\dd)\left|\begin{smallmatrix}
            \dd&1&\dr\\
            1&\dd&\kr\\
            \dr&\kr&\dr\kr
            \end{smallmatrix}\right|.
\]
          Laplace expanding the first of these determinants results in the following:
          \begin{equation}\label{eq:gramdetexpand}\dr\kr\left|\begin{smallmatrix}
            \dd&1&0&0\\
            1&\dd&1&0\\
            0&1&\dd&1\\
            0&0&1&\dd
          \end{smallmatrix}\right|-\kr^2\left|\begin{smallmatrix}
            \dd&1&0\\
            1&\dd&1\\
            0&1&\dd
            \end{smallmatrix}\right|+(2\dr\kr-\dd\dr^2)\left|\begin{smallmatrix}
              \dd&1\\
              1&\dd
              \end{smallmatrix}\right|.\end{equation}
            Since $\delta=[2]$, we can use the identities for quantum
            integers to show that each of these determinants is equal
            to $[n+1]$, where $n$ is the size of the matrix. Using our
            parameterisation for $\dl,\dr,\kl,\kr$, we then see that
            \eqref{eq:gramdetexpand} is equal to
            \begin{align*}
              [w_2+1]^{-2}&\left([w_2][w_2+1][5]-[w_2+1]^2[4]+(2[w_2][w_2+1]-[2][w_2]^2)[3]\right)\\
              &=[w_2+1]^{-2}\left([w_2][w_2+1][5]-[w_2+1]^2[4]+[w_2]([w_2+1]-[w_2-1])[3]\right)\\
              &=[w_2+1]^{-2}\left([w_2]([w_2+1][5]-[w_2-1][3])-[w_2+1]([w_2+1][4]-[w_2][3])\right)\\
              &=[w_2+1]^{-2}\left([w_2][2][w_2+4]-[w_2+1][w_2+4]\right)\\
              &=[w_2+1]^{-2}[w_2+4][w_2-1].
              \end{align*}
              Expanding the other matrices in the same way, we see that 
              \begin{align*}
                \left(\frac{[w_1][w_2]}{[w_1+1][w_2+1]}\right)^{-6}&[w_1+1]^2[w_2+1]^2G^{(5,2)}_{-,-}\\
                &=[w_1][w_1+1][w_2+4][w_2-1]-[w_1+1]^2[w_2+3][w_2-1]\\&\hspace{5em}+\left(2[w_1][w_1+1]-[w_1]^2[2]\right)[w_2+2][w_2-1]\\
                &=[w_2-1]\Big([w_1][w_1+1][w_2+4]-[w_1+1]^2[w_2+3]\\&\hspace{5em}+[w_1]([w_1+1]-[w_1-1])[w_2+2]\Big).
              \end{align*}
Note that the last expression in the brackets takes the same form as
the second line in our evaluation of \eqref{eq:gramdetexpand}
above. Hence we finally arrive at 
\[\Gramdet{5,2}{-,-}=\frac{[w_1]^6[w_2]^6}{[w_1+1]^{8}[w_2+1]^{8}}[w_1-1][w_2-1][w_1+w_2+3].\]
\end{eg}

As demonstrated by this example, calculating $\WWGam$ 
is non-trivial.
However 
we can apply  results of \cite{degiernichols} to calculate it
with respect to the path basis, which we will
see is easier.

\begin{prop}[{\cite[Proposition 5.13]{degiernichols}}]
  In the path basis, 
$\Gramcon$
is diagonal. 
\end{prop}

\begin{defn}[{\cite[Definition 5.14]{degiernichols}}]
  We define the functions $f(h)$ and $g(h)$ to be:
  \begin{align*}
    f(h)&=r(w_1-h)r(-w_1+h)\\
    g(h)&=k(w_1-h)k(-w_1+h)
  \end{align*}
  where $r(u)$ and $k(u)$ are as in \eqref{eq:r(u)k(u)}.
\end{defn}
\begin{prop}[{\cite[Proposition 5.15]{degiernichols}}]\label{prop:patheigenvalues}
  The eigenvalue $\lambda_p$ of the Gram matrix,  $\Gramcon$,
  for each path $p$ is
  given by the following recursive procedure. Let $p_0$ be the
  fundamental path, and let $p'$ be a path obtained from another path
  $p$ by the addition of a tile (or half tile) at point $i$. The
  following hold: 
  \begin{itemize}
  \item $\lambda_{p_0}=1$.
  \item If $p'$ and $p'$ differ by a full tile we have $\lambda_{p'}=f(h_{i-1})\lambda_p$.
  \item If $p'$ and $p'$ differ by a half tile we have $\lambda_{p'}=g(h_{n-1})\lambda_p$.
  \end{itemize}
\end{prop}

Thus to find 
$\WWGam$ 
we  take the
product of the eigenvalues corresponding to the paths that form a
basis of that cell module (having chosen $\theta$ appropriately). 

To illustrate, we 
return to Example \ref{eg:gram1} above, and recalculate the Gram
matrix with respect to the path basis. 

\begin{eg}
The basis of $\W{5,2}{-,-}$ here consists of all paths of a final height $-3$ or lower, and our value of $\theta$ is $-2-w_1-w_2$. These paths are given below, along with the tiles that are needed to construct them.
\[\begin{tikzpicture}[scale=0.4]
  \draw (0,0)--++(1,-1)--++(1,1)--++(1,-1)--++(1,-1)--++(1,-1);
\end{tikzpicture}\hspace{1em}
\begin{tikzpicture}[scale=0.4]
  \draw (0,0)--++(1,1)--++(1,-1)--++(1,-1)--++(1,-1)--++(1,-1);
  \draw[dashed] (0,0)--++(1,-1) node[yshift=12pt]{$0$} --++(1,1);
\end{tikzpicture}\hspace{1em}
\begin{tikzpicture}[scale=0.4]
  \draw (0,0)--++(1,-1)--++(1,-1)--++(1,1)--++(1,-1)--++(1,-1);
  \draw[dashed] (1,-1)--++(1,1) node[yshift=-12pt]{$-1$} --++(1,-1);
\end{tikzpicture}\hspace{1em}
\begin{tikzpicture}[scale=0.4]
  \draw (0,0)--++(1,-1)--++(1,-1)--++(1,-1)--++(1,+1)--++(1,-1);
  \draw[dashed] (1,-1)--++(1,1) node[yshift=-12pt]{$-1$} --++(1,-1);
  \draw[dashed] (2,-2)--++(1,1) node[yshift=-12pt]{$-2$} --++(1,-1);
\end{tikzpicture}\hspace{1em}
\begin{tikzpicture}[scale=0.4]
  \draw (0,0)--++(1,-1)--++(1,-1)--++(1,-1)--++(1,-1)--++(1,1);
  \draw[dashed] (1,-1)--++(1,1) node[yshift=-12pt]{$-1$} --++(1,-1);
  \draw[dashed] (2,-2)--++(1,1) node[yshift=-12pt]{$-2$} --++(1,-1);
  \draw[dashed] (3,-3)--++(1,1) node[yshift=-12pt]{$-3$} --++(1,-1);
\end{tikzpicture}\hspace{1em}
\begin{tikzpicture}[scale=0.4]
  \draw (0,0)--++(1,-1)--++(1,-1)--++(1,-1)--++(1,-1)--++(1,-1);
  \draw[dashed] (1,-1)--++(1,1) node[yshift=-12pt]{$-1$} --++(1,-1);
  \draw[dashed] (2,-2)--++(1,1) node[yshift=-12pt]{$-2$} --++(1,-1);
  \draw[dashed] (3,-3)--++(1,1) node[yshift=-12pt]{$-3$} --++(1,-1);
  \draw[dashed] (4,-4)--++(1,1) node[yshift=-12pt]{$-4$} --++(0,-2);
\end{tikzpicture}\hspace{1em}\]
The eigenvalues of the Gram matrix for these paths
are \[1,f(0),f(-1),f(-1)f(-2),f(-1)f(-2)f(-3),f(-1)f(-2)f(-3)g(-4)\]
respectively. The Gram determinant is the product of these, which we
evaluate to be
\[\Gramdet{5,2}{-,-}=\frac{[w_1]^2}{[w_1+1]^4[w_2+1]^2}[w_1-1][w_2-1][w_1+w_2+3].\]
Note that this is the same result as Example \ref{eg:gram1}, up to rescaling by a power of the parameters.
\end{eg}

Note this is  easier than the calculation in Example \ref{eg:gram1}. 
However in general $\WWGam$ may 
still be difficult to calculate, due to the large number of paths and
tiles as the cell modules increase in size. 
We next appeal to
results about changing bases and the effect on the Gram determinant. 
\begin{thm}\label{thm:gram}
  With respect to the path basis above, the Gram determinant of 
$\W{n,m}{\ee_1,\ee_2}$ is 
\begin{multline*}
\Gramdet{n,m}{\ee_1,\ee_2}=(\dl^{\frac{1}{2}(1-\ee_1)}\dr^{\frac{1}{2}(1-\ee_2)})^{\dim
  \W{n,m}{\ee_1,\ee_2}}\prod_{k=0}^{\frac{1}{2}(n-m-3)}\Big([\half(n-m-2k-1)]\\
\times[\ee_1w_1-\half(-n+m+2k+1)][\ee_2w_2-\half(-n+m+2k+1)]\\
\times[\ee_1w_1+\ee_2w_2-\half(n+m-2k-1)][w_1+1]^{-2}[w_2+1]^{-2}\Big)^{\dim\W{n,n-1-2k}{\ee_1,\ee_2}}.
\end{multline*}
\end{thm}
\begin{proof}
  From the definitions of $f(h)$ and $g(h)$, we see that
  $\Gramdet{n,m}{\ee_1,\ee_2}$ is a product of box numbers of the
  form $[a_1w_1+a_2w_2-b]^c$, where $a_i\in\{-1,0,1\}$ and
  $b,c\in\Z$. Note that all such terms with either $a_1=0$ or
  $a_1,a_2\neq0$ arise from the contributions of some $g(h)$ at the
  right boundary. Moreover, apart from $[w_2+1]$ they all appear to a
  positive power. Therefore when calculating the product over all
  permitted paths, there can be no cancellation of these terms. To
  determine the power of $[w_2+1]$, we multiply the above product by
  $[w_2+1]^{-2}$ for every factor $g(h)$. The power to which $g(h)$
  appears in the product is the number of paths of final height $h'$
  for $|h'|\geq|h|$, which in turn is the dimension of the cell module
  defined by paths of such height. Therefore we see that 
\begin{multline*}
  \mu\prod_{k=0}^{\frac{1}{2}(n-m-3)}\Big([\half(n-m-2k-1)][\ee_2w_2-\half(-n+m+2k+1)]\\
\times[\ee_1w_1+\ee_2w_2-\half(n+m-2k-1)][w_2+1]^{-2}\Big)^{\dim \W{n,n-1-2k}{\ee_1,\ee_2}}
\end{multline*}
is a factor of $\Gramdet{n,m}{\ee_1,\ee_2}$, where $\mu$ is a
product of box numbers of the form $[w_1-a]$ for $a\in\Z$.

In order to determine the other factors, we will change basis and
recalculate the Gram determinant. First, note from the proof of
Theorem \ref{thm:isomorphism} that the change of basis matrix between
the standard and path bases is upper triangular, with diagonal entries
equal to powers of the parameters for the symplectic blob
algebra. Note also that these diagonal entries do not contain the
parameter $\dd$. Indeed, the relations that could result in a factor
of $\dd$ must be the standard Temperley-Lieb relations, i.e.
\[e_ie_{i\pm1}e_i=e_i\text{~~~~~and~~~~~}e_i^2=\dd e_i,\]
but these cannot appear in the leading term of the path basis as we
can never add tiles in position $i$, followed by $i\pm1$, then in $i$
again, nor can we add tiles in position $i$ twice in a row. We also
cannot obtain a $\dd$ by adding to the initial diagram $d_{m+1}$ (or
$d_{-m-1},d'_{m+1},d'_{-m-1}$). 

From the standard diagram basis, we change to an alternative path
basis, which we obtain by replacing $e_i$ by $e_{n-i}$, $w_1$ by $w_2$
and $\ee_1$ by $\ee_2$ in the above. In other words, we are working
with the path basis defined by the right blob as opposed to the
left. For the same reasons as in Theorem \ref{thm:isomorphism}, the
change of basis matrix is again upper triangular. Therefore the change
of basis matrix between the first and second path bases is upper
triangular, and has determinant equal to a product of powers of the
parameters (except $\dd$, as before). Moreover, by considering the
contribution at the half tile boundary in the second path basis, we
see that 
\begin{multline*}
  \mu'\prod_{k=0}^{\frac{1}{2}(n-m-3)}\Big([\half(n-m-2k-1)][\ee_1w_1-\half(-n+m+2k+1)]\\
\times[\ee_1w_1+\ee_2w_2-\half(n+m-2k-1)][w_1+1]^{-2}\Big)^{\dim \W{n,n-1-2k}{\ee_1,\ee_2}}
\end{multline*}
is a factor of $\Gramdet{n,m}{\ee_1,\ee_2}$, where $\mu'$ is a
product of powers of the parameters and box numbers of the form
$[w_2-a]$ for $a\in\Z$. When we combine these two results we have
\begin{multline*}
\Gramdet{n,m}{\ee_1,\ee_2}=
\lambda\prod_{k=0}^{\frac{1}{2}(n-m-3)}\Big([\half(n-m-2k-1)][\ee_1w_1-\half(-n+m+2k+1)]\\
\times[\ee_2w_2-\half(-n+m+2k+1)][\ee_1w_1+\ee_2w_2-\half(n+m-2k-1)]
[w_1+1]^{-2}[w_2+1]^{-2}\Big)^{\dim \W{n,n-1-2k}{\ee_1,\ee_2}},
\end{multline*}
where $\lambda$ is a product of powers of the parameters $\dl,\dr$
(since our parameterisation has $\kl=\kr=1$). To determine $\lambda$,
we return to the Gram matrix of the standard diagram basis and
determine the highest powers of $\dl$ and $\dr$ which divide the
determinant. Since propagating lines cannot cross, any non-zero entry
in the Gram matrix must have a factor of $\dl$ (resp. $\dr$) if there
is a left (resp. right) blob on propagating lines. Therefore we can
extract a factor of
$(\dl^{\frac{1}{2}(1-\ee_1)}\dr^{\frac{1}{2}(1-\ee_2)})^{\dim\W{n,m}{\ee_1,\ee_2}}$
from the matrix. In fact, this is the largest power of $\dl$ and $\dr$
we can extract from any row. We deal with factors that may arise from
further left blobs, those arising from the right follow by
symmetry. Suppose a diagram has a horizontal arc with a left
blob. Then this must be the outermost arc of a left-exposed nest of
arcs. We construct a diagram which, when taking the inner product with
the first, does not add any factors of $\dl$ with this nested set of
arcs. Indeed, we simply place undecorated arcs in the leftmost side of
the diagram so that the blobbed arc forms a closed loop after taking
the inner product. This results in a factor of $\kl=1$ appearing, and
no $\dl$. 

Finally, the range of values over which we take the product ensure
that neither $[w_1]$ nor $[w_2]$ can appear. Therefore $\lambda$ must
be the greatest factor of $\dl$ and $\dr$, and the result follows.
\end{proof}
This final example returns to the cell module $\W{5,2}{-,-}$.
\begin{eg}
For $n=5$, $m=2$, we have $\half(n-m-3)=0$. Therefore by Theorem \ref{thm:gram}, we have
\begin{align*}
  \Gramdet{5,2}{-,-}&=([w_1][w_1+1]^{-1}[w_2][w_2+1]^{-1})^6[w_1+1]^{-2}[w_2+1]^{-2}[1][-w_1+1][-w_2+1][-w_1-w_2-3]\\
  &=[w_1]^6[w_1+1]^{-8}[w_2]^6[w_2+1]^{-8}[w_1-1][w_2-1][w_1+w_2+3].
\end{align*}
We can compare this with Example \ref{eg:gram1} to see that we indeed have the Gram determinant.
\end{eg}

\section{Homological tools for decomposition matrices and blocks of $\bnp$}
\label{sec:hom}

In this section, we will use the constants
$\alpha^{(n,m)}_{\ee_1,\ee_2}$ and homomorphisms from \cite{mgp3} to
determine the block structure of $b_n^x$ for $\ddd\in \C^6$.
Fixing $\ddd$ is done by choosing
values for  $q, w_1$ and $w_2$.
We will here restrict those
values such that none of $[w_1],[w_2],[w_1+1]$ or $[w_2+1]$ are zero.
We now change our parameterisation
to that of GMP2 in \S\ref{sect:notn} above, in order to
use
the results of \cite{mgp3}. This is achieved by
rescaling generators in the following way: 
\begin{align*}
  e_0&\mapsto-[w_1+1]e_0,\\
  e_i&\mapsto-e_i\text{ for }1\leq i\leq n-1 \\
  e_n&\mapsto-[w_2+1]e_n.
\end{align*}
\subsection{Globalisation functors}

We will also use the globalisation functors to work in a
``large $n$
limit'' symplectic blob algebra where both parameters are
positive. Having determined blocks in this limit, we will then
localise back to the original algebra with original parameter
values. The following proposition taken from 
\cite[\S3]{mgp3} justifies this. 
\begin{prop}[{\cite[\S3]{mgp3}}]\label{prop:glob}
  There exist right exact globalisation functors
  \begin{align*}
    G:b^x_n-mod \longrightarrow b^x_{n+1}-mod
  , \hspace{1in}  
    G':b^x_n-mod \longrightarrow b^x_{n+1}-mod
  \end{align*}
with the following properties:
\begin{enumerate}
\item There is a parameter change from $\bnx$ to $b^x_{n+1}$ under $G$
  which sends $w_1\mapsto -w_1-1$;
\item There is a parameter change under $G'$ 
  which sends $w_2\mapsto -w_2-1$;
\item $G\W{n,m}{\ee_1,\ee_2}=\W{n+1,m+\ee_1}{-\ee_1,\ee_2}$;
\item $G'\W{n,m}{\ee_1,\ee_2}=\W{n+1,m+\ee_2}{\ee_1,-\ee_2}$.
\end{enumerate}
There are also exact localisation functors
\begin{align*}
  F:b^x_n-mod &\longrightarrow b^x_{n-1}-mod \\
  F': b^x_n-mod &\longrightarrow b^x_{n-1}-mod 
\end{align*}
such that $F\circ G=\mathrm{id}$ and $F'\circ G'=\mathrm{id}$, and also
\begin{enumerate}
\item $F\W{n,m}{\ee_1,\ee_2}=\begin{cases}
    \W{n-1,m+\ee_1}{-\ee_1,\ee_2}&\text{ if }m\neq n-1\text{ or }\ee_1=-1\\
    0&\text{ if }m=n-1\text{ and }\ee_1=1;
    \end{cases}$
\item $F'\W{n,m}{\ee_1,\ee_2}=\begin{cases}
    \W{n-1,m+\ee_2}{\ee_1,-\ee_2}&\text{ if }m\neq n-1\text{ or }\ee_2=-1\\
    0&\text{ if }m=n-1\text{ and }\ee_2=1.
    \end{cases}$
\end{enumerate}
\end{prop}
Note that the localisation functor can annihilate modules, and
therefore it is possible for a block to ``break up'' when
localising. We will address this on a case by case basis when
determining the blocks below. Also, since we will always localise back
after globalising, we need only consider in the arguments below cell
modules $\W{N,m}{\ee_1,\ee_2}$ with $m\ll N$. 

\subsection{On standard module homomorphisms}

We now recall the homomorphisms from \cite{mgp3} and reformulate them
into the notation consistent with this paper.
\begin{thm}[{\cite[Theorem \ref{thm:hom4glob1}]{mgp3}}]\label{thm:qhom}
  Let $q$ be a primitive $2\ell$-th root of unity and
  $w_1$, $w_2\not\in\Z$. Suppose that $m-2\ell\geq0$ (with equality if
  and only if $\ee_1=\ee_2=1$). Then there exists a non-zero
  homomorphism
  \[\psi:\W{n,m}{\ee_1,\ee_2}\longrightarrow \W{n,m-2\ell}{\ee_1,\ee_2}.\]
\end{thm}
\begin{thm}[{\cite[Theorem \ref{thm:hom1glob1}]{mgp3}}]\label{thm:w1hom}
  Let $q$ be a primitive $2\ell$-th root of unity and
  $w_1\in\Z$. Suppose for $r\in\Z$ that
  $m>m-2(\ee_1w_1+r\ell)>0$. Then there exists a non-zero homomorphism
  \[\psi:\W{n,m}{\ee_1,\ee_2}\longrightarrow \W{n,m-2(\ee_1w_1+r\ell)}{-\ee_1,\ee_2}.\]
  If $q$ is not a root of unity, then set $\ell=0$ in the above.
\end{thm}
\begin{thm}[{\cite[Theorems \ref{thm:hom2glob1} and
    \ref{thm:hom2glob3}]{mgp3}}]\label{thm:w2hom}
  Let $q$ be a primitive $2\ell$-th root of unity and
  $w_2\in\Z$. Suppose for $r\in\Z$ that
  $m>m-2(\ee_2w_2+r\ell)>0$. Then there exists a non-zero homomorphism
  \[\psi:\W{n,m}{\ee_1,\ee_2}\longrightarrow \W{n,m-2(\ee_2w_2+r\ell)}{\ee_1,-\ee_2}.\]
  If $q$ is not a root of unity then set $\ell=0$ in the above.
\end{thm}
\begin{thm}[{\cite[Theorems \ref{thm:hom3glob1}  and
    \ref{thm:hom3glob2}]{mgp3}}]\label{thm:w1w2hom}
  Let $q$ be a primitive $2\ell$-th root of unity and
  $\ee_1w_1+\ee_2w_2\in\Z$. Suppose for $r\in\Z$ that
  $m>2(\ee_1w_1+\ee_2w_2+r\ell)-m\geq0$ (with equality only if
  $\ee_1=\ee_2=1$). Then there exists a non-zero homomorphism
  \[\psi:\W{n,m}{\ee_1,\ee_2}\longrightarrow \W{n,2(\ee_1w_1+\ee_2w_2+r\ell)-m}{\ee_1,\ee_2}.\]
If $q$ is not a root of unity then set $\ell=0$ in the above.
\end{thm}

\subsection{Block master equations}

By Proposition \ref{pr:dn210}
 a necessary condition for any two cell modules to be
in the same block is that $Z_n$ acts by the same constant on both
modules. Notice that
\begin{align*}
  \alpha^{(n,m)}_{\ee_1,\ee_2}&=[n]\frac{[2(-m+\ee_1w_1+\ee_2w_2)]}{[-m+\ee_1w_1+\ee_2w_2]}\\
  &=[n]\left(q^{-m+\ee_1w_1+\ee_2w_2}+q^{m-\ee_1w_1-\ee_2w_2}\right)
\end{align*}
Therefore if $\alpha^{(n,m)}_{\ee_1,\ee_2}=\alpha^{(n,t)}_{\eta_1,\eta_2}$ and $[n]\neq0$, then
\[q^{-m+\ee_1w_1+\ee_2w_2}+q^{m-\ee_1w_1-\ee_2w_2}=q^{-t+\eta_1w_1+\eta_2w_2}+q^{t-\eta_1w_1-\eta_2w_2}.\]
Thus we have
$q^{-m+\ee_1w_1+\ee_2w_2}=q^{\pm(-t+\eta_1w_1+\eta_2w_2)}$.
This can only be satisfied if either
\begin{align*}
  -(m-t)+(\ee_1-\eta_1)w_1+(\ee_2-\eta_2)w_2&\equiv0\text{ (mod }2\ell),\text{ or}\\
  -(m+t)+(\ee_1+\eta_1)w_1+(\ee_2+\eta_2)w_2&\equiv0\text{ (mod }2\ell).
\end{align*}
In the first case, the allowed values of $\eta_1,\eta_2$ lead to the following possibilities:
\begin{align}
  \ee_1\neq\eta_1,\ee_2\neq\eta_2&\implies m-t\equiv2\ee_1w_1+2\ee_2w_2\text{ (mod }2\ell)\label{eq:w1w2neg}\\
  \ee_1\neq\eta_1,\ee_2=\eta_2&\implies m-t\equiv2\ee_1w_1\text{ (mod }2\ell)\label{eq:w1neg}\\
  \ee_1=\eta_1,\ee_2\neq\eta_2&\implies m-t\equiv2\ee_2w_2\text{ (mod }2\ell)\label{eq:w2neg}\\
  \ee_1=\eta_1,\ee_2=\eta_2&\implies m-t\equiv0\text{ (mod }2\ell),\label{eq:trivial}
\end{align}
and in the second case we have:
\begin{align}
  \ee_1=\eta_1,\ee_2=\eta_2&\implies m+t\equiv2\ee_1w_1+2\ee_2w_2\text{ (mod }2\ell)\label{eq:w1w2pos}\\
  \ee_1=\eta_1,\ee_2\neq\eta_2&\implies m+t\equiv2\ee_1w_1\text{ (mod }2\ell)\label{eq:w1pos}\\
  \ee_1\neq\eta_1,\ee_2=\eta_2&\implies m+t\equiv2\ee_2w_2\text{ (mod }2\ell)\label{eq:w2pos}\\
  \ee_1\neq\eta_1,\ee_2\neq\eta_2&\implies m+t\equiv0\text{ (mod }2\ell).\label{eq:impossible}
\end{align}
If $q$ is not a root of unity, then all of the congruences modulo
$2\ell$ in the above become equalities.

If $[n]=0$, then we can still use equations \eqref{eq:w1w2neg}--\eqref{eq:impossible} by first globalising to $b^x_N$ where $[N]\neq0$, determining the blocks there, and then localising again.

\section{Decomposition matrices and blocks of $\bnp$}\label{sec:decompblocks}
In the following subsections we will consider separately various cases
relating to whether or not certain linear combinations of $w_1$ and
$w_2$ are integers.

To visualise solutions to the master equations 
(\ref{eq:w1w2neg}-\ref{eq:impossible}),
we will plot points in the
plane corresponding to cell modules,
in such a way that solutions are manifested geometrically.
(Remark: this indicates the potential for a
geometric linkage principle, cf. \cite{jantz},
to describe the representation theory of the algebra.)
The cell module $\WW$ is given `weight' coordinates
$$
\WW \mapsto \big(\ee_1(m-\ee_1w_1-\ee_2w_2),\ee_2(m-\ee_1w_1-\ee_2w_2)\big) ,
$$
--- see e.g. Figure \ref{fig:cgl10}, Figure \ref{fig:w1orw2}.
In this geometry  
(in the $q$ not a root of unity case) two cell modules have the
same $Z_n$-eigenvalue if and only if one can be reached from the other
by successive reflections in the coordinate axes.
As a guide to the eye, the cell module
$\WW$ is plotted on the `arm' labelled by $\ee_1,\ee_2$.

\begin{figure}
\centering
  \begin{tikzpicture}[scale=0.3,>=latex] 
      \pgfmathsetmacro{\w}{1/2}
      \pgfmathsetmacro{\ww}{3/4}
      \draw (0,-8)--(0,8);
      \draw (-8,0)--(8,0);
      \foreach \m in {1,3,5,7}
      \foreach \e in {-1,1}
      \foreach \f in {-1,1}
      {       
        \draw (\e * \m - \e * \e * \w - \e * \f * \ww,\f * \m - \f * \e * \w - \f * \f * \ww)  node {$\bullet$};
      }
      \draw[white,fill=white] (-1 - \w - \ww, -1 - \w - \ww) circle (10pt);

      \draw (1 - \w - \ww, 1 - \w - \ww) -- ++(6,6) node[anchor=south west] {$+,+$};
      \draw (-1 - \w + \ww, 1 + \w - \ww) -- ++(-6,6) node[anchor=south east] {$-,+$};
      \draw (1 - \w + \ww, -1 + \w - \ww) -- ++(6,-6) node[anchor=north west] {$+,-$};
      \draw (-3 - \w - \ww, -3 - \w - \ww) -- ++(-4,-4) node[anchor=north east] {$-,-$};
    \end{tikzpicture}
\caption{Graphical depiction of the cell modules of $b_8'$ with
  $w_1=\half$ and $w_2=\frac{3}{4}$.
}
\label{fig:cgl10}
\end{figure}

\subsection{Cases with none of $w_1,w_2,w_1+w_2,w_1-w_2$ integral}\label{sec:noneint}

Suppose first that $q$ is not a root of unity. Since $m$ and $t$ are
positive integers, it is only possible for at most one of
\eqref{eq:w1w2neg} and \eqref{eq:w1w2pos} to be satisfied (similarly
for \eqref{eq:w1neg} and \eqref{eq:w1pos}; \eqref{eq:w2neg} and
\eqref{eq:w2pos}; and \eqref{eq:trivial} and \eqref{eq:impossible}).
The case \eqref{eq:impossible} is impossible as both $m$ and $t$ are
positive integers, and at most one can be zero. The case
\eqref{eq:trivial} is trivial, as the two modules are equal here. Also
since $\ee_1$ and $\ee_2$ take values $\pm1$, we can have non-trivial
coincidences of the eigenvalues of $Z_n$ if and only if
$\{w_1,w_2,w_1+w_2,w_1-w_2\}\,\cap\,\Z\neq\emptyset$. This leads to
the first main theorem of this paper:
\begin{thm}\label{thm:semisimple}
  Suppose $q$ is not a root of unity and
  $\{w_1,w_2,w_1+w_2,w_1-w_2\}\,\cap\,\Z=\emptyset$. Then the
algebra $b'_n$ is semisimple. If in addition, $\theta \ne \pm( -m \pm
w_1 \pm w_2)$ for any $m \in \Z$ then 
  symplectic blob algebra $b^x_n$ is semisimple.
\end{thm}
\begin{proof}
To prove the first statement, it suffices to show that the eigenvalues
of $Z_n$ are all distinct. Indeed, since none of $w_1,w_2,w_1+w_2$ or
$w_1-w_2$ are integral the only possible solution to equations
\eqref{eq:w1w2neg}--\eqref{eq:impossible} is the trivial one in
\eqref{eq:trivial}. Therefore each cell module is alone in its block
and the algebra is semisimple.

To prove the second, the only additional information needed is that
$\Wb$ is simple. This is guaranteed as for our chosen value of
$\theta$, the Gram determinant of $\Wb$ is non-zero by
\cite[Theorem 5.17]{degiernichols}.
\end{proof}

\newcommand{\qellru}{$q^{2\ell} =1$}

If now $q$ is a $2\ell$-th root of unity, then we must consider
equations \eqref{eq:trivial} and \eqref{eq:impossible}. 
Note that the left- and right-blob algebras are semisimple, so if
$\eta_1\neq\ee_1$ and $\eta_2\neq\ee_2$ then by restricting to either
algebra and considering the standard contents of Table
\ref{tab:filtration} we see that there can be no homomorphisms between
any modules satisfying \eqref{eq:impossible}. It therefore remains to
consider \eqref{eq:trivial}. Suppose without loss of generality that
$t<m$. 
Then $0\leq m-2\ell$ (with equality only if $\ee_1=\ee_2=1$), so by
Theorem \ref{thm:qhom}, we have a non-zero homomorphism
\begin{equation} \label{eq:ww1}
\W{n,m}{\ee_1,\ee_2}\longrightarrow \W{n,m-2\ell}{\ee_1,\ee_2}
\end{equation}
Thus $\W{n,m}{\ee_1,\ee_2}$ and
$\W{n,m-2\ell}{\ee_1,\ee_2}$ are in the same block. 
Similarly we have $\W{n,m-2\ell}{\ee_1,\ee_2}$ and
$\W{n,m-4\ell}{\ee_1,\ee_2}$ in the same block, and so on. 
By transitivity, we therefore see that $\W{n,m}{\ee_1,\ee_2}$ and
$\W{n,t}{\ee_1,\ee_2}$ are in  the same block.  

\begin{thm}\label{thm:qrootofunity}
Suppose \qellru\  
and $\{w_1,w_2,w_1+w_2,w_1-w_2\}\,\cap\,\Z=\emptyset$. 
Then two cell modules $\W{n,m}{\ee_1,\ee_2}$ and
$\W{n,t}{\eta_1,\eta_2}$ are in the same block if and only if
$\ee_1=\eta_1$, $\ee_2=\eta_2$ and $m\equiv t$ (mod $2\ell$). 
\end{thm}

\begin{figure}
\centering
  \begin{tikzpicture}[scale=0.3,>=latex] 
      \pgfmathsetmacro{\w}{1/2}
      \pgfmathsetmacro{\ww}{3/4}
      \draw (0,-8)--(0,8);
      \draw (-8,0)--(8,0);
      \foreach \m in {1,3,5,7,9,11,13}
      \foreach \e in {-1,1}
      \foreach \f in {-1,1}
      {       
        \draw (\e * \m - \e * \e * \w - \e * \f * \ww,\f * \m - \f * \e * \w - \f * \f * \ww)  node {$\bullet$};
      }
      \draw[white,fill=white] (-1 - \w - \ww, -1 - \w - \ww) circle (10pt);

      \draw[dashed,->] (7 - \w - \ww, 7 - \w - \ww) to[out=-90,in=0] (1 - \w - \ww, 1 - \w - \ww);
      \draw[dashed,->] (9 - \w - \ww, 9 - \w - \ww) to[out=-90,in=0] (3 - \w - \ww, 3 - \w - \ww);
      \draw[dashed,->] (11 - \w - \ww, 11 - \w - \ww) to[out=-90,in=0] (5 - \w - \ww, 5 - \w - \ww);
      \draw[dashed,->] (13 - \w - \ww, 13 - \w - \ww) to[out=-90,in=0] (7 - \w - \ww, 7 - \w - \ww);

      \draw[dashed,->] (-7 - \w + \ww, 7 + \w - \ww) to[out=-90,in=180] (-1 - \w + \ww, 1 + \w - \ww);
      \draw[dashed,->] (-9 - \w + \ww, 9 + \w - \ww) to[out=-90,in=180] (-3 - \w + \ww, 3 + \w - \ww);
      \draw[dashed,->] (-11 - \w + \ww, 11 + \w - \ww) to[out=-90,in=180] (-5 - \w + \ww, 5 + \w - \ww);
      \draw[dashed,->] (-13 - \w + \ww, 13 + \w - \ww) to[out=-90,in=180] (-7 - \w + \ww, 7 + \w - \ww);

      \draw[dashed,->] (7 - \w + \ww, -7 + \w - \ww) to[out=180,in=-90] (1 - \w + \ww, -1 + \w - \ww);
      \draw[dashed,->] (9 - \w + \ww, -9 + \w - \ww) to[out=180,in=-90] (3 - \w + \ww, -3 + \w - \ww);
      \draw[dashed,->] (11 - \w + \ww, -11 + \w - \ww) to[out=180,in=-90] (5 - \w + \ww, -5 + \w - \ww);
      \draw[dashed,->] (13 - \w + \ww, -13 + \w - \ww) to[out=180,in=-90] (7 - \w + \ww, -7 + \w - \ww);

      \draw[dashed,->] (-9 - \w - \ww, -9 - \w - \ww)
      to[out=0,in=-90] (-3 - \w - \ww, -3 - \w - \ww);
      \draw[dashed,->] (-11 - \w - \ww, -11 - \w - \ww)
      to[out=0,in=-90] (-5 - \w - \ww, -5 - \w - \ww);
      \draw[dashed,->] (-13 - \w - \ww, -13 - \w - \ww)
      to[out=0,in=-90] (-7 - \w - \ww, -7 - \w - \ww);

      \draw (1 - \w - \ww, 1 - \w - \ww) -- ++(12,12)
         node[anchor=south west] {};           
      \draw (-1 - \w + \ww, 1 + \w - \ww) -- ++(-12,12)
         node[anchor=south east] {}; 
      \draw (1 - \w + \ww, -1 + \w - \ww) -- ++(12,-12)
         node[anchor=north west] {}; 
      \draw (-3 - \w - \ww, -3 - \w - \ww) -- ++(-10,-10)
         node[anchor=north east] {}; 
    \end{tikzpicture}
  \caption{
    Cell modules of $b_{14}'$ with $w_1=\half$ and
    $w_2=\frac{3}{4}$.
    Dashed lines indicate modules in the same block when $\ell=3$.}
\label{fig:cgl1paul}
\end{figure}
\begin{figure}
\centering
  \begin{tikzpicture}[scale=0.3,>=latex] 
      \pgfmathsetmacro{\w}{1/2}
      \pgfmathsetmacro{\ww}{3/4}
      \draw (0,-8)--(0,8);
      \draw (-8,0)--(8,0);
      \foreach \m in {1,3,5,7}
      \foreach \e in {-1,1}
      \foreach \f in {-1,1}
      {       
        \draw (\e * \m - \e * \e * \w - \e * \f * \ww,\f * \m - \f * \e * \w - \f * \f * \ww)  node {$\bullet$};
      }
      \draw[white,fill=white] (-1 - \w - \ww, -1 - \w - \ww) circle (10pt);

      \draw[dashed,->] (7 - \w - \ww, 7 - \w - \ww) to[out=-90,in=0] (1 - \w - \ww, 1 - \w - \ww);
      \draw[dashed,->] (-7 - \w + \ww, 7 + \w - \ww) to[out=-90,in=180] (-1 - \w + \ww, 1 + \w - \ww);
      \draw[dashed,->] (7 - \w + \ww, -7 + \w - \ww) to[out=180,in=-90] (1 - \w + \ww, -1 + \w - \ww);

      \draw (1 - \w - \ww, 1 - \w - \ww) -- ++(6,6) node[anchor=south west] {$+,+$};
      \draw (-1 - \w + \ww, 1 + \w - \ww) -- ++(-6,6) node[anchor=south east] {$-,+$};
      \draw (1 - \w + \ww, -1 + \w - \ww) -- ++(6,-6) node[anchor=north west] {$+,-$};
      \draw (-3 - \w - \ww, -3 - \w - \ww) -- ++(-4,-4) node[anchor=north east] {$-,-$};
    \end{tikzpicture}
\caption{Graphical depiction of the cell modules of $b_8'$ with
  $w_1=\half$ and $w_2=\frac{3}{4}$. The dashed lines indicate
  modules in the same block when $\ell=3$.}
\label{fig:cgl1}
\end{figure}

The combinatorial-geometric expression of linkage in this case is as
in Figure~\ref{fig:cgl1paul}. See Figure \ref{fig:cgl1} for the
truncation to $n=8$.

\subsection{Either $w_1$ or $w_2$ integral}\label{sec:w1orw2}

We will determine the blocks when precisely one of $w_1$ and $w_2$ is
integral. We begin with the case $w_1\in\Z,w_2\not\in\Z$, and first
assume that $q$ is not a root of unity. Now the only equations from
\eqref{eq:w1w2neg}--\eqref{eq:impossible} with non-trivial solutions
are \eqref{eq:w1neg} and \eqref{eq:w1pos}, and by fixing $m, \ee_1$
and $\ee_2$ we see that blocks have size at most two. Note that if
$w_2\in\half\Z$ then we still do not obtain extra solutions since
$m\pm t$ is always even.

Consider first the case \eqref{eq:w1pos}, where we have $\ee_1=\eta_1$
and $\ee_2\neq\eta_2$. We will show that although these two modules
have the same eigenvalue, they are not in the same block. As
right-blob modules, they have a filtration as in Table
\ref{tab:filtration}. However since the parameter $w_2$ is not
integral the right-blob algebra is semisimple, and thus it is not
possible to have a non-zero homomorphism between the modules. Since
the block has size at most two, we deduce that these modules are not
in the same block.

Now consider \eqref{eq:w1neg}. Here, we have $\ee_1\neq\eta_1$ and
$\ee_2=\eta_2$. By swapping labels if necessary we may assume that
$m>t$ (equality is not possible due to \cite[Proposition
3.4.1]{mgp3}). Since both $m$ and $t$ are non-negative integers we
must have $\ee_1=\mathrm{sgn}(w_1)$, thus
$m>t=m-2\ee_1w_1>0$. Therefore the conditions of Theorem
\ref{thm:w1hom} are satisfied and we have a homomorphism
$\W{n,m}{\ee_1,\ee_2}\longrightarrow \W{n,t}{-\ee_1,\ee_2}$. 

We will now consider the case when $q$ is a $2\ell$-th root of unity
and $w_1\in\Z$, $w_2\not\in\Z$. In this case, the only equations from
\eqref{eq:w1w2neg}--\eqref{eq:impossible} with solutions are
\eqref{eq:w1neg}, \eqref{eq:trivial}, \eqref{eq:w1pos} and
\eqref{eq:impossible}. We still do not obtain extra solutions if
$w_2\in\half\Z$ by parity considerations in the same way as above.

Begin by fixing the cell module with labels $m, \ee_1$ and $\ee_2$. By
restricting to the right-blob algebra as before, any other cell module
$\W{n,t}{\eta_1,\eta_2}$ in this block has $\eta_2=\ee_2$. We can
therefore rule out equations \eqref{eq:w1pos} and
\eqref{eq:impossible}. In the case of equation \eqref{eq:w1neg} we
again see that the conditions of Theorem \ref{thm:w1hom} are satisfied
(this time with $\ell\neq0$), and so these cell modules are in the
same block. So it remains to consider the case of
\eqref{eq:trivial}. We will begin by showing that
$\W{n,m}{\ee_1,\ee_2}$ and $\W{n,m+2\ell}{\ee_1,\ee_2}$ are in
the same block, and the general result will follow. Indeed, we choose
$r\in\Z$ such that $0<\ee_1w_1+r\ell<\ell$, then by Theorem
\ref{thm:w1hom} we have non-zero homomorphisms
$\W{n,m+2\ell}{\ee_1,\ee_2}\longrightarrow
\W{n,m+2\ell-2(\ee_1w_1+r\ell)}{-\ee_1,\ee_2}$ and
$\W{n,m+2\ell-2(\ee_1w_1+r\ell)}{-\ee_1,\ee_2}\longrightarrow
\W{n,m}{\ee_1,\ee_2}$. Therefore our original pair of cell modules
are in the same block.

The proof for $w_2\in\Z$, $w_1\not\in\Z$ is similar, except we must
consider cases \eqref{eq:w2neg} and \eqref{eq:w2pos}, and use Theorem
\ref{thm:w2hom} in place of Theorem \ref{thm:w1hom}. We therefore have
the following theorem:

\begin{thm}\label{thm:w1int}
  Suppose $q$ is a $2\ell$-th root of unity and
  $w_1\in\Z$, $w_2\not\in\Z$. Then two cell modules
  $\W{n,m}{\ee_1,\ee_2}$ and $\W{n,t}{\eta_1,\eta_2}$ are in the
  same block if and only if $\ee_2=\eta_2$ and
  $|m-\ee_1w_1-\ee_2w_2|\equiv|t-\eta_1w_1-\ee_2w_2|\pmod{2\ell}$.

If now $w_2\in\Z,w_1\not\in\Z$, then two cell modules
$\W{n,m}{\ee_1,\ee_2}$ and $\W{n,t}{\eta_1,\eta_2}$ are in the
same block if and only if $\ee_1=\eta_1$ and
$|m-\ee_1w_1-\ee_2w_2|\equiv|t-\eta_1w_1-\ee_2w_2|\pmod{2\ell}$.

If $q$ is not a root of unity, then replace the above two congruences modulo $2\ell$ by equalities.
\end{thm}

Figure \ref{fig:w1orw2} shows two plots of the cell modules, when just
$w_1$ and just $w_2$ are integral respectively. The arrows indicate a
homomorphism between the corresponding modules.

\begin{figure}
  \centering
  (i)\begin{tikzpicture}[scale=0.25,>=latex,baseline=0]
    \pgfmathsetmacro{\w}{1}
    \pgfmathsetmacro{\ww}{3/4}
    \draw (0,-8)--(0,8);
    \draw (-8,0)--(8,0);
    \foreach \m in {1,3,5,7}
    \foreach \e in {-1,1}
    \foreach \f in {-1,1}
    {       
      \draw (\e * \m - \e * \e * \w - \e * \f * \ww,\f * \m - \f * \e * \w - \f * \f * \ww)  node {$\bullet$};
    }
    \draw[white,fill=white] (-1 - \w - \ww, -1 - \w - \ww) circle (10pt);
    \draw (1 - \w - \ww, 1 - \w - \ww) -- ++(6,6) node[anchor=south west] {$+,+$};
    \draw (-1 - \w + \ww, 1 + \w - \ww) -- ++(-6,6) node[anchor=south east] {$-,+$};
    \draw (1 - \w + \ww, -1 + \w - \ww) -- ++(6,-6) node[anchor=north west] {$+,-$};
    \draw (-3 - \w - \ww, -3 - \w - \ww) -- ++(-4,-4) node[anchor=north east] {$-,-$};

    \draw[dashed,->] (3 - \w - \ww, 3 - \w - \ww) -- (-3 + \w + \ww, 3 - \w - \ww);
    \draw[dashed,->] (5 - \w - \ww, 5 - \w - \ww) -- (-5 + \w + \ww, 5 - \w - \ww);
    \draw[dashed,->] (7 - \w - \ww, 7 - \w - \ww) -- (-7 + \w + \ww, 7 - \w - \ww);
    \draw[dashed,->] (5 - \w + \ww, -5 + \w - \ww) -- (-5 + \w - \ww, -5 + \w - \ww);
    \draw[dashed,->] (7 - \w + \ww, -7 + \w - \ww) -- (-7 + \w - \ww, -7 + \w - \ww);
  \end{tikzpicture}\hspace{1em}
  (ii)\begin{tikzpicture}[scale=0.25,>=latex,baseline=0] 
    \pgfmathsetmacro{\w}{-1/4}
    \pgfmathsetmacro{\ww}{1}
    \draw (0,-8)--(0,8);
    \draw (-8,0)--(8,0);
    \foreach \m in {0,2,4,6,8}
    \foreach \e in {-1,1}
    \foreach \f in {-1,1}
    {       
      \draw (\e * \m - \e * \e * \w - \e * \f * \ww,\f * \m - \f * \e * \w - \f * \f * \ww)  node {$\bullet$};
    }
    \fill[white] (0 -\w + \ww, 0 + \w - \ww) circle (10pt);
    \fill[white] (0 -\w - \ww, 0 + \w - \ww) circle (10pt);
    \draw (0 - \w - \ww, 0 - \w - \ww) -- ++(8,8) node[anchor=south west] {$+,+$};
    \draw (-2 - \w + \ww, 2 + \w - \ww) -- ++(-6,6) node[anchor=south east] {$-,+$};
    \draw (2 - \w + \ww, -2 + \w - \ww) -- ++(6,-6) node[anchor=north west] {$+,-$};
    \draw (-2 - \w - \ww, -2 - \w - \ww) -- ++(-6,-6) node[anchor=north east] {$-,-$};

    \draw[dashed,->] (4 - \w - \ww, 4 - \w - \ww) -- (4 - \w - \ww, -4 + \w + \ww);
    \draw[dashed,->] (6 - \w - \ww, 6 - \w - \ww) -- (6 - \w - \ww, -6 + \w + \ww);
    \draw[dashed,->] (8 - \w - \ww, 8 - \w - \ww) -- (8 - \w - \ww, -8 + \w + \ww);
    \draw[dashed,->] (-4 - \w + \ww, 4 + \w - \ww) -- (-4 - \w + \ww, -4 - \w + \ww);
    \draw[dashed,->] (-6 - \w + \ww, 6 + \w - \ww) -- (-6 - \w + \ww, -6 - \w + \ww);
    \draw[dashed,->] (-8 - \w + \ww, 8 + \w - \ww) -- (-8 - \w + \ww, -8 - \w + \ww);
  \end{tikzpicture}
  \caption{Graphical depiction of the cell modules of (i) $b_8'$ with $w_1=1$ and $w_2=\frac{3}{4}$; and (ii) $b_9'$ with $w_1=-\frac{1}{4}$ and $w_2=1$.}\label{fig:w1orw2}
\end{figure}

\subsection{Either $w_1+w_2$ or $w_1-w_2$ integral}\label{sec:w1+w2orw1-w2}

We first turn to the case $w_1+w_2\in\Z$ but
$w_1-w_2\not\in\Z$. Again, we begin by taking $q$ to not be a root of
unity. Here, we are looking to satisfy equations \eqref{eq:w1w2neg}
and \eqref{eq:w1w2pos} with $\ee_1=\ee_2$. Once more, we see that
blocks have size at most two. Now if we have a solution to equation
\eqref{eq:w1w2neg}, then for these modules to be in the same block we
must have a non-zero homomorphism
\[\W{n,m}{+,+}\longrightarrow \W{n,m-2w_1-2w_2}{-,-}.\]
However by restricting both modules to the left-blob algebra we see
from Table \ref{tab:filtration} that the two modules have different
standard contents. Thus, since $w_1\not\in\Z$, we deduce that there
can be no such homomorphism.

In the case of equation \eqref{eq:w1w2pos}, we can again assume that
$m>t$. Then since $m$ and $t$ are both non-negative integers we can
only have a solution if $\ee_1=\ee_2=\mathrm{sgn}(w_1+w_2)$. Therefore
we have $m>t=2(\ee_1w_1+\ee_2w_2)-m\geq0$ (with equality only if
$\ee_1=\ee_2=1$), and thus we can use Theorem \ref{thm:w1w2hom} to
show that the cell modules $\W{n,m}{+,+}$ and $\W{n,t}{+,+}$ are
in the same block.

If $q$ is a $2\ell$-th root of unity, then we must consider equations
\eqref{eq:w1w2neg}, \eqref{eq:trivial}, \eqref{eq:w1w2pos} and
\eqref{eq:impossible}. Again, since neither $w_1$ nor $w_2$ are
integers we can rule out \eqref{eq:w1w2neg} and \eqref{eq:impossible}
by restricting to either the left- or right-blob algebra. In the case
of \eqref{eq:w1w2pos}, we can show that there exists a non-zero
homomorphism between the cell modules in the same way as above. It
remains, therefore, to deal with \eqref{eq:trivial}. We will begin by
showing that $\W{n,m}{\ee_1,\ee_2}$ and
$\W{n,m+2\ell}{\ee_1,\ee_2}$ are in the same block, and the general
result will follow. Indeed, we choose $r\in\Z$ such that
$m+2\ell>2(\ee_1w_1+\ee_2w_2+r\ell)-(m+2\ell)>m$, then by Theorem
\ref{thm:w1w2hom} we have non-zero homomorphisms
$\W{n,m+2\ell}{\ee_1,\ee_2}\longrightarrow
\W{n,2(\ee_1w_1+\ee_2w_2+r\ell)-(m+2\ell)}{\ee_1,\ee_2}$ and
$\W{n,2(\ee_1w_1+\ee_2w_2+r\ell)-(m+2\ell)}{\ee_1,\ee_2}\longrightarrow
\W{n,m}{\ee_1,\ee_2}$. Therefore our original pair of cell modules
are in the same block.
\begin{thm}\label{thm:w1+w2}
  Suppose $q$ is a $2\ell$-th root of unity and
  $w_1+w_2\in\Z,w_1-w_2\not\in\Z$. Then two cell modules
  $W^{(n,m)}_{\ee_1,\ee_2}$ and $W^{(n,t)}_{\eta_1,\eta_2}$ are in the
  same block if and only if $\ee_1=\eta_1=\ee_2=\eta_2$ and
  $|m-\ee_1w_1-\ee_2w_2|\equiv|t-\eta_1w_1-\eta_2w_2|$ (mod $2\ell$).

If $q$ is not a root of unity then replace the above congruence modulo $2\ell$ by an equality.
\end{thm}

The case $w_1-w_2\in\Z_{>0}$ but $w_1$, $w_2$, $w_1+w_2\not\in\Z$ is proved similarly.
\begin{thm}\label{thm:w1-w2}
  Suppose $q$ is a $2\ell$-th root of unity and
  $w_1-w_2\in\Z,w_1+w_2\not\in\Z$. Then two cell modules
  $\W{n,m}{\ee_1,\ee_2}$ and $\W{n,t}{\eta_1,\eta_2}$ are in the
  same block if and only if $\ee_1=\eta_1=-\ee_2=-\eta_2$ and
  $|m-\ee_1w_1-\ee_2w_2|\equiv|t-\eta_1w_1-\eta_2w_2|$ (mod $2\ell$).

If $q$ is not a root of unity then replace the above congruence modulo $2\ell$ by an equality.
\end{thm}

Figure \ref{fig:w1+w2orw1-w2} shows two plots of the cell modules,
when just $w_1+w_1$ and just $w_1-w_2$ are integral respectively. The
arrows indicate a homomorphism between the corresponding modules.

\begin{figure}
  \centering
  (i)\begin{tikzpicture}[scale=0.25,>=latex,baseline=0]
    \pgfmathsetmacro{\w}{1/4}
    \pgfmathsetmacro{\ww}{11/4}
    \draw (0,-8)--(0,8);
    \draw (-8,0)--(8,0);
    \foreach \m in {0,2,4,6,8}
    \foreach \e in {-1,1}
    \foreach \f in {-1,1}
    {       
      \draw (\e * \m - \e * \e * \w - \e * \f * \ww,\f * \m - \f * \e * \w - \f * \f * \ww)  node {$\bullet$};
    }
    \fill[white] (0 -\w + \ww, 0 + \w - \ww) circle (10pt);
    \fill[white] (0 -\w - \ww, 0 + \w - \ww) circle (10pt);
    \draw (0 - \w - \ww, 0 - \w - \ww) -- ++(8,8) node[anchor=south west] {$+,+$};
    \draw (-2 - \w + \ww, 2 + \w - \ww) -- ++(-6,6) node[anchor=south east] {$-,+$};
    \draw (2 - \w + \ww, -2 + \w - \ww) -- ++(6,-6) node[anchor=north west] {$+,-$};
    \draw (-2 - \w - \ww, -2 - \w - \ww) -- ++(-6,-6) node[anchor=north east] {$-,-$};

    \draw[dashed,->] (4 - \w - \ww, 4 - \w - \ww) to[out=180,in=90] (2 - \w - \ww, 2 - \w - \ww);
    \draw[dashed,->] (6 - \w - \ww, 6 - \w - \ww) to[out=180,in=90] (0 - \w - \ww, 0 - \w - \ww);
  \end{tikzpicture}\hspace{1em}
  (ii)\begin{tikzpicture}[scale=0.25,>=latex,baseline=0] 
    \pgfmathsetmacro{\w}{1/4}
    \pgfmathsetmacro{\ww}{-7/4}
    \draw (0,-8)--(0,8);
    \draw (-8,0)--(8,0);
    \foreach \m in {1,3,5,7}
    \foreach \e in {-1,1}
    \foreach \f in {-1,1}
    {       
      \draw (\e * \m - \e * \e * \w - \e * \f * \ww,\f * \m - \f * \e * \w - \f * \f * \ww)  node {$\bullet$};
    }
    \fill[white] (-1 - \w - \ww, -1 - \w -\ww) circle (10pt);
    \draw (1 - \w - \ww, 1 - \w - \ww) -- ++(6,6) node[anchor=south west] {$+,+$};
    \draw (-1 - \w + \ww, 1 + \w - \ww) -- ++(-6,6) node[anchor=south east] {$-,+$};
    \draw (1 - \w + \ww, -1 + \w - \ww) -- ++(6,-6) node[anchor=north west] {$+,-$};
    \draw (-3 - \w - \ww, -3 - \w - \ww) -- ++(-4,-4) node[anchor=north east] {$-,-$};
    
    \draw[dashed,->] (3 - \w + \ww, -3 + \w - \ww) to[out=90,in=0] (1 - \w + \ww, -1 + \w - \ww);
  \end{tikzpicture}
  \caption{Graphical depiction of the cell modules of (i) $b_9'$ with $w_1=\frac{1}{4}$ and $w_2=\frac{11}{4}$; and (ii) $b_8'$ with $w_1=\frac{1}{4}$ and $w_2=-\frac{7}{4}$.}\label{fig:w1+w2orw1-w2}
\end{figure}

\subsection{Both $w_1+w_2$ and $w_1-w_2$ integral}\label{sec:w1+w2andw1-w2}

If now we have $w_1+w_2$, $w_1-w_2\in\Z$ but $w_1$, $w_2\not\in\Z$, then we
must have $w_1$, $w_2\in\frac{1}{2}\Z\backslash\Z$. The labels $m$ and
$t$ for cell modules all have the same parity, in particular $m\pm t$
is always even, whereas both $2w_1$ and $2w_2$ are odd. Therefore
there can be no solutions to
\eqref{eq:w1neg},\eqref{eq:w2neg},\eqref{eq:w1pos} nor
\eqref{eq:w2pos}. Thus we simply combine Theorems \ref{thm:w1+w2} and
\ref{thm:w1-w2}.
\begin{thm}\label{thm:w1+w2andw1-w2}
  Suppose $q$ is a $2\ell$-th root of unity and
  $w_1+w_2\in\Z$, $w_1-w_2\in\Z$ but $w_1$, $w_2\not\in\Z$. Then two cell
  modules $\W{n,m}{\ee_1,\ee_2}$ and $\W{n,t}{\eta_1,\eta_2}$
  are in the same block if and only if
  $|m-\ee_1w_1-\ee_2w_2|\equiv|t-\eta_1w_1-\eta_2w_2|$ (mod $2\ell$)
  and either
  \begin{enumerate}
  \item $\ee_1=\eta_1=\ee_2=\eta_2$, or
  \item $\ee_1=\eta_1=-\ee_2=-\eta_2$.
  \end{enumerate}
  If $q$ is not a root of unity then replace the above congruence modulo $2\ell$ by an equality.
\end{thm}

Figure \ref{fig:w1+w2andw1-w2} shows a plot the cell modules when both
$w_1+w_2$ and $w_1-w_2$ are integral, but not $w_1$ nor $w_2$. The
arrows indicate a homomorphism between the corresponding modules.

\begin{figure}
  \centering
  \begin{tikzpicture}[scale=0.3,>=latex,baseline=0]
    \pgfmathsetmacro{\w}{5/2}
    \pgfmathsetmacro{\ww}{-1/2}
    \draw (0,-8)--(0,8);
    \draw (-8,0)--(8,0);
    \foreach \m in {1,3,5,7}
    \foreach \e in {-1,1}
    \foreach \f in {-1,1}
    {       
      \draw (\e * \m - \e * \e * \w - \e * \f * \ww,\f * \m - \f * \e * \w - \f * \f * \ww)  node {$\bullet$};
    }
    \fill[white] (-1 - \w- \ww, -1 - \w - \ww) circle (10pt);
    \draw (1 - \w - \ww, 1 - \w - \ww) -- ++(6,6) node[anchor=south west] {$+,+$};
    \draw (-1 - \w + \ww, 1 + \w - \ww) -- ++(-6,6) node[anchor=south east] {$-,+$};
    \draw (1 - \w + \ww, -1 + \w - \ww) -- ++(6,-6) node[anchor=north west] {$+,-$};
    \draw (-3 - \w - \ww, -3 - \w - \ww) -- ++(-4,-4) node[anchor=north east] {$-,-$};

    \draw[dashed,->] (5 - \w + \ww, -5 + \w - \ww) to[out=90,in=0] (1 - \w + \ww, -1 + \w - \ww);
    \draw[dashed,->] (3 - \w - \ww, 3 - \w - \ww) to[out=-90,in=0] (1 - \w - \ww, 1 - \w - \ww);
  \end{tikzpicture}
\caption{Graphical depiction of the cell modules of $b_8'$ with $w_1=\frac{5}{2}$ and $w_2=-\half$.}\label{fig:w1+w2andw1-w2}
\end{figure}
\subsection{Both $w_1$ and $w_2$ integral}\label{sec:w1andw2}

Finally, we consider the case when $w_1$, $w_2\in\Z$ and $q$ not a
root of unity. As explained in the beginning of this section, we can
globalise appropriately so that we only consider cell modules
$\W{n,m}{\ee_1,\ee_2}$ with $m\ll N$ and $w_1,w_2>0$. We may also
assume that $w_1\leq w_2$, since we can swap blobs and flip diagrams
horizontally to make this the case. By fixing
$\ee_1$ and $\ee_2$ and considering the equations
\eqref{eq:w1w2neg}--\eqref{eq:impossible} we see that each block has
size at most 4 as  $q$ is not a root of unity. 
By repeated
globalisation we will determine the blocks containing modules of the
form $\W{N,m}{+,+}$ where $N\gg m$. By then repeated localising, we
see that we will in fact deal with each cell module in $b'_n$. Care
must be taken when localising, as we will encounter blocks of size 4
in the large $N$ limit which may break up into singleton blocks when
localising back to $n$.

Consider first the case $m<w_1+w_2$. Then we see by Theorem
\ref{thm:w1w2hom} that we have a homomorphism
\[\W{N,2w_1+2w_2-m}{+,+}\longrightarrow \W{N,m}{+,+}.\]
Now we are assuming that $0<w_1\leq w_2$. Thus
$2w_1+2w_2-m>w_1+w_2\geq2w_1$, therefore $2w_1+2w_2-m-2w_1=2w_2-m>0$
and we can use Theorem \ref{thm:w1hom} to obtain a homomorphism
\[\W{N,2w_1+2w_2-m}{+,+}\longrightarrow \W{N,2w_2-m}{-,+}.\]
When localising it is possible that the module
$\W{N,2w_1+2w_2-m}{+,+}$ will be annihilated, in which case the
remaining two modules may not be linked via homomorphisms. We now show
that this is the case:
\begin{lem}\label{lem:nohom}
  Suppose $q\in\C^\times$ is not a root of unity and
  $w_1$, $w_2\in\Z_{>0}$ be as above. Let $m<w_1+w_2$ be a non-negative
  integer. Then we have
  \[\Hom(\W{n,m}{+,+}, \W{n,2w_2-m}{-,+})=\Hom(\W{n,2w_2-m}{-,+},\W{n,m}{+,+})=0.\]
\end{lem}
\begin{proof}
  If we have a non-zero homomorphism between the cell modules, then
  this must restrict to a homomorphism between cell modules for the
  left blob algebra. We will show that there can be no such
  homomorphism.

Suppose first that we have a non-zero homomorphism
$\W{n,m}{+,+}\longrightarrow \W{n,2w_2-m}{-,+}$. By applying the
localisation functor $F$ an even number of times (so that no parameter
change occurs), we may assume that $n=m+1$. Now restricting to the
left blob algebra, we see from Table \ref{tab:filtration} that
$\W{m+1,m}{+,+}$ has standard content $W_{m+1}(m+1)$ and
$\W{m+1,2w_2-m}{-,+}$ has standard content
$W_{-(m+1)}(m+1),W_{-(m-1)}(m+1),\ldots,W_{-(2w_2-m+1)}(m+1)$. As such,
we must have a left blob homomorphism from the trivial module to one
of the latter standard modules. However since $q$ is not a root of
unity, the module $W_{m+1}(m+1)$ is mapped only to
$W_{2w_1-m-1}(m+1)$. We then note that $2w_1-m-1>w_1-w_2-1$, whereas
$-(2w_2-m+1)<w_1-w_2-1$. Therefore there is no left-blob homomorphism
in the restriction, and thus no symplectic blob homomorphism.

If we now assume that we have a homomorphism
$\W{n,2w_2-m}{-,+}\longrightarrow \W{n,m}{+,+}$, then by
applying the localisation functor $F$ an odd number of times we may
assume that $n=2w_2-m$. Then by restricting to the left blob algebra
we see that $\W{2w_2-m,2w_2-m-1}{+,+}$ has standard content
$W_{2w_2-m}(2w_2-m)$, whereas $\W{2w_2-m,m+1}{-,+}$ has standard
content $W_{-(2w_2-m)}(2w_2-m),\dots,W_{-(m+2)}(2w_2-m)$. Again by
considering left blob homomorphisms with $q$ not a root of unity, and
taking the parameter change $w_1\mapsto -w_1-1$ into account, we see
that the trivial module is mapped only to
$2(-w_1-1)-2w_2+m=-2(w_1+2w_2)+m-2$. Note now that this is less than
$-2w_2+m$, which is the least label in the latter set of standard
contents. Thus there can be no left blob homomorphism, and hence no
symplectic blob homomorphism.
\end{proof}
So far we have found three distinct modules with the same eigenvalue,
all linked via homomorphisms.

If also $2w_1+2w_2-m>2w_2$ (so that $m<2w_1$), then we can use Theorem
\ref{thm:w2hom} to obtain a homomorphism
\[\W{N,2w_1+2w_2-m}{+,+}\longrightarrow \W{N,2w_1-m}{+,-}.\]
In a manner similar to Lemma \ref{lem:nohom}, we can show that there
are no homomorphisms between $\W{N,2w_1-m}{+,-}$ and
$\W{N,2w_2-m}{-,+}$ or $\W{N,m}{+,+}$. Since blocks have size at
most four, this block has the structure as in Figure \ref{fig:m<2w1},
where arrows indicate the existence of a homomorphism and lack of
arrows indicates the non-existence of a homomorphism.

\begin{figure}[ht]
  \centering
  \tikz{
  \node (A) {$\W{N,2w_1+2w_2-m}{+,+}$};
  \node[below left of = A,node distance=7em, yshift=2em] (B) {$\W{N,m}{+,+}$};
  \node[below right of = A, node distance=7em,yshift=2em] (C) {$\W{N,2w_2-m}{-,+}$};
  \node[below of = A, node distance=5em] (D) {$\W{N,2w_1-m}{+,-}$};
  \draw[->] (A)--(B);
  \draw[->] (A)--(C);
  \draw[->] (A)--(D);}
\caption{The block structure for $m<2w_1$.}\label{fig:m<2w1}
\end{figure}

Note that if $n\leq 2w_1+2w_2-m$, then the module
$\W{n,2w_1+2w_2-m}{+,+}$ does not exist. Since there are no
homomorphisms between the remaining three modules, the block breaks up
into singleton blocks.

If now $2w_1+2w_2-m=2w_2$ (so that $m=2w_1$), then we claim that there
are at most three possible solutions to equations
\eqref{eq:w1w2neg}--\eqref{eq:impossible}. In particular we can
satisfy neither \eqref{eq:w1neg} nor \eqref{eq:w1pos} and so the block
has size at most three, and we have found enough homomorphisms. Since
we have fixed $\ee_1=1$, equations \eqref{eq:w1neg} and
\eqref{eq:w1pos} reduce to $m\pm t=m$, so that $t=0$. However the only
module that exists when $t=0$ also must have $\eta_1=\eta_2=1$, which
is not valid when considering this pair of equations. We thus have the
block structure as in Figure \ref{fig:m=2w1}, with the same convention
for arrows (or lack thereof).

\begin{figure}[ht]
  \centering
  \tikz{
  \node (A) {$\W{N,2w_2}{+,+}$};
  \node[below left of = A,node distance=7em] (B) {$\W{N,2w_1}{+,+}$};
  \node[below right of = A, node distance=7em] (C) {$\W{N,2w_2-2w_1}{-,+}$};
    \draw[->] (A)--(B);
    \draw[->] (A)--(C);}
\caption{The block structure for $m=2w_1$.}\label{fig:m=2w1}
\end{figure}

Again we see that if $n\leq 2w_2$ then the module $\W{n,2w_2}{+,+}$
does not exist. Therefore the block once more breaks up into singleton
blocks.

Finally, if $2w_1+2w_2-m<2w_2$, then we have $m>2w_1$ and so we can
again use Theorem \ref{thm:w1hom} to obtain
\[\W{N,m}{+,+}\longrightarrow \W{N,m-2w_1}{-,+}.\]
We also see that the modules $\W{N,2w_2-m}{-,+}$ and
$\W{N,m-2w_1}{-,+}$ satisfy the conditions of Theorem
\ref{thm:w1w2hom}, and thus have a homomorphism
\[\W{N,2w_2-m}{-,+}\longrightarrow \W{N,m-2w_1}{-,+}.\]
Again, since blocks have size at most four, this block has the
structure as in Figure \ref{fig:m>2w1}, where again the arrows
indicate the existence of a homomorphism, the lack of arrows the
non-existence, and a dotted line indicates an unknown (which will not
matter when considering the block structure).

\begin{figure}[ht]
\centering
\tikz{
  \node (A) {$\W{N,2w_1+2w_2-m}{+,+}$};
  \node[below left of = A,node distance=7em] (B) {$\W{N,m}{+,+}$};
  \node[below right of = A, node distance=7em] (C) {$\W{N,2w_2-m}{-,+}$};
  \node[below left of = C, node distance=7em] (D) {$\W{N,m-2w_1}{-,+}$};
  \draw[->] (A)--(B);
  \draw[->] (A)--(C);
  \draw[->] (B)--(D);
  \draw[->] (C)--(D);
  \draw[dotted] (A)--(D);}
\caption{The block structure for $2w_1<m<w_1+w_2$}\label{fig:m>2w1}
\end{figure}

As in the previous two cases, if $n\leq2w_1+2w_2-m$ then the module
$\W{n,2w_1+2w_2-m}{+,+}$ no longer exists. However since we have
homomorphisms between the remaining modules, the block does not
decompose further. The same is true if $n\leq m$ or $n<2w_2-m$.

  We have now accounted for all modules $\W{n,m}{+,+}$ with $0\leq
  m\leq 2w_1+2w_2$ $(m\neq w_1+w_2)$, $\W{n,m}{-,+}$ with
  $0<m\leq2w_2$ $(m\neq w_2-w_1)$, and $\W{n,m}{+,-}$ with
  $0<m\leq2w_1$.

Suppose now $m=w_1+w_2$, again with $w_1\leq w_2$. In this case, we
show that there are at most two modules in the block and find a
homomorphism between them. Indeed, from equations
\eqref{eq:w1w2neg}--\eqref{eq:impossible} the only non-trivial
solution is $l=w_2-w_1$, $\ee_1\neq\eta_1$, $\ee_2=\eta_2$. But then
this satisfies the conditions of Theorem \ref{thm:w1hom}, and we have
a homomorphism
\[\W{N,w_1+w_2}{+,+}\longrightarrow \W{N,w_2-w_1}{-,+}.\]
This gives a block structure as in Figure \ref{fig:m=w1+w2}.

\begin{figure}[ht]
  \centering
  \tikz{
    \node (A) {$\W{N,w_1+w_2}{+,+}$};
    \node[below of = A, node distance=7em] (B) {$\W{N,w_2-w_1}{-,+}$};
    \draw[->] (A)--(B);}
  \caption{The block structure for $m=w_1+w_2$.}
  \label{fig:m=w1+w2}
\end{figure}

This means that now we have covered all modules $\W{n,m}{+,+}$ with
$0\leq m\leq 2w_1+2w_2$, $\W{n,m}{-,+}$ with $0<m\leq2w_2$, and
$\W{n,m}{+,-}$ with $0<m\leq2w_1$.

So it remains to consider the case $m>2w_1+2w_2$. Here we have both
$2w_1<m$ and $2w_2<m$, so both Theorems \ref{thm:w1hom} and
\ref{thm:w2hom} are satisfied and we have homomorphisms
\begin{align*}
  \W{N,m}{+,+}&\longrightarrow \W{N,m-2w_1}{-,+}\text{, and}\\
  \W{N,m}{+,+}&\longrightarrow \W{N,m-2w_2}{+,-}.
\end{align*}
But in this case we also have $m-2w_2>2w_1$ and $m-2w_1>2w_2$, so we
can again use Theorems \ref{thm:w1hom} and \ref{thm:w2hom} to obtain
homomorphisms
\begin{align*}
 \W{N,m-2w_2}{+,-}&\longrightarrow \W{N,m-2w_1-2w_2}{-,-}\text{, and}\\
 \W{N,m-2w_1}{-,+}&\longrightarrow \W{N,m-2w_1-2w_2}{-,-}.
\end{align*}
Note that this final case deals with all modules $\W{n,m}{+,+}$
with $m>2w_1+2w_2$, $\W{n,m}{-,+}$ with $m>2w_2$, $\W{m,n}{+,-}$
with $m>2w_1$ and $\W{n,m}{-,-}$ with $m>0$. In this case, the
block structure is shown in Figure \ref{fig:stable}.

\begin{figure}[ht]
  \centering
  \tikz{
    \node (A) {$\W{N,m}{+,+}$};
    \node[below left of = A,node distance=7em] (B) {$\W{N,m-2w_1}{-,+}$};
    \node[below right of = A, node distance=7em] (C) {$\W{N,m-2w_2}{+,-}$};
    \node[below left of = C, node distance=7em] (D) {$\W{N,m-2w_1-2w_2}{-,-}$};
    \draw[->] (A)--(B);
    \draw[->] (A)--(C);
    \draw[->] (B)--(D);
    \draw[->] (C)--(D);
    \draw[dotted] (A)--(D);
    \draw[dotted] (B)--(C);}
  \caption{The block structure for large $m$.}\label{fig:stable}
\end{figure}

The cases displayed in Figures \ref{fig:m<2w1}--\ref{fig:stable} above
exhaust the list of modules and we are able to give the final main
result.
\begin{thm}\label{thm:qnot1w1w2}
  Suppose $q$ is not a root of unity and $w_1$, $w_2\in\Z$. Let
  $\sigma_1=\mathrm{sgn}(w_1),\sigma_2=\mathrm{sgn}(w_2).$ Then for
  $n\geq2|w_1|+2|w_2|+\half(\sigma_1+\sigma_2)$, two cell modules
  $\W{n,m}{\ee_1,\ee_2}$ and $\W{n,l}{\eta_1,\eta_2}$ are in the
  same block if and only if
  $|m-\ee_1w_1-\ee_2w_2|=|l-\eta_1w_1-\eta_2w_2|$.
\end{thm}
\begin{proof}
  In the case $n\geq2|w_1|+2|w_2|+\half(\sigma_1+\sigma_2)$, none of
  the blocks described by Figures \ref{fig:m<2w1} and \ref{fig:m=2w1}
  break up into singleton blocks. Therefore in the globalised case
  with $w_1,w_2>0$, the Theorem follows. Notice that
  \begin{align*}
    |m-\ee_1w_1-\ee_2w_2|&=|(m+\ee_1)-(-\ee_1)(-w_1-1)-\ee_2w_2|\\
    &=|(m+\ee_2)-\ee_1w_1-(-\ee_2)(-w_2-1)|\\
    &=|(m+\ee_1+\ee_2)-(-\ee_1)(-w_1-1)-(-\ee_2)(-w_2-1)|,
  \end{align*}
and so the result holds after localising.
\end{proof}
In the case $n<2|w_1|+2|w_2|+\half(\sigma_1+\sigma_2)$, we do not have
a succinct characterisation of the blocks of $b'_n$. We therefore make
the following statement.
\begin{thm}\label{thm:qnot1w1w2loc}
Suppose $q$ is not a root of unity and $w_1,w_2\in\Z$. Let
$\sigma_1=\mathrm{sgn}(w_1),\sigma_2=\mathrm{sgn}(w_2).$ Then for
$n<2|w_1|+2|w_2|+\half(\sigma_1+\sigma_2)$, the blocks of $b'_n$ are
obtained by considering the blocks of the algebra $b'_N$ with $N\gg
n$, $w_1,w_2>0$ as in Figures \ref{fig:m<2w1}--\ref{fig:stable}, then
localising back to $b'_n$. The blocks are then given by removing any
annihilated modules and associated arrows from the appropriate figure
and taking the connected components of what remains.
\end{thm}
Turning now to the case when $q$ is a $2\ell$-th root of unity, we
will determine blocks in the large $N$ limit. Here, we will show that
any pair of cell modules satisfying equations
\eqref{eq:w1w2neg}--\eqref{eq:impossible} are in the same
block. Therefore in what follows, assume that
$\W{N,m}{\ee_1,\ee_2}$ and $\W{N,t}{\eta_1,\eta_2}$ are two cell
modules satisfying the current equation in question. This will
determine $\eta_1,\eta_2$ in terms of $\ee_1,\ee_2$, and also the
congruence class of $t$ modulo $2\ell$.

We deal with equations \eqref{eq:w1neg}, \eqref{eq:w2neg},
\eqref{eq:trivial} and \eqref{eq:w1w2pos} simultaneously. By applying
the arguments from Sections \ref{sec:w1orw2} and
\ref{sec:w1+w2orw1-w2}, we see that $\W{N,m}{\ee_1,\ee_2}$ and
$\W{N,t}{\eta_1,\eta_2}$ are in the same block.

Now consider equation \eqref{eq:w1w2neg}. In this case we choose
$r\in\Z$ such that $m-2(\ee_1w_1+r\ell)>0$ and consider the cell
module $\W{N,m-2(\ee_1w_1+r\ell)}{-\ee_1,\ee_2}$. We then see that
$\W{N,m-2(\ee_1w_1+r\ell)}{-\ee_1,\ee_2}$ and
$\W{N,m}{\ee_1,\ee_2}$ satisfy \eqref{eq:w1neg}, so are in the same
block by previous arguments. Moreover
$\W{N,m-2(\ee_1w_1+r\ell)}{-\ee_1,\ee_2}$ and
$\W{N,t}{\eta_1,\eta_2}$ satisfy \eqref{eq:w2neg} (due to our
assumptions on $\eta_1$,  $\eta_2$ and $t$), and are therefore in the
same block. By transitivity, we see that the original two modules are
in the same block.

If the two modules satisfy \eqref{eq:w1pos}, then then we will again
link these modules via a third. In particular we consider
$\W{N,2(\ee_1w_1+\ee_2w_2+r\ell)-m}{\ee_1,\ee_2}$, where $r\in\Z$
is chosen so that $2(\ee_1w_1+\ee_2w_2+r\ell)-m>0$. Then this module
and $\W{N,m}{\ee_1,\ee_2}$ satisfy \eqref{eq:w1w2pos} and are
therefore in the same block. Moreover,
$\W{N,2(\ee_1w_1+\ee_2w_2+r\ell)-m}{\ee_1,\ee_2}$ and
$\W{N,t}{\eta_1,\eta_2}$ satisfy \eqref{eq:w2neg} and are in the
same block. Therefore the original pair of modules are in the same
block.

If the two modules satisfy \eqref{eq:w2pos} then we use an argument
analogous to the previous paragraph.

Finally, if two modules satisfy \eqref{eq:impossible}, then we again
use a third module to show that the original two are in the same
block. In particular we choose $r\in\Z$ so that
$m-2(\ee_1w_1+r\ell)>0$, then we see that $\W{N,m}{\ee_1,\ee_2}$
and $\W{N,m-2(\ee_1w_1+r\ell)}{-\ee_1,\ee_2}$ are in the same block
(as they satisfy \eqref{eq:w1neg}), and also
$\W{N,m-2(\ee_1w_1+r\ell)}{-\ee_1,\ee_2}$ and
$\W{N,t}{\eta_1,\eta_2}$ are in the same block (as they satisfy
\eqref{eq:w1pos}).

We therefore arrive at the following theorem:

\begin{thm}\label{thm:qw1w2}
  Suppose $q$ is a $2\ell$-th root of unity and $w_1$,
  $w_2\in\Z$. Then for $N\gg\mathrm{max}\{m,t\}$, two cell modules
  $\W{N,m}{\ee_1,\ee_2}$ and $\W{N,t}{\eta_1,\eta_2}$ are in the
  same block if and only if
  $|m-\ee_1w_1-\ee_2w_2|\equiv|t-\eta_1w_1-\eta_2w_2|$ (mod $2\ell$).
\end{thm}

Figure \ref{fig:w1andw2} shows a plots of the cell modules when both
$w_1$ and $w_2$ are integral. The arrows indicate a homomorphism
between the corresponding modules, and the dashed square indicate that
these modules are in the same block. Note that away from the extremes
of each arm, there is a uniform pattern of concentric squares.

\begin{figure}
  \centering
  \begin{tikzpicture}[scale=0.2,>=latex,baseline=0] 
    \pgfmathsetmacro{\w}{3}
    \pgfmathsetmacro{\ww}{1}
    \draw (0,-16)--(0,16);
    \draw (-16,0)--(16,0);
    \foreach \m in {0,2,4,6,8,10,12,14,16}
    \foreach \e in {-1,1}
    \foreach \f in {-1,1}
    {       
      \draw (\e * \m - \e * \e * \w - \e * \f * \ww,\f * \m - \f * \e * \w - \f * \f * \ww)  node {$\bullet$};
    }
    \fill[white] (0 -\w + \ww, 0 + \w - \ww) circle (15pt);
    \fill[white] (0 -\w - \ww, 0 + \w - \ww) circle (15pt); 
    \draw (0 - \w - \ww, 0 - \w - \ww) -- ++(16,16) node[anchor=south west] {$+,+$};
    \draw (-2 - \w + \ww, 2 + \w - \ww) -- ++(-14,14) node[anchor=south east] {$-,+$};
    \draw (2 - \w + \ww, -2 + \w - \ww) -- ++(14,-14) node[anchor=north west] {$+,-$};
    \draw (-2 - \w - \ww, -2 - \w - \ww) -- ++(-14,-14) node[anchor=north east] {$-,-$};

    \draw[dashed,->] (0,0) .. controls (-3,0) and (0,3) .. (0,0);
    \draw[dashed,->] (6 - \w - \ww, 6 - \w - \ww) -- (6 - \w - \ww,-6 + \w + \ww);
    \draw[dashed,->] (6 - \w - \ww, 6 - \w - \ww) to[out=-90,in=0] (-6 + \w + \ww, -6 + \w + \ww);
    \draw[dashed,->] (8 - \w - \ww, 8 - \w - \ww) -- (8 - \w - \ww,-8 + \w + \ww);
    \draw[dashed,->] (8 - \w - \ww, 8 - \w - \ww) -- (-8 + \w + \ww,8 - \w - \ww);
    \draw[dashed,->] (8 - \w - \ww, 8 - \w - \ww) to[out=180,in=90] (-8 + \w + \ww, -8 + \w + \ww);

    \draw[dashed] (10 - \w - \ww, 10 - \w - \ww) rectangle (-10 + \w + \ww,-10 + \w + \ww);
    \draw[dashed] (12 - \w - \ww, 12 - \w - \ww) rectangle (-12 + \w + \ww,-12 + \w + \ww);
    \draw[dashed] (14 - \w - \ww, 14 - \w - \ww) rectangle (-14 + \w + \ww,-14 + \w + \ww);
    \draw[dashed] (16 - \w - \ww, 16 - \w - \ww) rectangle (-16 + \w + \ww,-16 + \w + \ww);

    \draw[dashed,->] (-10 - \w - \ww, -10 - \w - \ww) -- (10 + \w + \ww, -10 - \w - \ww);
    \draw[dashed,->] (-10 - \w - \ww, -10 - \w - \ww) -- (-10 - \w - \ww, 10 + \w + \ww);

    \draw[dashed,->] (-12 - \w - \ww, -12 - \w - \ww) -- (-12 - \w - \ww, 12 + \w + \ww);
    \draw[dashed,->] (-14 - \w - \ww, -14 - \w - \ww) -- (-14 - \w - \ww, 14 + \w + \ww);
  \end{tikzpicture}
  \caption{Graphical depiction of the cell modules of $b_{13}'$ with $w_1=3$ and $w_2=1$.}\label{fig:w1andw2}
\end{figure}

\section{Linkage via the module $\Wb$}\label{sec:wnb}

In the above we have been working  with cell modules for 
$\bnp := \bnx /I_n(0)$.
This precisely excludes  
the $2^n$-dimensional module $\Wb$. 
We will now deal with linkage via this module
in $\bnx$ and thus complete our investigation into the block structure. 
We begin with the following theorem:

\begin{thm}[{\cite[Theorem 5.17]{degiernichols}}]
  The Gram determinant $\Gamma^{n}_b$ of $\Wb$ 
(with respect to the  
path basis, as defined in (\ref{eq:defdetgamma}))
is given by:
  \begin{itemize}
  \item for $n$ even:
    \[\Gamma^{n}_b=\alpha_n\prod_{m=0}^{(n-2)/2}\left(\prod_{\ee_1,\ee_2,\ee_3=\pm1}[(1+2m+\ee_1w_1+\ee_2w_2+\ee_3\theta)/2]\right)^{\sum_{i=1}^{(n-2m)/2}\binom{n}{(n-2m-2i)/2}};\]
  \item for $n$ odd:
\begin{multline*}
\Gamma^{n}_b=
\alpha_n\left(\prod_{\ee_2,\ee_3=\pm1}[(w_1+\ee_2w_2+\ee_3\theta)/2]\right)^{2^{n-1}}\\
\times
\prod_{m=1}^{(n-1)/2}
\left(\prod_{\ee_1,\ee_2,\ee_3=\pm1}[(2m+\ee_1w_1+\ee_2w_2+\ee_3\theta)/2]
\right)^{\sum_{i=1}^{(n-2m+1)/2}\binom{n}{(n-2m-2i+1)/2}},
\end{multline*}
  \end{itemize}
where $\alpha_n$ is given in both cases by
\[\alpha_n=([w_1][w_2+1])^{-2\sum_{m=0}^{n-1}\left(\sum_{i=1}^{m+1}\binom{n}{m-i+1}\right)},
\]
up to factors that are units under our standing assumptions. 
\end{thm}
Assuming $\alpha_n\neq0$, we therefore see that the module $\Wb$
is irreducible unless $\theta$ is congruent to
$\pm(-m+\ee_1w_1+\ee_2w_2)$ modulo $2\ell$ for some integer $m$. Since
this module has 
label larger than all the other cell modules in the poset ordering,
being irreducible implies that it is alone in its block. (This remains
true even in the non-quasi-hereditary case as this module has
dimension larger than all the other cell modules.)
We will say that if  $\theta\equiv\pm(-m+\ee_1w_1+\ee_2w_2)$ (mod
$2\ell$) then $\theta$ is \emph{critical}.
If $\theta$ is not critical, then the module $\Wb$ is always a
singleton block and the algebra $b_n^x$ has the same blocks as $b_n'$
with the added singleton block. 
Therefore
we henceforth suppose that $\theta$ is critical.

Suppose that we have two non-isomorphic submodules
$\W{n,m}{\epsilon_1, \epsilon_2}$,
$\W{n,t}{\eta_1, \eta_2}$ of $\Wb$. These each correspond to a
singular factor of the gram determinant, $\Gamma^{n}_b$.
The condition for two of the factors in the gram determinant to be
equal to zero is the same condition for the central element $Z_n$ to
act by the same constant on the modules  $\W{n,m}{\ee_1, \ee_2}$ and
$\W{n,t}{\eta_1, \eta_2}$ with the added condition
that $[\left(m-\ee_1w_1 -\ee_2w_2 + \theta\right)/2] =0$. This is not so
surprising as if the modules 
$\W{n,m}{\epsilon_1, \epsilon_2}$,
$\W{n,t}{\eta_1, \eta_2}$  appear in the indecomposable module
$\Wb$ they must be in the same block.  This does mean, that a
careful rereading of the proofs of the blocks for the algebra
$b'_n$ will give the blocks for the full algebra $b_n^x$.

We have several cases. In each case, the blocks for $b_n^x$ are the
same as for the algebra $b'_n$, except for the block that contains
$\Wb$. In most cases, the module now joins two blocks from $b'_n$
together, except for the case where the modules that have the same
action by the central element, already form a block. 

To state our result we will only specify the block that can change,
i.e. the one that contains the module $\Wb$. \emph{The blocks that
don't contain $\Wb$ remain the same as for $b'_n$.}

\begin{thm}\label{thm:blocksbnx}
The block of the symplectic blob algebra, $b_n^x$, containing
$\Wb$ when $\theta$ has
a critical value, $\theta = -m + \epsilon_1w_1 + \epsilon_2w_2$,  is as follows:
\begin{enumerate}
\item[(i)]
If $q$ is not a root of unity and none of $w_1$, $w_2$, $w_1 \pm w_2$
are integral, then the only non-singleton block is the one containing
$\Wb$ and then $\Wb$ and $\W{n,m}{\epsilon_1, \epsilon_2}$ are in the same block.   
\item[(ii)]
If $q$ is a primitive  $2\ell$-th root of unity and none of $w_1$, $w_2$, $w_1 \pm w_2$
are integral, 
then the module $\Wb$ is in the same block as all
$\W{n,t}{\eta_1, \eta_2}$  with $\ee_1=\eta_1$, $\ee_2=\eta_2$ and
$m\equiv t$ (mod $2\ell$) and all 
$\W{n,t'}{\eta'_1, \eta'_2}$  with $\ee_1=-\eta'_1$, $\ee_2=-\eta'_2$ and
$m\equiv -t'$ (mod $2\ell$).
\item[(iii)]
If
$w_1+w_2 \in \Z$  and
none of $w_1$, $w_2$, $w_1 - w_2$
are integral, then 
the module $\Wb$ is in the same block as all
$\W{n,t}{\eta_1, \eta_2}$  with $\ee_1=\eta_1=\ee_2=\eta_2$ and
$|m-\ee_1w_1-\ee_2w_2| \equiv |t-\eta_1w_1-\eta_2w_2|$ (mod $2\ell$),
and all 
$\W{n,t'}{\eta'_1, \eta'_2}$  with $\ee_1=-\eta'_1=\ee_2=-\eta'_2$ and
$|m-\ee_1w_1-\ee_2w_2| \equiv |t-\eta'_1w_1-\eta'_2w_2|$ (mod $2\ell$),
where the $q$ is a primitive $2\ell$-th root of unity,
or the congruence is an equality if $q$ is not an $2\ell$-th root of
unity. 
\item[(iv)]
If 
$w_1-w_2 \in \Z$  and
none of $w_1$, $w_2$, $w_1 + w_2$
are integral, then 
the module $\Wb$ is in the same block as all
$\W{n,t}{\eta_1, \eta_2}$  with $\ee_1=\eta_1=-\ee_2=-\eta_2$ and
$|m-\ee_1w_1-\ee_2w_2| \equiv |t-\eta_1w_1-\eta_2w_2|$ (mod $2\ell$),
and all 
$\W{n,t'}{\eta'_1, \eta'_2}$  with $\ee_1=-\eta'_1=-\ee_2=\eta'_2$ and
$|m-\ee_1w_1-\ee_2w_2| \equiv |t-\eta'_1w_1-\eta'_2w_2|$ (mod $2\ell$),
where the $q$ is a primitive $2\ell$-th root of unity,
or the congruence is an equality if $q$ is not an $2\ell$-th root of
unity. 
\item[(v)]
If $w_1+w_2$ and
$w_1-w_2 \in \Z$  and
none of $w_1$, $w_2$,
are integral, then 
the module $\Wb$ is in the same block as all
$\W{n,t}{\eta_1, \eta_2}$  with 
$|m-\ee_1w_1-\ee_2w_2| \equiv |t-\eta_1w_1-\eta_2w_2|$ (mod $2\ell$),
where the $q$ is a primitive $2\ell$-th root of unity,
or the congruence is an equality if $q$ is not an $2\ell$-th root of
unity. 
\item[(vi)]
If
$w_1\in \Z$  or $w_2 \in \Z$  (but not both) and
none of $w_2$, $w_1 \pm w_2$
are integral, then 
the module $\Wb$ is in the same block as all
$\W{n,t}{\eta_1, \eta_2}$  with 
$|m-\ee_1w_1-\ee_2w_2| \equiv |t-\eta_1w_1-\eta_2w_2|$ (mod $2\ell$),
where the $q$ is a primitive $2\ell$-th root of unity,
or the congruence is an equality if $q$ is not an $2\ell$-th root of
unity. 
\item[(vii)]
If both $w_1$, $w_2 \in \Z$
then the module $\Wb$ is in the same block as all
$\W{n,t}{\eta_1, \eta_2}$  with 
$|m-\ee_1w_1-\ee_2w_2| \equiv |t-\eta_1w_1-\eta_2w_2|$ (mod $2\ell$),
where the $q$ is a primitive $2\ell$-th root of unity,
or the congruence is an equality if $q$ is not an $2\ell$-th root of
unity. (This is the only case where the module $\Wb$ does not
join two blocks from $b'_n$.)
\end{enumerate}
\end{thm}
\begin{proof}
(i). Clear as all $-m+\ee_1w_1+\ee_2w_2 -\theta$ are distinct in this
case and considering the Gram determinant of $\Wb$.

(ii). The module $\W{n,m}{\ee_1,\ee_2}$ embeds into $\Wb$, by
the assumption on $\theta$. Thus $\Wb$ is in the same block as
all $\W{n,t}{\eta_1,\eta_2}$ that are linked to
  $\W{n,m}{\ee_1,\ee_2}$. 
Pick $a \in \Z$ such that $0 \le 2a\ell-m \le n$. 
Then $[(2a\ell-m+\ee_1 w_1 + \ee_2w_2 - \theta)/2] =0$, so
$\W{n,2a\ell-m}{-\ee_1, -\ee_2}$ is a submodule of $\Wb$. Hence
$\Wb$ is in the same block as all 
all $\W{n,t)}{\eta'_1,\eta'_2}$ that are linked to
  $\W{n,2a\ell-m}{-\ee_1,-\ee_2}$. 

Thus using the proof of Theorem~\ref{thm:qrootofunity}, the block
containing $\Wb$ is the same as the solutions to
equations \eqref{eq:trivial} and \eqref{eq:impossible}.

(iii). As in (ii), $\Wb$ is in the same block as 
$\W{n,m}{\ee_1,\ee_2}$ and this gives the first condition.
Pick $a \in \Z$ such that $0 \le 2a\ell+m -2\ee_1(w_1+w_2) \le n$. 
(NB: $\ell$ is taken to be zero if $q$ is not a primitive $2\ell$-th root of
unity. Of course, then such an $a$ may not exist, but then there are
no further solutions to the equations that need to be considered.)
Consider $t=m-2\ee_1(w_1+w_2)$ and
the module $\W{n,t}{-\ee_1, -\ee_1}$. 
For this module 
$[(t+\ee_1 w_1 + \ee_2w_2 + \theta)/2] =0$, so $\W{n,t}{-\ee_1, -\ee_1}$
is a submodule of $\Wb$ and thus they are in the same block. 
This gives the second condition using Theorem \ref{thm:w1+w2}. 
These two conditions combined give all solutions to equations
\eqref{eq:w1w2neg}, \eqref{eq:trivial}, \eqref{eq:w1w2pos} and
\eqref{eq:impossible}, thus this is the whole block.

(iv). This is similar to (iii).

(v). This merges (iii) and (iv) and says that the block is determined
by the action of the central element. 

(vi). As in (ii), $\Wb$ is in the same block as 
$\W{n,m}{\ee_1,\ee_2}$. 
Pick $a \in \Z$ such that $0 \le 2a\ell-m +2\ee_1w_1 \le n$. 
(NB: $l$ is taken to be zero if $q$ is not a primitive $2\ell$-th root of
unity. Of course, then such an $a$ may not exist, but then there are
no further solutions to the equations that need to be considered.)
Consider $t= -m+2\ee_1w_1$ and
the module $\W{n,t}{\ee_1, -\ee_2}$. 
For this module 
$[(-t+\ee_1 w_1 - \ee_2w_2 + \theta)/2] =0$, so $\W{n,t}{\ee_1, -\ee_2}$
is a submodule of $\Wb$ and thus they are in the same block. 
This gives the condition as stated using Theorem \ref{thm:w1int},
which is the same as the condition for the central element to act by
zero.

(vii).
 Clear as the block is already determined by the action of the
central element.
\end{proof}

\appendix
\section{Reduction to Hom spaces}

It is well known that we may identify the blocks of a finite
dimensional algebra with the connected components of the $\Ext^1$
quiver between simple
modules.
Here we prove that finding these connected components is equivalent to
determining the connected components of the $\Hom$ quiver between
standards.
Thus in determining the blocks, it is only necessary to compute enough
homomorphisms between standard modules in order to find these
connected components.

(Here by the $\Hom$ quiver between standards, we mean take the quiver
whose vertices are labelled by standard modules, $\Delta(\mu)$ and
with the number of
arrows from $\lambda$ to $\mu$ equal to the dimension of
$\Hom(\Delta(\lambda), \Delta(\mu))$.)
 
\begin{prop} 
Let $A$ be a quasi-hereditary algebra with a simple preserving duality
and poset $(\Lambda, \le)$. For $\lambda \in \lambda$, let the
standard modules be
denoted by $\Delta(\lambda)$, costandards by $\nabla(\lambda)$, the
principal indecomposable modules by $P(\lambda)$, the irreducible head
of this module by
$L(\lambda)$ and the indecomposable injective hulls by $I(\lambda)$.

The blocks of $A$ may be identified with the connected components of
the $\Hom$ quiver between standards.
\end{prop}

\begin{proof} 
Now if $\Hom(\Delta(\lambda), \Delta(\mu))$ is non-zero then
$L(\lambda)$ must be a composition factor of $\Delta(\mu)$ as
$\Delta(\lambda)$ has simple head
$L(\lambda)$. Thus $L(\lambda)$ and and $L(\mu)$ are in the same
block. Thus the connected components of the $\Ext^1$ quiver on simples
are disjoint unions
of the connected components of the $\Hom$ quiver on standards.

We now prove the converse, that the connected components of the $\Hom$
quiver on standards are disjoint unions of the the $\Ext^1$ quiver on
simples.  Let
$\lambda$ and $\mu \in \Lambda$ with $\Ext^1(L(\lambda),
L(\mu))$. Without loss of generality we may assume that $\lambda \ge \mu$ as $A$ has a
simple preserving duality (and
hence $\Ext^1(L(\lambda),L(\mu)) = \Ext^1(L(\mu),L(\lambda))$).  Since
the extension of $L(\lambda)$ by $L(\mu)$ has simple socle and
$\lambda \ge \mu$,
this extension must be a quotient of $\Delta(\lambda)$ and in
particular $[\Delta(\lambda):L(\mu)] \ne 0$.  Thus $\Hom(P(\mu),
\Delta(\lambda))$ is
non-zero, as the dimension of $\Hom(P(\mu), \Delta(\lambda))$ is equal
to $[\Delta(\lambda):L(\mu)]$.

Now if $\Hom(\Delta(\mu), \Delta(\lambda))$ is non-zero then we are
done, so suppose that $\Hom(\Delta(\mu), \Delta(\lambda))=0$.  As
$\Hom(P(\mu),
\Delta(\lambda))$ is non-zero, there is a submodule of
$\Delta(\lambda)$, $M$, which is a quotient of $P(\mu)$ and hence has
simple head $L(\mu)$. As
$\Hom(\Delta(\mu), \Delta(\lambda))$ is zero, this quotient $M$ must
contain a composition factor, $L(\nu)$ say, for which $\nu \not \le
\mu$. Let $L(\nu)$
be a largest composition factor in $M$.

Now as $M$ is a submodule of $\Delta(\lambda)$ we must have that $\nu
< \lambda$. (It cannot be equal as then $M$ would be the whole of
$\Delta(\lambda)$
contradicting $\Hom(\Delta(\mu), \Delta(\lambda)) =0$.)  Now take the
largest submodule of $M$ with $L(\nu)$ as its simple head.  As $\nu$
was maximal, this
submodule has composition factors strictly less than $\nu$ and hence
is a quotient of $\Delta(\nu)$ and so $\Hom(\Delta(\nu),
\Delta(\lambda)) \ne 0$.

Also, as $M$ is not a quotient of $\Delta(\mu)$, it must contain a
proper submodule $N$ which is is a quotient of the kernel $Q$, of the
projection map from
$P(\mu)$ to $\Delta(\mu)$.  This proper submodule $N$ must contain the
$L(\nu)$ as $L(\nu)$ is not a composition factor of $\Delta(\mu)$. As
$\nu$ is
maximal, this $\nu$ must be the head of some $\Delta$ appearing in a
$\Delta$-filtration of $Q$. I.e. $\Delta(\nu)$ is a section of $Q$.
Thus $\Hom(P(\mu),
\nabla(\nu))$ is nonzero. Using the duality we then have
$\Hom(\Delta(\nu), I(\mu))\ne 0$, which implies that $L(\mu)$ is a
composition factor of
$\Delta(\mu)$ and thus that $\Hom(P(\mu), \Delta(\nu))$ is non-zero.

If $\Hom(\Delta(\mu), \Delta(\nu))$ is non-zero we may stop, otherwise
we repeat the argument until we have a chain of $\nu_i$'s with
$\lambda > \nu_1 >
\nu_2 >\dots > \nu_m $ and $\Hom(\Delta(\nu_i), \Delta(\nu_{i+1})) \ne
0$. Since $\Lambda$ is finite this chain must stop eventually with
$\Hom(\Delta(\nu_m), \Delta(\mu)) \ne 0$.  We may thus conclude that
$\lambda$ and $\mu$ are in the same connected component of the $\Hom$
quiver on
standards.
\end{proof}

\providecommand{\bysame}{\leavevmode\hbox to3em{\hrulefill}\thinspace}
\providecommand{\MR}{\relax\ifhmode\unskip\space\fi MR }
\providecommand{\MRhref}[2]{%
  \href{http://www.ams.org/mathscinet-getitem?mr=#1}{#2}
}
\providecommand{\href}[2]{#2}


\begin{thebibliography}{10}

\bibitem{ander1}
H.~H. Andersen, \emph{The strong linkage principle}, J. reine angew. Math.
  \textbf{315} (1980), 53--59.

\bibitem{Baxter81}
R.~J. Baxter, \emph{Exactly solved models in statistical mechanics}, Academic,
  1981.

\bibitem{benson}
D.~J. Benson, \emph{{R}epresentations and {C}ohomology {I}}, Cambridge Studies
  in Advanced Mathematics, no.~30, Cambridge University Press, 1995.

\bibitem{bowmancoxspeyer}
Christopher Bowman, Anton Cox, and Liron Speyer, \emph{A family of graded
  decomposition numbers for diagrammatic {C}herednik algebras}, Int. Math. Res.
  Not. IMRN (2017), no.~9, 2686--2734. \MR{3658213}

\bibitem{Brauer39}
R.~Brauer, \emph{On modular and p-adic representations of algebras},
  Proceedings of the National Academy of Sciences \textbf{25} (1939), no.~5,
  252--258.

\bibitem{brundankleshchev}
Jonathan Brundan and Alexander Kleshchev, \emph{Blocks of cyclotomic {H}ecke
  algebras and {K}hovanov-{L}auda algebras}, Invent. Math. \textbf{178} (2009),
  no.~3, 451--484. \MR{2551762}

\bibitem{blobcgm}
A.~G. Cox, J.~J. Graham, and P.~P. Martin, \emph{The blob algebra in positive
  characteristic}, J. Algebra \textbf{266} (2003), 584--635.

\bibitem{cmpx}
A.~G. Cox, P.~P. Martin, A.~E. Parker, and C.~Xi, \emph{Representation theory
  of towers of recollement: theory, notes, and examples}, J. Algebra
  \textbf{302} (2006), 340--360.

\bibitem{CoxDevisscherMartin0609}
A~G Cox, M~De Visscher, and P~P Martin, \emph{A geometric characterisation of
  the blocks of the {B}rauer algebra}, JLMS \textbf{80} (2009), 471--494,
  (math.RT/0612584).

\bibitem{daughram}
Z.~Daugherty and A.~Ram, \emph{Two boundary {H}ecke algebras and combinatorics
  of type {$C$}}, arXiv:1804.10296 (2018).

\bibitem{degiernichols}
J.~de~Gier and A.~Nichols, \emph{The two-boundary {T}emperley-{L}ieb algebra},
  J. Algebra \textbf{321} (2009), no.~4, 1132--1167.

\bibitem{Deodhar94}
V.~Deodhar, \emph{A brief survey of {K}azhdan--{L}usztig theory and related
  topics}, Proceedings of the Summer Research Institute on Algebraic Groups and
  their generalisations, {J}uly 6-26, 1991, AMS, 1994, pp.~105--124.

\bibitem{dlabring2}
V.~Dlab and C.~M. Ringel, \emph{The module theoretical approach to
  quasi-hereditary algebras}, Representations of algebras and related topics
  (Kyoto, 1990) (H.~Tachikawa and S.~Brenner, eds.), London Math. Soc. Lecture
  Note Ser., vol. 168, Cambridge Univ. Press, Cambridge, 1992, pp.~200--224.

\bibitem{brauerstroppelehrig}
M.~Ehrig and C.~Stroppel, \emph{{S}chur-{W}eyl duality for the {B}rauer algebra
  and the ortho-symplectic {L}ie superalgebra}, arXiv:1412.7853 (2014).

\bibitem{GrahamLehrer}
J.~Graham and G.~Lehrer, \emph{Cellular algebras}, Inventiones Math.
  \textbf{123} (1996), 1--34.

\bibitem{gensymp}
R.~M. Green, P.~P. Martin, and A.~E. Parker, \emph{Towers of recollement and
  bases for diagram algebras: planarity and beyond}, J. Algebra \textbf{316}
  (2007), 392--452.

\bibitem{mgp2I}
\bysame, \emph{A presentation for the symplectic blob algebra}, J. Algebra
  Appl. \textbf{11} (2012), no.~3, 1250060, 22.

\bibitem{mgp3}
\bysame, \emph{On quasi-heredity and cell module homomorphisms in the
  symplectic blob algebra}, arXiv:1707.06520 (2017).

\bibitem{Hartshorne}
R.~Hartshorne, \emph{Algebraic geometry}, Springer, 1977.

\bibitem{humph2}
J.~E. Humphreys, \emph{{R}eflection {G}roups and {C}oxeter {G}roups}, Cambridge
  Studies in Advanced Mathematics, no.~30, Cambridge University Press, 1990.

\bibitem{saleurcombinatorial}
J.~L. Jacobsen and H.~Saleur, \emph{Combinatorial aspects of boundary loop
  models}, J. Stat. Mech. (2008).

\bibitem{jantz}
J.~C. Jantzen, \emph{{R}epresentations of {A}lgebraic {G}roups}, Pure Appl.
  Math., vol. 131, Academic Press, San Diego, 1987.

\bibitem{kholaud}
M.~Khovanov and A.~D. Lauda, \emph{A diagrammatic approach to categorification
  of quantum groups. {I}}, Represent. Theory \textbf{13} (2009), 309--347.

\bibitem{LibedinskyPlaza}
N~Libedinsky and D~Plaza, \emph{Blob algebra approach to modular representation
  theory}, arxiv:1801.07200 (2018).

\bibitem{Lusztig83}
G.~Lusztig, \emph{Left cells in {W}eyl groups}, LNM 1024, Spinger (1983).

\bibitem{Lusztig99}
\bysame, \emph{Lectures on {H}ecke algebras with unequal parameters}, MIT
  Lecture notes (1999).

\bibitem{marbk}
P.~P. Martin, \emph{Potts models and related problems in statistical
  mechanics}, World Scientific, Singapore, 1991.

\bibitem{Martin92}
P~P Martin, \emph{On {S}chur-{W}eyl duality, {$A_n$} {H}ecke algebras and
  quantum {$sl(N)$}}, Int J Mod Phys A \textbf{7 suppl.1B} (1992), 645--674.

\bibitem{martsaleur}
P.~P. Martin and H.~Saleur, \emph{The blob algebra and the periodic
  {T}emperley-{L}ieb algebra}, Lett. Math. Phys. \textbf{30} (1994), no.~3,
  189--206.

\bibitem{marwood98}
P.~P. Martin and D.~Woodcock, \emph{The partition algebras and a new
  deformation of the {S}chur algebras}, J. Algebra \textbf{203} (1998),
  91--124.

\bibitem{marwood00}
\bysame, \emph{On the structure of the blob algebra}, J. Algebra \textbf{225}
  (2000), 957--988.

\bibitem{MartinWoodcock03}
Paul~P. Martin and David Woodcock, \emph{Generalized blob algebras and alcove
  geometry}, LMS Journal of Computation and Mathematics \textbf{6} (2003),
  249–296.

\bibitem{Plaza}
D.~Plaza and S.~Ryom-Hansen, \emph{Graded cellular bases for {T}emperley-{L}ieb
  algebras of type {A} and {B}}, J. Algebraic Combin. \textbf{40} (2014),
  no.~1, 137--177.

\bibitem{reeves2}
A.~Reeves, \emph{Tilting modules for the symplectic blob algebra},
  arXiv:1111.0146 (2012).

\bibitem{ringel2}
C.~M. Ringel, \emph{The category of modules with good filtrations over a
  quasi-hereditary algebra has almost split sequences}, Math. Z. \textbf{208}
  (1991), 209--225.

\bibitem{rouquier}
R.~Rouquier, \emph{{$q$}-{S}chur algebras and complex reflection groups}, Mosc.
  Math. J. \textbf{8} (2008), no.~1, 119--158, 184.

\bibitem{schiffler}
R.~Schiffler, \emph{Quiver representations}, Springer, 2014.

\bibitem{Soergel97b}
W.~Soergel, \emph{Charakterformeln f{\"u}r {K}ipp-{M}oduln {\"u}ber
  {K}ac-{M}oody-{A}lgebren}, Representation Theory \textbf{1} (1997), 115--132.

\bibitem{Soergel97a}
\bysame, \emph{{K}azhdan-{L}usztig polynomials and a combinatoric for tilting
  modules}, Representation Theory \textbf{1} (1997), 83--114.

\bibitem{TL}
H.~N.~V. Temperley and E.~H. Lieb, \emph{Relations between the ``percolation''
  and ``colouring'' problem and other graph-theoretical problems associated
  with regular planar lattices: some exact results for the ``percolation''
  problem}, Proc. Roy. Soc. London Ser. A \textbf{322} (1971), no.~1549,
  251--280.

\end{thebibliography}
\end{document}